\documentclass[12pt]{article}
\usepackage{amsfonts, amsmath, amssymb, latexsym, eucal, amscd} 

\usepackage[dvips]{pict2e}

\usepackage{amsthm}
\usepackage{amscd}

\usepackage[dvipdfmx]{graphics}
\usepackage{cite}

\newtheorem{definition}{Definition}[section]
\newtheorem{theorem}[definition]{Theorem}
\newtheorem{lemma}[definition]{Lemma}
\newtheorem{corollary}[definition]{Corollary}
\newtheorem{remark}[definition]{Remark}
\newtheorem{example}[definition]{Example}

\newtheorem{problem}[definition]{Problem}
\newtheorem{note}[definition]{Note}

\newtheorem{proposition}[definition]{Proposition}

\begin{document} 

\title{\bf 
The Lie algebra $\mathfrak{sl}_4(\mathbb C)$
  \\
and the hypercubes
}
\author{ 
William J. Martin \\
Paul Terwilliger
}
\date{}
\maketitle
\begin{abstract} We describe a relationship between the Lie algebra $\mathfrak{sl}_4(\mathbb C)$ and the hypercube graphs.
Consider the $\mathbb C$-algebra $P$ of polynomials in four commuting variables. We turn $P$ into an $\mathfrak{sl}_4(\mathbb C)$-module
on which each element of $\mathfrak{sl}_4(\mathbb C)$ acts as a derivation. 
Then $P$ becomes a direct sum of irreducible $\mathfrak{sl}_4(\mathbb C)$-modules
 $P = \sum_{N\in \mathbb N} P_N$, where $P_N$ is the $N$th homogeneous component
of $P$. 
For  $N\in \mathbb N$ we construct some additional $\mathfrak{sl}_4(\mathbb C)$-modules  ${\rm Fix}(G)$ and $T$.
For these modules the underlying vector space is described as follows.
Let $X$ denote the vertex set of the hypercube $H(N,2)$, and let $V$ denote the
$\mathbb C$-vector space with basis $X$. For the automorphism
group $G$  of  $H(N,2)$,  the action of $G$ on $X$ turns $V$ into a $G$-module. The vector space $V^{\otimes 3} = V \otimes V \otimes V$
becomes a $G$-module such that $g(u \otimes v \otimes w)= g(u) \otimes g(v) \otimes g(w)$  for $g\in G$ and $u,v,w \in V$. The subspace ${\rm Fix}(G)$
 of $V^{\otimes 3}$ consists of the vectors in $V^{\otimes 3}$ that are fixed by every element in $G$. 
 Pick $\varkappa \in X$. The corresponding subconstituent algebra $T$ of $H(N,2)$ is the subalgebra of ${\rm End}(V)$ generated by the adjacency
map $\sf A$ of $H(N,2)$ and the dual adjacency map ${\sf A}^*$ of $H(N,2)$ with respect to $\varkappa$.
In our main results, we turn ${\rm Fix}(G)$ and $T$ into $\mathfrak{sl}_4(\mathbb C)$-modules,
and display $\mathfrak{sl}_4(\mathbb C)$-module isomorphisms $P_N \to {\rm Fix}(G) \to T$.
 We describe the  $\mathfrak{sl}_4(\mathbb C)$-modules $P_N$, ${\rm Fix}(G)$, $T$ from multiple points of view.
 \bigskip

 \end{abstract}
 
\section{Introduction} 

The subconstituent algebra of a distance-regular graph was introduced in  \cite{tSub1,tSub2,tSub3}. 
This algebra is finite-dimensional, semisimple, and noncommutative in general. 
Its basic properties are described in  \cite{bbit,cerzo, int}.
The subconstituent algebra has been used to study
tridiagonal pairs \cite{bbit,itoTD, someAlg,augIto, altDRG},
spin models \cite{CW, C:thin, HSD,nomSM,nomSpin},
codes \cite{code2,code1, tanaka2},
projective geometries \cite{jhl3, koolen,liang1,liang2,seong},
quantum groups \cite{bcv, curtin6, qtetanddrg, itoTer3,boyd},
DAHA of rank one \cite{jhl1,jhl2, jhl4},
and some areas of mathematical physics \cite{bcv3,bcv2, bernard3, krystal}.
Further applications can be found in the survey \cite{dkt}.
\medskip

\noindent For someone seeking an introduction to the subconstituent algebra, the  hypercube graphs offer an attractive and accessible example.
For these graphs the subconstituent algebra is described in \cite{go}. It is apparent from \cite{go} that for the hypercube graphs,  the subconstituent algebra is closely
related to the Lie algebra $\mathfrak{sl}_2(\mathbb C)$. 
\medskip

\noindent  In the present paper, we will investigate the hypercube graphs using an approach that is different from the one in \cite{go}.
We will use the $S_3$-symmetric approach that was suggested in \cite{S3}. As it turns out, 
 in this approach the Lie algebra $\mathfrak{sl}_4(\mathbb C)$ plays the key role. Consequently, we will begin our investigation with
a detailed study of $\mathfrak{sl}_4(\mathbb C)$. For the rest of this section, we summarize our results.
 \medskip
 
\noindent  Our main topic is the Lie algebra $\mathfrak{sl}_4(\mathbb C)$, although the Lie algebra $\mathfrak{sl}_2(\mathbb C)$
will play a supporting role. Recall that $\mathfrak{sl}_2(\mathbb C)$ has a basis 
\begin{align*}
E= \begin{pmatrix} 0 & 1 \\ 0&0\end{pmatrix}, \qquad \quad
F= \begin{pmatrix} 0 & 0 \\ 1&0\end{pmatrix}, \qquad \quad
H= \begin{pmatrix} 1 & 0 \\ 0&-1\end{pmatrix}
\end{align*}
and Lie bracket
\begin{align*}
\lbrack H,E\rbrack = 2E, \qquad \quad \lbrack H,F\rbrack=-2F, \qquad \quad \lbrack E,F\rbrack=H.
\end{align*}
We now consider $\mathfrak{sl}_4(\mathbb C)$. We show that $\mathfrak{sl}_4(\mathbb C)$ has a generating set with six generators
\begin{align*}
&A_1 = \begin{pmatrix}              0&1&0&0\\
                                         1&0&0&0 \\
                                          0&0&0&1 \\
                                         0&0&1& 0
                                           \end{pmatrix}, \qquad \qquad A^*_1 = {\rm diag}(1,1,-1,-1), 
                                               \\
   &A_2 = \begin{pmatrix} 
                                           0&0&1&0\\
                                         0 &0&0&1 \\
                                         1 &0&0&0 \\
                                        0 &1&0& 0
  \end{pmatrix}, \qquad \qquad A^*_2 =  {\rm diag}(1,-1,1,-1),
                                               \\
&A_3 =  \begin{pmatrix}
                                   0&0&0&1\\
                                          0&0&1&0 \\
                                         0 &1&0&0 \\
                                         1&0&0& 0
                                          \end{pmatrix}, \qquad \qquad 
                                                                A^*_3 = {\rm diag}(1,-1,-1,1).                                                                                                   
\end{align*}
Using these generators, we give a presentation of $\mathfrak{sl}_4(\mathbb C)$ by generators and relations.
We show that  $A_1, A_2, A_3$ form a basis for a Cartan subalgebra $\mathbb H$ of  $\mathfrak{sl}_4(\mathbb C)$.
We show that  $A^*_1, A^*_2, A^*_3$ form a basis for a Cartan subalgebra $\mathbb H^*$ of  $\mathfrak{sl}_4(\mathbb C)$.
Let $i,j,k$ denote a permutation of $\lbrace 1,2,3\rbrace$. We show that $A_i, A^*_i$ commute.
We show that  $A_j, A^*_k$ generate a Lie subalgebra of $\mathfrak{sl}_4(\mathbb C)$ isomorphic
to $\mathfrak{sl}_2(\mathbb C)$. We show that $A_j, A_k, A^*_j, A^*_k$ generate a Lie subalgebra of $\mathfrak{sl}_4(\mathbb C)$ isomorphic to 
$\mathfrak{sl}_2(\mathbb C) \oplus \mathfrak{sl}_2(\mathbb C)$. We call this Lie subalgebra the $i$th Lie subalgebra of  $\mathfrak{sl}_4(\mathbb C)$ isomorphic to 
$\mathfrak{sl}_2(\mathbb C) \oplus \mathfrak{sl}_2(\mathbb C)$.
We display an automorphism $\tau$ of  $\mathfrak{sl}_4(\mathbb C)$ that swaps 
$A_1 \leftrightarrow A^*_1$, $A_2 \leftrightarrow A^*_2$, $A_3 \leftrightarrow A^*_3$.
\medskip

\noindent Next, we bring in a polynomial algebra $P$.
Let $x,y,z,w$ denote mutually commuting indeterminates, and consider the polynomial algebra $P=\mathbb C \lbrack x,y,z,w\rbrack$. 
We turn $P$ into an $\mathfrak{sl}_4(\mathbb C)$-module on which each element of $\mathfrak{sl}_4(\mathbb C)$ acts as a derivation.
We display two bases for $P$.  The first basis is
\begin{align} \label{eq:INTbasis1}
x^r y^s z^t w^u \qquad \qquad r,s,t,u \in \mathbb N.
\end{align}
 Define
\begin{align*}
& x^* = \frac{x+y+z+w}{2}, \qquad \qquad y^* = \frac{x+y-z-w}{2}, \\
& z^*= \frac{x-y+z-w}{2}, \qquad \qquad w^* = \frac{x-y-z+w}{2}.
\end{align*}
\noindent The second basis is
\begin{align} \label{eq:INTbasis2}
x^{*r} y^{*s} z^{*t} w^{*u} \qquad \qquad r,s,t,u \in \mathbb N.
\end{align}
We show how the $\mathfrak{sl}_4(\mathbb C)$-generators act on the bases \eqref{eq:INTbasis1} and \eqref{eq:INTbasis2}.
As we will see, the basis \eqref{eq:INTbasis1} diagonalizes  $\mathbb H^*$ and the basis  \eqref{eq:INTbasis2} diagonalizes $\mathbb H$. 
We display an automorphism $\sigma$ of the algebra $P$ that sends 
\begin{align*}
x \leftrightarrow x^*, \qquad \quad
y \leftrightarrow y^*, \qquad \quad
z \leftrightarrow z^*, \qquad \quad
w \leftrightarrow w^*.
\end{align*}
We show that for $\varphi \in \mathfrak{sl}_4(\mathbb C)$ the equation $\tau(\varphi) = \sigma \varphi \sigma^{-1}$ holds on $P$.
\medskip

\noindent Next, we consider the homogeneous components of $P$. For $N \in \mathbb N$ let $P_N$ denote the subspace of $P$ consisting of the homogeneous polynomials that have total degree $N$.
One basis for $P_N$  consists of the polynomials in \eqref{eq:INTbasis1} that have total degree $N$.
Another basis for $P_N$ consists of the polynomials in \eqref{eq:INTbasis2} that have total degree $N$.
By construction, the sum $P=\sum_{N \in \mathbb N} P_N$ is  direct. 
We show that each summand $P_N$ is an irreducible $\mathfrak{sl}_4(\mathbb C)$-submodule of $P$. We show that on the  $\mathfrak{sl}_4(\mathbb C)$-module  $P_N$,
each $\mathbb H$-weight space has dimension one and each  $\mathbb H^*$-weight space has dimension one.
By construction, each $\mathfrak{sl}_4(\mathbb C)$-generator is diagonalizable on $P_N$.
We show that the eigenvalues are $\lbrace N-2n\rbrace_{n=0}^N$, and for $0 \leq n \leq N$
the $(N-2n)$-eigenspace  has dimension $(n+1)(N-n+1)$. 
\medskip

\noindent Next, 
we display  a Hermitian form $\langle \,,\,\rangle$ on $P$ with respect to which the basis \eqref{eq:INTbasis1} is orthogonal
and the basis \eqref{eq:INTbasis2} is orthogonal.  For each basis we compute the square norm of the basis vectors. 
For $f,g \in P$ we show that $\langle \sigma f, \sigma g \rangle = \langle f,g \rangle$ and 
\begin{align*}
\langle A_i f, g \rangle = \langle f, A_i g\rangle,  \qquad \qquad \langle A^*_i f, g \rangle = \langle f, A^*_i g\rangle \qquad \qquad i \in \lbrace 1,2,3\rbrace.
\end{align*}
We  describe the inner product between each basis vector in \eqref{eq:INTbasis1} and each basis vector in \eqref{eq:INTbasis2}. 
We express these inner products in two ways; using a generating function and as hypergeometric sums. We display some orthogonality relations
and recurrence relations that involve the hypergeometric sums. For $N \in \mathbb N$ we display a polynomial $\mathcal P^\vee$ in six variables with
the following property: for
$r,s,t,u \in \mathbb N$ such that $r+s+t+u=N$,
\begin{align*}
\mathcal P^\vee(s,t,u; A_1, A_2, A_3) x^N &= x^r y^s z^t w^u; \\
\mathcal P^\vee(s,t,u; A^*_1, A^*_2, A^*_3) x^{*N} &= x^{*r} y^{*s} z^{*t} w^{*u}.
\end{align*}
Using this, we show that $P_N$ has a basis
\begin{align*}
A_1^s A_2^t A_3^u x^N \qquad \qquad s,t,u \in \mathbb N, \qquad s+t+u \leq N
\end{align*}
and a basis
\begin{align*}
A_1^{*s} A_2^{*t} A_3^{*u} x^{*N} \qquad \qquad s,t,u \in \mathbb N, \qquad s+t+u \leq N.   
\end{align*}

\noindent  Next, for $i \in \lbrace 1,2,3\rbrace$ we introduce a ``lowering map''  $L_i \in {\rm End}(P)$ and a ``raising map''  $R_i \in {\rm End}(P)$.
We define 
\begin{align*}
L_1 = D_x D_y -D_z D_w, \qquad \quad 
L_2 = D_x D_z -D_w D_y, \qquad \quad 
L_3 =  D_x D_w -D_y D_z, 
\end{align*}
where $D_x, D_y, D_z, D_w$ are the partial derivatives with respect to $x,y,z,w$ respectively.  Let $i \in \lbrace 1,2,3 \rbrace$. We show that $L_i(P_N)= P_{N-2}$ for  $N \in \mathbb N$,
where $P_{-1}=0$ and $P_{-2}=0$.
By construction $L_1, L_2, L_3$ mutually commute. We define
  \begin{align*}
 R_1 =  M_x M_y  - M_z M_w,      \qquad \quad R_2  = M_x M_z - M_w M_y,  \qquad \quad R_3 = M_x M_w - M_y M_z,
 \end{align*}
where $M_x, M_y, M_z, M_w$ denote multiplication by $x,y,z,w$ respectively. Let $i \in \lbrace 1,2,3 \rbrace$.  
By construction, $R_i(P_N) \subseteq P_{N+2}$ for  $N \in \mathbb N$.
Also by construction, $R_1, R_2, R_3$ are injective and  mutually commute.
Let  $i \in \lbrace 1,2,3\rbrace$. We give the action of $L_i$, $R_i$ on the basis \eqref{eq:INTbasis1} and the basis \eqref{eq:INTbasis2}.
We show that
\begin{align*}
\langle L_i f, g\rangle = \langle f, R_i g\rangle, \qquad \qquad \langle R_i f, g \rangle = \langle f, L_i g\rangle \qquad \qquad f,g \in P.
\end{align*}
We show that each of $L_i, R_i$ commutes with $\sigma$.
We show that each of $L_i, R_i$ commutes with each of $A_j, A_k, A^*_j, A^*_k$ where $j,k$ are the elements in $\lbrace 1,2,3 \rbrace \backslash \lbrace i \rbrace$.
We show that for $N \in \mathbb N$ the following sum is orthogonal and direct:
\begin{align*}
P_N = R_i (P_{N-2}) + {\rm Ker}(L_i) \cap P_N.
\end{align*}
Expanding on this, we obtain an orthogonal direct sum
\begin{align} \label{eq:IntroPN}
P_N=  \sum_{\ell=0}^{\lfloor N/2 \rfloor}  R_i^\ell \Bigl( {\rm Ker}(L_i ) \cap P_{N-2 \ell}\Bigr)
\end{align}
and an orthogonal direct sum
\begin{align} \label{eq:ODSintro}
     P = \sum_{N \in \mathbb N} \sum_{\ell \in \mathbb N} R_i^\ell \Bigl( {\rm Ker}(L_i ) \cap P_N\Bigr).
\end{align}
Next, we investigate the summands in  \eqref{eq:ODSintro}.
For each summand in  \eqref{eq:ODSintro} we display an orthogonal basis.
We show that for $N, \ell \in \mathbb N$ the corresponding summand in \eqref{eq:ODSintro} is an irreducible submodule for  
the $i$th Lie subalgebra of $\mathfrak{sl}_4(\mathbb C)$ isomorphic to  $\mathfrak{sl}_2(\mathbb C) \oplus \mathfrak{sl}_2(\mathbb C)$. 
This irreducible
submodule has dimension $(N+1)^2$. It is isomorphic to ${\mathbb V}_N \otimes {\mathbb V}_N$, where $\mathbb V_N$ denotes the irreducible $\mathfrak{sl}_2(\mathbb C)$-module with dimension $N+1$.
\medskip

\noindent Next, for $N \in \mathbb N$ we consider the sum
\begin{align} \label{eq:IntroELL}
 \sum_{\ell \in \mathbb N} R_i^\ell \Bigl( {\rm Ker}(L_i ) \cap P_N\Bigr).
\end{align}
As we investigate \eqref{eq:IntroELL}, it is convenient to define $\Omega \in {\rm End}(P)$ as follows. 
For $N \in \mathbb N$ the subspace $P_N$ is an eigenspace for $\Omega$ with eigenvalue $N$.
For $i \in \lbrace 1,2,3\rbrace$ we show that
\begin{align*}
\lbrack L_i, R_i \rbrack = \Omega +2I, \qquad \quad 
\lbrack \Omega, R_i \rbrack = 2 R_i, \qquad \quad
\lbrack \Omega, L_i \rbrack = -2 L_i.
\end{align*}
Consequently, $P$ becomes an $\mathfrak{sl}_2(\mathbb C)$-module on which $E, F, H$ act as follows:
\begin{align*} 
\begin{tabular}[t]{c|ccc}
{\rm element $\varphi$ }&$E$ &$F$&$H$ 
 \\
 \hline 
 {\rm action of $\varphi$ on $P$} & $-L_i$ &$R_i$ &$-\Omega-2I$ 
   \end{tabular}
\end{align*}
We show that the subspace \eqref{eq:IntroELL} is
an  $\mathfrak{sl}_2(\mathbb C)$-submodule of $P$. We express  \eqref{eq:IntroELL} as an orthogonal direct sum
of irreducible $\mathfrak{sl}_2(\mathbb C)$-modules. The irreducible $\mathfrak{sl}_2(\mathbb C)$-modules in the sum
are mutually isomorphic; they are all  highest-weight with highest weight $-N-2$.
\medskip

\noindent Next, we introduce some maps $C_1, C_2, C_3 \in {\rm End}(P)$.
Let $i,j,k$ denote a permutation of $\lbrace 1,2,3\rbrace$. We show that
\begin{align*}
\frac{(\Omega+2I)^2}{2} - L_i R_i - R_i L_i  &= \frac{4 A^2_j + 4 A^{*2}_k - (A_j A^*_k - A^*_k A_j)^2 }{8} \\
 &= \frac{4 A^{*2}_j + 4 A^{2}_k - (A^{*}_j A_k - A_k A^{*}_j)^2 }{8}.
\end{align*}
Call this common value $C_i$. We interpret $C_i$ using the concept of a Casimir operator.
 We compute the action of $C_i$ on the basis \eqref{eq:INTbasis1} and the basis \eqref{eq:INTbasis2}.
We show that $C_i$ commutes with each of $\Omega, L_i, R_i, A_j, A_k, A^*_j, A^*_k, \sigma$.
We show that
\begin{align*}
\langle C_i f, g \rangle = \langle f, C_i g \rangle \qquad \qquad f, g \in P.
\end{align*}
 We show that  \eqref{eq:IntroELL} is an eigenspace for the action of $C_i$ on $P$; the eigenvalue is $N(N+2)/2$.
 \medskip
 
 \noindent
Let $i \in \lbrace 1,2,3\rbrace$ and $N \in \mathbb N$. From our previous discussion we draw the following conclusions about \eqref{eq:IntroPN}.
For $0 \leq \ell \leq \lfloor N/2 \rfloor$ the $\ell$-summand in \eqref{eq:IntroPN} is an irreducible  submodule for the $i$th Lie subalgebra of
$\mathfrak{sl}_4(\mathbb C)$ isomorphic to 
$\mathfrak{sl}_2(\mathbb C) \oplus \mathfrak{sl}_2(\mathbb C)$.
This  $\ell$-summand  has dimension $(N-2\ell+1)^2$ and is isomorphic to ${\mathbb V}_{N-2\ell } \otimes {\mathbb V}_{N-2\ell}$.
This  $\ell$-summand  is an eigenspace for
the action of $C_i $ on $P_N$, with eigenvalue $(N-2\ell)(N-2\ell+2)/2$. 
\medskip

\noindent We will return to the decomposition \eqref{eq:IntroPN} later in this section.
\medskip

\noindent Next, we bring in the hypercube graphs. For the rest of this section, fix $N \in \mathbb N$. We  consider the $N$-cube $H(N,2)$. Let $X$ denote the vertex set of $H(N,2)$
and let $V$ denote the vector space with basis $X$. 
For $x \in X$ the set $\Gamma(x)$ consists of the vertices in $X$ that are adjacent to $x$.
The adjacency map  ${\sf A} \in {\rm End}(V)$ satisfies
\begin{align*}
{\sf A} x = \sum_{\xi \in \Gamma(x)} \xi, \qquad \qquad x \in X.       
\end{align*}
The vector space $V^{\otimes 3} = V \otimes V \otimes V$ has a basis
\begin{align*}
X^{\otimes 3} = \lbrace x \otimes y \otimes z\vert x,y,z \in X \rbrace.
\end{align*}
Let $G$ denote the automorphism group of  $H(N,2)$.
The action of $G$ on $X$ turns $V$ into a $G$-module. The vector space $V^{\otimes 3}$ becomes a $G$-module 
 such that for $g \in G$ and $u,v,w \in V$, 
\begin{align*}
g(u \otimes v \otimes w) = g(u) \otimes g(v) \otimes g(w).       
\end{align*}
Define the subspace
\begin{align*}
{\rm Fix}(G) = \lbrace v \in V^{\otimes 3} \vert g(v)=v \;\forall g \in G\rbrace.
\end{align*}
We  turn ${\rm Fix}(G)$ into an  $\mathfrak{sl}_4(\mathbb C)$-module as follows.
Define $A^{(1)}, A^{(2)}, A^{(3)} \in {\rm End}(V^{\otimes 3})$ such that
 for $x\otimes y\otimes z \in X^{\otimes 3}$,
\begin{align*}
A^{(1)} (x\otimes y \otimes z) &=  {\sf A}x \otimes y \otimes z, \\
A^{(2)} (x\otimes y \otimes z) &=  x \otimes {\sf A}y \otimes z, \\
A^{(3)} (x\otimes y \otimes z) &= x \otimes y \otimes {\sf A}z.
\end{align*}
 For notational convenience, define $\theta^*_i = N-2i$ for $0 \leq i \leq N$. For  $x,y \in X$ let $\partial(x,y)$ denote the path-length distance between $x,y$.
Define $A^{*(1)}, A^{*(2)}, A^{*(3)} \in {\rm End}(V^{\otimes 3})$ such that for $x\otimes y\otimes z \in X^{\otimes 3}$,
\begin{align*}
A^{*(1)} (x\otimes y \otimes z) &=  x \otimes y \otimes z\,\theta^*_{\partial(y,z)}, \\
A^{*(2)} (x\otimes y \otimes z) &=  x \otimes y\otimes z\, \theta^*_{\partial(z,x)}  , \\
A^{*(3)} (x\otimes y \otimes z) &=  x \otimes y \otimes z\, \theta^*_{\partial(x,y)}.
\end{align*}
We show that 
 ${\rm Fix}(G)$ is invariant under 
$A^{(i)}$ and $ A^{*(i)}$ for $ i \in \lbrace 1,2,3 \rbrace$.
We show that   ${\rm Fix}(G)$ is an  $\mathfrak{sl}_4(\mathbb C)$-module on which
\begin{align*}
A_i = A^{(i)}, \qquad \quad A^*_i = A^{*(i)} \qquad \qquad i \in \lbrace 1,2,3\rbrace.
\end{align*}
We endow $V^{\otimes 3}$ with a Hermitian form 
 $\langle\,,\,\rangle$ 
with respect to which the basis $X^{\otimes 3}$  is orthonormal. 
We display an  $\mathfrak{sl}_4(\mathbb C)$-module isomorphism $\ddag : P_N \to {\rm Fix}(G)$
such that
\begin{align*}
\langle f,g \rangle = \langle f^\ddag, g^\ddag \rangle \qquad \qquad f,g \in P_N.
\end{align*}

\noindent Our treatment of  $H(N,2)$ follows the $S_3$-symmetric approach discussed in \cite{S3}; see Note \ref{note:S3} 
below.
\medskip

\noindent Next, we consider the subconstituent algebras of $H(N,2)$.
Recall the adjacency map ${\sf A} \in {\rm End}(V)$ for $H(N,2)$.
For the rest of this section, fix $\varkappa \in X$. The corresponding dual adjacency map ${\sf A}^* = {\sf A}^*(\varkappa) \in {\rm End}(V)$ satisfies
\begin{align*}
{\sf A}^* x = \theta^*_{\partial(x, \varkappa)} x, \qquad \qquad x \in X. 
\end{align*}
By construction, the map ${\sf A}^*$ is diagonalizable with eigenvalues $\lbrace \theta^*_i \rbrace_{i=0}^N$.
The subconstituent algebra $T=T(\varkappa)$ is the subalgebra of ${\rm End}(V)$ generated by $\sf A, A^*$; see \cite[Definition~2.1]{go}.
 By \cite[Corollary~14.15]{go} we have ${\rm dim}\, T = \binom{N+3}{3}$.
\medskip

\noindent We mention some bases for the vector space $T$.
As we will see, the adjacency map $\sf A$ is diagonalizable with eigenvalues
$\theta_i = N-2i $ $ (0 \leq i \leq N)$.
For $0 \leq i \leq N$  let ${\sf E}_i \in {\rm End}(V)$ denote the primitive idempotent of $\sf A$ associated with $\theta_i$.
 For $x \in X$ and $0 \leq i \leq N$, define the set $\Gamma_i(x)=\lbrace y \in X \vert \partial(x,y)=i\rbrace$.
 For $0 \leq i \leq N$ define ${\sf A}_i \in {\rm End}(V)$ such that
\begin{align*}
{\sf A}_i x = \sum_{\xi \in \Gamma_i(x)} \xi, \qquad \qquad x \in X.
\end{align*}
For $0 \leq i \leq N$ define ${\sf E}^*_i = {\sf E}^*_i(\varkappa) \in {\rm End}(V)$ such that
\begin{align*}
{\sf E}^*_i x =  \begin{cases} x, & {\mbox{\rm if $\partial(x,\varkappa)=i$}}; \\
0, & {\mbox{\rm if $\partial(x,\varkappa)\not= i$}},
\end{cases} 
\qquad \qquad x \in X.
\end{align*}
By construction, ${\sf E}^*_i$ is the primitive idempotent of ${\sf A}^*$ for the eigenvalue $\theta^*_i$.
For $0 \leq i \leq N$  define ${\sf A}^*_i = {\sf A}^*_i(\varkappa) \in {\rm End}(V)$ such that
\begin{align*}
{\sf A}^*_i x = 2^N  \langle  {\sf E}_i \varkappa,  x \rangle x, \qquad \qquad x \in X.
\end{align*}
For notational convenience, let the set $\mathcal P''_N$ consist of the $3$-tuples of integers $(h,i,j)$ such that
\begin{align*}
&0 \leq h,i,j\leq N, \qquad  h+i+j \;\hbox{\rm is even}, \qquad  h+i+j \leq 2N, \\
&h \leq i+j, \quad \qquad \qquad i \leq j+h, \quad \qquad \qquad j \leq h+i.
\end{align*}
As we will see in Lemma \ref{lem:Tbasis}, the vector space $T$ has a basis
\begin{align*}
{\sf E}^*_i {\sf A}_h {\sf E}^*_j \qquad \qquad (h,i,j) \in \mathcal P''_N
\end{align*}
and a basis
\begin{align*}
{\sf E}_i {\sf A}^*_h {\sf E}_j \qquad \qquad (h,i,j) \in \mathcal P''_N.
\end{align*}
Define ${\mathcal A}^{(1)}, {\mathcal A}^{(2)}, {\mathcal A}^{(3)} \in {\rm End}(T)$ such that for $(h,i,j) \in \mathcal P''_N$,
\begin{align*}
{\mathcal A}^{(1)} \bigl( {\sf E}_i {\sf A}^*_h {\sf E}_j \bigr) &= \theta_h {\sf E}_i {\sf A}^*_h {\sf E}_j, \\
{\mathcal A}^{(2)} \bigl( {\sf E}_i {\sf A}^*_h {\sf E}_j \bigr) &= \theta_i {\sf E}_i {\sf A}^*_h {\sf E}_j, \\
{\mathcal A}^{(3)} \bigl( {\sf E}_i {\sf A}^*_h {\sf E}_j \bigr) &= \theta_j {\sf E}_i {\sf A}^*_h {\sf E}_j.
\end{align*}
Define ${\mathcal A}^{*(1)}, {\mathcal A}^{*(2)}, {\mathcal A}^{*(3)} \in {\rm End}(T)$ such that for $(h,i,j) \in \mathcal P''_N$,
\begin{align*}
{\mathcal A}^{*(1)} \bigl( {\sf E}^*_i {\sf A}_h {\sf E}^*_j \bigr) &= \theta^*_h {\sf E}^*_i {\sf A}_h {\sf E}^*_j, \\
{\mathcal A}^{*(2)} \bigl( {\sf E}^*_i {\sf A}_h {\sf E}^*_j \bigr) &= \theta^*_j {\sf E}^*_i {\sf A}_h {\sf E}^*_j, \\
{\mathcal A}^{*(3)} \bigl( {\sf E}^*_i {\sf A}_h {\sf E}^*_j \bigr) &= \theta^*_i {\sf E}^*_i {\sf A}_h {\sf E}^*_j.
\end{align*}
We show that $T$ is an $\mathfrak{sl}_4(\mathbb C)$-module on which
\begin{align*}
A_i = {\mathcal A}^{(i)}, \qquad \qquad A^*_i = {\mathcal A}^{*(i)} \qquad \qquad i \in \lbrace 1,2,3\rbrace.
\end{align*}
For $x,y \in X$ define a map $e_{x,y} \in {\rm End}(V)$ that sends $y \mapsto x$ and all other vertices to $0$. Note
that $\lbrace e_{x,y} \rbrace_{x,y \in X}$ form a basis for  ${\rm End}(V)$.
 We endow ${\rm End}(V)$ with a Hermitian form $\langle \,,\,\rangle $ with respect to which the basis $\lbrace e_{x,y} \rbrace_{x,y \in X}$ is
orthonormal.
We display an $\mathfrak{sl}_4(\mathbb C)$-module isomorphism $\vartheta: P_N \to T$
such that 
\begin{align*}
\langle f,g \rangle = \langle \vartheta(f), \vartheta(g) \rangle \qquad \qquad f,g \in P_N.
\end{align*}
 We show
that for $r,s,t,u \in \mathbb N$ such that $r+s+t+u=N$,
 the map $\vartheta$ sends
\begin{align*}
x^r y^s z^t w^u &\mapsto \frac{r!s!t!u!}{(N!)^{1/2}} {\sf E}^*_j {\sf A}_h {\sf E}^*_i,         \\
x^{*r} y^{*s} z^{*t} w^{*u} &\mapsto \frac{r!s!t!u!}{(N!)^{1/2}} {\sf E}_i {\sf A}^*_h {\sf E}_j, 
\end{align*}
where
\begin{align*}
h= t+u, \qquad \quad i = u+s, \qquad \quad j= s+t.
\end{align*}

\noindent We return our attention to the decomposition of $P_N$ given in \eqref{eq:IntroPN}. Referring to that decomposition, let us take $i=1$.
We show that the isomorphism $\vartheta:P_N\to T$ maps the given decomposition of $P_N$ to the Wedderburn decomposition of $T$
from  \cite[Theorems~14.10, 14.14]{go}.
\medskip

\noindent The main results of this paper are 
Theorems \ref{lem:HermComp},
\ref{prop:GenAction},
\ref{prop:TisoFix},
\ref{lem:HFFT},
\ref{lem:VarThetaAct},
\ref{thm:main1},
\ref{thm:HFFT},
\ref{thm:final}.

\section{Preliminaries}
We now begin our formal argument.
The following concepts and notation will be used throughout the paper. 
Recall the natural numbers $\mathbb N = \lbrace 0,1,2,\ldots \rbrace$.
Let $\mathbb C$ denote the field of complex numbers.
Every vector space, algebra, and tensor product that we discuss, is understood to be over $\mathbb C$.
Every algebra without the Lie prefix that we discuss, is understood to be associative  and have a multiplicative identity.
A subalgebra has the same multiplicative identity as the parent algebra.
For a nonzero vector space $V$, the algebra ${\rm End}(V)$ consists of the $\mathbb C$-linear maps from $V$ to $V$.
Let $I$ denote the multiplicative identity in ${\rm End}(V)$.
An element $B \in {\rm End}(V)$ is called {\it diagonalizable} whenever $V$ is spanned by the eigenspaces of $B$.
Assume that $B$ is diagonalizable, and let $\lbrace V_i \rbrace_{i=0}^N$ denote an ordering of the eigenspaces of $B$.
The sum $V=\sum_{i=0}^N V_i$ is direct. For $0 \leq i \leq N$ let $\theta_i$ denote the eigenvalue of $B$ for $V_i$.
For $0 \leq i \leq N$ define  $E_i\in {\rm End}(V)$ such that $(E_i-I)V_i=0$ and $E_i V_j =0$ if $j \not=i$ $(0 \leq j \leq N)$.
We call  $E_i$  the {\it primitive idempotent} of $B$ associated with $V_i$ (or $\theta_i$).
We have (i) $E_i E_j = \delta_{i,j} E_i $ $(0 \leq i,j\leq N)$; (ii)  $I = \sum_{i=0}^N E_i$; (iii) $B= \sum_{i=0}^N \theta_i E_i$;
(iv) $BE_i = \theta_i E_i = E_i B$ $(0 \leq i \leq N)$; (v) $V_i = E_iV$ $(0 \leq i \leq N)$.
Moreover
\begin{align} \label{eq:Eprod}
  E_i=\prod_{\stackrel{0 \leq j \leq N}{j \neq i}}
          \frac{B-\theta_jI}{\theta_i-\theta_j} \qquad \qquad (0 \leq i \leq N).
\end{align}
The maps $\lbrace E_i \rbrace_{i=0}^N$ form a basis for the subalgebra of ${\rm End}(V)$ generated by $B$.
For $\alpha \in \mathbb C$ let $\overline \alpha$ denote the complex-conjugate of $\alpha$.
For a positive real number $\alpha$, let $\alpha^{1/2}$ denote the positive square root of $\alpha$.
Let $\mathcal B$ denote an algebra.  By an {\it automorphism} of $\mathcal B$ we mean an algebra isomorphism  $\mathcal B \to \mathcal B$.
Let the algebra $\mathcal B^{\rm opp}$ consist of the vector space $\mathcal B$ and the following multiplication.
For $a,b \in \mathcal B$ the product $ab$ (in $\mathcal B^{\rm opp}$) is equal to $ba$ (in $\mathcal B$).
By an {\it antiautomorphism} of $\mathcal B$ we mean an algebra isomorphism $\mathcal B \to \mathcal B^{\rm opp}$.
We will be discussing Lie algebras. Background information about Lie algebras can be found in \cite{carter, humphreys}.

\section{The Lie algebras $\mathfrak{sl}_2(\mathbb C)$ and $\mathfrak{sl}_4(\mathbb C)$}

For an integer $n\geq 1$, the algebra ${\rm Mat}_n(\mathbb C)$ consists of the $n \times n$ matrices with all entries in $\mathbb C$.
The Lie algebra $\mathfrak{gl}_n(\mathbb C)$  consists of  the vector space ${\rm Mat}_n(\mathbb C)$ and Lie bracket
\begin{align*}
\lbrack \varphi,\phi \rbrack = \varphi \phi - \phi \varphi  \qquad \qquad \varphi, \phi \in {\rm Mat}_n(\mathbb C).
\end{align*}
The Lie algebra $\mathfrak{sl}_n(\mathbb C)$ is a Lie subalgebra of $\mathfrak{gl}_n(\mathbb C)$, consisting of the
matrices in  $\mathfrak{gl}_n(\mathbb C)$ that have trace 0. In this paper we will mainly consider
$\mathfrak{sl}_4(\mathbb C)$, although $\mathfrak{sl}_2(\mathbb C)$ will play a supporting role.

\begin{example}  \label{ex:sl2} \rm The Lie algebra
$\mathfrak{sl}_2(\mathbb C)$ has a basis
\begin{align*}
E = \begin{pmatrix} 0 & 1 \\ 0 & 0 \end{pmatrix},
\qquad \quad 
F = \begin{pmatrix} 0 & 0 \\ 1& 0 \end{pmatrix},
\qquad \quad 
H = \begin{pmatrix} 1 & 0 \\ 0 & -1 \end{pmatrix}
\end{align*}
and Lie bracket
\begin{align*}
 \lbrack H,E \rbrack=2E, \qquad \quad \lbrack H,F \rbrack=-2F, \qquad \quad \lbrack E,F \rbrack=H.
\end{align*}
\end{example}
\noindent We describe a presentation of $\mathfrak{sl}_2(\mathbb C)$ by generators and relations.
\begin{definition} \label{def:sl2P} \rm Define a Lie algebra $\bf L$ by generators $A, A^*$
and relations
\begin{align*} 
\lbrack A, \lbrack A, A^* \rbrack \rbrack=4 A^*, \qquad \qquad \lbrack A^*, \lbrack A^*, A\rbrack \rbrack=4A.
\end{align*}
\end{definition}

\begin{lemma} \label{lem:Lnat} There exists a Lie algebra isomorphism $\natural: {\bf L} \to \mathfrak{sl}_2(\mathbb C)$ that sends
\begin{align*} 
A \mapsto E+F, \qquad \qquad A^*\mapsto H.
\end{align*}
\end{lemma}
\begin{proof} We have
\begin{align*}
&\lbrack E+F, \lbrack E+F, H\rbrack \rbrack = \lbrack E+F, 2F-2E\rbrack = 4\lbrack E, F \rbrack=4H, \\
& \lbrack H, \lbrack H, E+F \rbrack \rbrack = \lbrack H, 2E-2F \rbrack = 4(E+F).
\end{align*}
Thus the matrices $E+F, H$ satisfy the relations in Definition \ref{def:sl2P}.
Consequently, there exists a Lie algebra homomorphism
 $\natural: {\bf L} \to \mathfrak{sl}_2(\mathbb C)$ that sends $A \mapsto E+F$ and $A^* \mapsto H$. Since $E,F,H$ form a basis for $\mathfrak{sl}_2(\mathbb C)$,
 there exists a $\mathbb C$-linear map $\mathfrak{sl}_2(\mathbb C) \to {\bf L}$ that sends
\begin{align*}
E \mapsto \frac{2A-\lbrack A, A^* \rbrack}{4}, \qquad \qquad F \mapsto \frac{ \lbrack A, A^* \rbrack+2A}{4}, \qquad \quad H \mapsto A^*.
\end{align*}
One checks that this map is the inverse of $\natural$. The map $\natural$ is a bijection, and hence a
Lie algebra isomorphism.
\end{proof}
\noindent For the rest of this paper, we identify $\bf L$ and $ \mathfrak{sl}_2(\mathbb C)$ via the isomorphism $\natural$ in Lemma \ref{lem:Lnat}.
\medskip
\begin{lemma} \label{lem:sl2Basis} The following is a basis for the vector space  $\mathfrak{sl}_2(\mathbb C)$:
\begin{align*}
A, \qquad \quad A^*, \qquad \quad \lbrack A, A^*\rbrack.
\end{align*}
\end{lemma}
\begin{proof} The matrices $A, A^*$ generate  $\mathfrak{sl}_2(\mathbb C)$.
\end{proof}

\noindent Next, we describe $\mathfrak{sl}_4(\mathbb C)$. We will give a basis for $\mathfrak{sl}_4(\mathbb C)$, and a presentation of
 $\mathfrak{sl}_4(\mathbb C)$ by generators and relations.
For $1\leq i,j\leq 4$ let 
$E_{i,j} \in {\rm Mat}_4(\mathbb C)$ have $(i,j)$-entry 1 and all other entries 0.
The following is a basis for $\mathfrak{sl}_4(\mathbb C)$:
\begin{align} \label{eq:sl4basis}
E_{i,j} \qquad (1 \leq i,j\leq 4, \quad i \not=j), \qquad E_{i,i}-E_{i+1,i+1} \qquad (1 \leq i \leq 3).
\end{align}
\noindent Serre gave a presentation of $\mathfrak{sl}_4(\mathbb C)$ by generators and relations,
see \cite[p.~99]{humphreys}. We will use a different presentation, that is better suited to our purpose.

\begin{definition} \label{def:LL} \rm We define a Lie algebra $\mathbb L$ by generators 
\begin{align*}
A_i, \quad A^*_i\qquad \quad i \in \lbrace 1,2,3\rbrace
\end{align*}
and the following relations.
\begin{enumerate}
\item[\rm (i)] For distinct $i,j \in \lbrace 1,2,3\rbrace$,
\begin{align*}
\lbrack A_i, A_j \rbrack=0, \qquad \qquad \lbrack A_i^*, A_j^* \rbrack=0.
\end{align*}
\item[\rm (ii)] For $i \in \lbrace 1,2,3\rbrace$,
\begin{align*}
\lbrack A_i, A^*_i \rbrack=0.
\end{align*}
\item[\rm (iii)]
For distinct $i,j \in \lbrace 1,2,3\rbrace$,
\begin{align*}
&\lbrack A_i, \lbrack A_i, A^*_j \rbrack \rbrack=4 A^*_j, \qquad \qquad
\lbrack A^*_j, \lbrack A^*_j, A_i \rbrack \rbrack=4 A_i.
\end{align*}
\item[\rm (iv)] For mutually distinct $h,i,j \in \lbrace 1,2,3\rbrace$,
\begin{align*}
\lbrack A_h, \lbrack A^*_i, A_j \rbrack \rbrack = \lbrack A^*_h, \lbrack A_i, A^*_j \rbrack \rbrack = 
\lbrack A_j, \lbrack A^*_i, A_h \rbrack \rbrack = \lbrack A^*_j, \lbrack A_i, A^*_h \rbrack \rbrack.
\end{align*}
\end{enumerate}
\end{definition}

\begin{lemma} \label{lem:sharp} There exists a Lie algebra isomorphism $\sharp: {\mathbb L} \to \mathfrak{sl}_4(\mathbb C)$ that sends
\begin{align*}
&A_1 \mapsto  \begin{pmatrix}              0&1&0&0\\
                                         1&0&0&0 \\
                                          0&0&0&1 \\
                                         0&0&1& 0
                                           \end{pmatrix}, \qquad \qquad A^*_1 \mapsto  {\rm diag}(1,1,-1,-1), 
                                               \\
   &A_2 \mapsto \begin{pmatrix} 
                                           0&0&1&0\\
                                         0 &0&0&1 \\
                                         1 &0&0&0 \\
                                        0 &1&0& 0
  \end{pmatrix}, \qquad \qquad A^*_2 \mapsto  {\rm diag}(1,-1,1,-1),
                                               \\
&A_3 \mapsto  \begin{pmatrix}
                                   0&0&0&1\\
                                          0&0&1&0 \\
                                         0 &1&0&0 \\
                                         1&0&0& 0
                                          \end{pmatrix}, \qquad \qquad 
                                                                A^*_3 \mapsto  {\rm diag}(1,-1,-1,1).                                                                                                   
\end{align*}
\end{lemma}
\begin{proof}  Consider the six matrices from the lemma statement.
One checks that these six matrices satisfy the relations in Definition \ref{def:LL}. 
Consequently, there exists a Lie algebra homomorphism  $\sharp: {\mathbb L} \to \mathfrak{sl}_4(\mathbb C)$ that sends each $\mathbb L$-generator to the
corresponding matrix.
In  \eqref{eq:sl4basis} we gave a basis for $\mathfrak{sl}_4(\mathbb C)$. There exists a $\mathbb C$-linear map $\mathfrak{sl}_4(\mathbb C) \to \mathbb L$
that acts on this basis as follows. The map sends
\begin{align*}
E_{1,2} &\mapsto \frac{4 A_1 + 2\lbrack A_2^{*}, A_1\rbrack+2 \lbrack A_3^{*}, A_1 \rbrack + \lbrack A_2^{*}, \lbrack A_3^{*} , A_1 \rbrack \rbrack    }{16}, \\
E_{2,1} &\mapsto \frac{4 A_1 - 2\lbrack A_2^{*}, A_1\rbrack-2 \lbrack A_3^{*}, A_1\rbrack + \lbrack A_2^{*}, \lbrack A_3^{*} , A_1 \rbrack \rbrack    }{16}, \\
E_{3,4} &\mapsto \frac{4 A_1 + 2\lbrack A_2^{*}, A_1\rbrack-2 \lbrack A_3^{*}, A_1\rbrack - \lbrack A_2^{*}, \lbrack A_3^{*} , A_1\rbrack \rbrack    }{16}, \\
E_{4,3} &\mapsto \frac{4 A_1 - 2\lbrack A_2^{*}, A_1\rbrack+2 \lbrack A_3^{*}, A_1 \rbrack - \lbrack A_2^{*}, \lbrack A_3^{*} , A_1 \rbrack \rbrack    }{16}
\end{align*}
and
\begin{align*}
E_{1,3} &\mapsto \frac{4 A_2 + 2\lbrack A_3^{*}, A_2\rbrack+2 \lbrack A_1^{*}, A_2 \rbrack + \lbrack A_3^{*}, \lbrack A_1^{*} , A_2 \rbrack \rbrack    }{16}, \\
E_{3,1} &\mapsto \frac{4 A_2 - 2\lbrack A_3^{*}, A_2\rbrack-2 \lbrack A_1^{*}, A_2 \rbrack + \lbrack A_3^{*}, \lbrack A_1^{*} , A_2 \rbrack \rbrack    }{16},\\
E_{4,2} &\mapsto \frac{4 A_2 + 2\lbrack A_3^{*}, A_2\rbrack-2 \lbrack A_1^{*}, A_2 \rbrack - \lbrack A_3^{*}, \lbrack A_1^{*} , A_2 \rbrack \rbrack    }{16}, \\
E_{2,4} &\mapsto \frac{4 A_2 - 2\lbrack A_3^{*}, A_2\rbrack+2 \lbrack A_1^{*}, A_2 \rbrack - \lbrack A_3^{*}, \lbrack A_1^{*} , A_2 \rbrack \rbrack    }{16}
\end{align*}
and
\begin{align*}
E_{1,4} &\mapsto \frac{4 A_3 + 2\lbrack A_1^{*}, A_3\rbrack+2 \lbrack A_2^{*}, A_3 \rbrack + \lbrack A_1^{*}, \lbrack A_2^{*} , A_3 \rbrack \rbrack    }{16}, \\
E_{4,1} &\mapsto \frac{4 A_3 - 2\lbrack A_1^{*}, A_3\rbrack-2 \lbrack A_2^{*}, A_3 \rbrack + \lbrack A_1^{*}, \lbrack A_2^{*} , A_3 \rbrack \rbrack    }{16},\\
E_{2,3} &\mapsto \frac{4 A_3 + 2\lbrack A_1^{*}, A_3\rbrack-2 \lbrack A_2^{*}, A_3 \rbrack - \lbrack A_1^{*}, \lbrack A_2^{*} , A_3 \rbrack \rbrack    }{16}, \\
E_{3,2} &\mapsto \frac{4 A_3 - 2\lbrack A_1^{*}, A_3\rbrack+2 \lbrack A_2^{*}, A_3 \rbrack - \lbrack A_1^{*}, \lbrack A_2^{*} , A_3 \rbrack \rbrack    }{16}
\end{align*}
and
\begin{align*}
E_{1,1}-E_{2,2} &\mapsto \frac{A_2^{*}+A_3^{*}}{2}, \\
E_{2,2}-E_{3,3} &\mapsto \frac{A_1^{*}-A_2^{*}}{2}, \\
E_{3,3}-E_{4,4} &\mapsto \frac{A_2^{*}-A_3^{*}}{2}.
\end{align*}
One checks that the above map $\mathfrak{sl}_4(\mathbb C) \to \mathbb L$ is the inverse of $\sharp$. The map $\sharp$ is a bijection, and hence
a Lie algebra isomorphism.
\end{proof}

\begin{note} \label{note:S3} \rm (See \cite[Definition~4.1]{S3}.) The universal enveloping algebra $U(\mathbb L)$ is a homomorphic image of the 
$S_3$-symmetric tridiagonal algebra $\mathbb T(2,0,0,4,4)$.
\end{note}

\noindent For the rest of this paper, we identify the Lie algebras $\mathbb L$ and $\mathfrak{sl}_4(\mathbb C)$ via the isomorphism $\sharp$ from Lemma
\ref{lem:sharp}.

\begin{lemma} \label{lem:sl4Basis} The following is a basis for the vector space $\mathfrak{sl}_4(\mathbb C)$:
\begin{align*}
&A_1, \qquad A_2, \qquad A_3, \qquad A^*_1, \qquad A^*_2, \qquad A^*_3,\\
&\lbrack A_1, A^*_2 \rbrack, \quad
\lbrack A_2, A^*_3 \rbrack, \quad
\lbrack A_3, A^*_1 \rbrack, \quad
\lbrack A^*_1, A_2 \rbrack, \quad
\lbrack A^*_2, A_3 \rbrack, \quad
\lbrack A^*_3, A_1 \rbrack,  \\
& \lbrack A^*_1, \lbrack A^*_2, A_3 \rbrack \rbrack, \qquad \quad
\lbrack A^*_2, \lbrack A^*_3, A_1 \rbrack \rbrack, \qquad \quad
\lbrack A^*_3, \lbrack A^*_1, A_2 \rbrack \rbrack.
\end{align*}
\end{lemma}
\begin{proof}  
There are 15 matrices in the list, and 
the dimension of $\mathfrak{sl}_4(\mathbb C)$ is 15. By linear algebra, it suffices to show that the listed matrices span $\mathfrak{sl}_4(\mathbb C)$. It is clear
from the proof of Lemma \ref{lem:sharp}, that the listed matrices span $\mathfrak{sl}_4(\mathbb C)$.
\end{proof}

\begin{lemma} \label{lem:LtoL} For distinct $i, j \in \lbrace 1,2,3\rbrace $ there exists a Lie algebra homomorphism $\mathfrak{sl}_2(\mathbb C) \to \mathfrak{sl}_4(\mathbb C)$
that sends
\begin{align*}
A \mapsto A_i, \qquad \qquad A^* \mapsto A^*_j.
\end{align*}
This homomorphism is injective.
\end{lemma}
\begin{proof}  To see that the homomorphism exists, compare the relations in Definition  \ref{def:sl2P} and Definition \ref{def:LL}(iii).
The homomorphism is injective by Lemmas \ref{lem:sl2Basis}, \ref{lem:sl4Basis}.
\end{proof}

\begin{corollary} \label{cor:6} For distinct $i,j \in \lbrace 1,2,3\rbrace$ the matrices $A_i, A^*_j$ generate a Lie subalgebra of $\mathfrak{sl}_4(\mathbb C)$
that is isomorphic to $\mathfrak{sl}_2(\mathbb C)$.
\end{corollary}
\begin{proof} By Lemma \ref{lem:LtoL}.
\end{proof}

\noindent  Our  presentation of $\mathfrak{sl}_4(\mathbb C)$  is described by the diagram below:
\vspace{.5cm}

\begin{center}
\begin{picture}(0,60)
\put(-100,0){\line(1,1.73){30}}
\put(-100,0){\line(1,-1.73){30}}
\put(-70,52){\line(1,0){60}}
\put(-70,-52){\line(1,0){60}}
\put(-10,-52){\line(1,1.73){30}}
\put(-10,52){\line(1,-1.73){30}}

\put(-102,-3){$\bullet$}
\put(-72.5,49){$\bullet$}
\put(-72.5,-55){$\bullet$}

\put(16.5,-3){$\bullet$}
\put(-13.5,49){$\bullet$}
\put(-13.5,-55){$\bullet$}

\put(-120,-3){$A_1$}
\put(-11,60){$A_2$}
\put(-11,-66){$A_3$}

\put(26,-3){$A^*_1$}
\put(-82,-66){$A^*_2$}
\put(-82,60){$A^*_3$}

\end{picture}
\end{center}
\vspace{2.5cm}

\hspace{2.2cm}
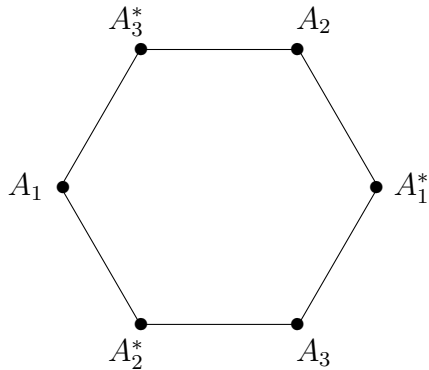
 Fig. 1.  Nonadjacent matrices commute. \\
${} \qquad $ Adjacent matrices generate a Lie subalgebra isomorphic to $\mathfrak{sl}_2(\mathbb C)$.
\vspace{.5cm}

\begin{lemma} \label{lem:sl2sl2} For distinct $j,k \in \lbrace 1,2,3\rbrace$ there exists a Lie algebra homomorphism
 $\mathfrak{sl}_2(\mathbb C) \oplus \mathfrak{sl}_2(\mathbb C) \to  \mathfrak{sl}_4(\mathbb C)$ that sends
 \begin{align*}
 (A,0) \mapsto A_j, \qquad \quad
 (A^*, 0) \mapsto A^*_k, \qquad \quad
 (0, A) \mapsto A_k, \qquad \quad
 (0, A^*) \mapsto A^*_j.
 \end{align*}
 This homomorphism is injective.
\end{lemma}
\begin{proof}  The homomorphism exists by Lemma  \ref{lem:LtoL} and since each of $A_j, A^*_k$ commutes with each of $A^*_j, A_k$.
The homomorphism is injective because the matrices
\begin{align*}
 A_j, \qquad A_k, \qquad A^*_j, \qquad A^*_k, \qquad \lbrack A_j, A^*_k \rbrack, \qquad  \lbrack A^*_j, A_k\rbrack
 \end{align*}
are linearly independent by Lemma \ref{lem:sl4Basis}.
\end{proof} 

\begin{corollary} \label{cor:LL} For distinct $j,k \in \lbrace 1,2,3\rbrace$ the matrices $A_j, A_k, A^*_j, A^*_k$ generate a Lie subalgebra of $\mathfrak{sl}_4(\mathbb C)$
that is isomorphic to $\mathfrak{sl}_2(\mathbb C) \oplus \mathfrak{sl}_2(\mathbb C)$.
\end{corollary}
\begin{proof} By Lemma \ref{lem:sl2sl2}.
\end{proof}

\begin{definition}\label{def:sl2sl2sub} \rm Let $i\in \lbrace 1,2,3\rbrace$. By the {\it $i$th Lie subalgebra of $\mathfrak{sl}_4(\mathbb C)$
isomorphic to $\mathfrak{sl}_2(\mathbb C) \oplus \mathfrak{sl}_2(\mathbb C)$}, we mean the Lie subalgebra of $\mathfrak{sl}_4(\mathbb C)$
generated by $A_j, A_k, A^*_j, A^*_k$ where $j,k$ are the  elements in $\lbrace 1,2,3\rbrace \backslash \lbrace i \rbrace$.
\end{definition}

\section{The automorphism $\tau$ of $\mathfrak{sl}_4(\mathbb C)$}

\noindent We continue to discuss the Lie algebra $\mathfrak{sl}_4(\mathbb C)$. In this section, we introduce
an automorphism $\tau$ of  $\mathfrak{sl}_4(\mathbb C)$ that swaps $A_i, A^*_i$ for $i \in \lbrace 1,2,3\rbrace$.

\begin{definition} \label{def:T} \rm Define $\Upsilon \in {\rm Mat}_4(\mathbb C)$ by
\begin{align*}
\Upsilon = \frac{1}{2}  \begin{pmatrix}  1&1&1&1\\
                                         1 &1&-1&-1 \\
                                         1 &-1&1&-1 \\
                                        1 &-1&-1& 1
                                          \end{pmatrix}.
  \end{align*}
  \end{definition}
 \noindent  Note that $\Upsilon^2 = I$.
  
  \begin{lemma} \label{lem:ATA} For $i \in \lbrace 1,2,3\rbrace$ we have
  \begin{align*}
A_i \Upsilon = \Upsilon A^*_i, \qquad \qquad A^*_i \Upsilon = \Upsilon A_i.
\end{align*}
\end{lemma}
\begin{proof} By matrix multiplication using Lemma \ref{lem:sharp} and the comment above Lemma \ref{lem:sl4Basis}.
\end{proof}

  \noindent By an {\it automorphism} of the Lie algebra $\mathfrak{sl}_4(\mathbb C)$, we mean a Lie algebra isomorphism
  $\mathfrak{sl}_4(\mathbb C) \to \mathfrak{sl}_4(\mathbb C)$.
  
  \begin{lemma} \label{lem:tau} There exists an automorphism $\tau$ of $\mathfrak{sl}_4(\mathbb C)$ that sends
  $\varphi  \mapsto \Upsilon \varphi \Upsilon^{-1}$ for all $\varphi \in \mathfrak{sl}_4(\mathbb C)$. Moreover, $\tau^2 = {\rm id}$.
  \end{lemma}
  \begin{proof} By $\Upsilon^2=I$ and linear algebra.
  \end{proof}

\begin{lemma} \label{lem:Tint}  The automorphism $\tau$ from Lemma \ref{lem:tau} swaps
  \begin{align*}
  A_i \leftrightarrow  A^*_i  \qquad \qquad i \in \lbrace 1,2,3\rbrace.
  \end{align*}
\end{lemma}
\begin{proof} By Lemmas \ref{lem:ATA}, \ref{lem:tau}.
\end{proof}

\noindent We will be discussing Cartan subalgebras of $\mathfrak{sl}_4(\mathbb C)$. The definition of a Cartan subalgebra can be found in \cite[p.~23]{carter}.

\begin{lemma} \label{lem:Cartan} The following {\rm (i)--(iv)} hold.
\begin{enumerate}
\item[\rm (i)] 
The elements $A_1, A_2, A_3$ form a basis for a Cartan subalgebra $\mathbb H$ of $\mathfrak{sl}_4(\mathbb C)$.
\item[\rm (ii)] 
The elements $A^*_1, A^*_2, A^*_3$ form a basis for a Cartan subalgebra $\mathbb H^*$ of $\mathfrak{sl}_4(\mathbb C)$.
\item[\rm (iii)] The automorphism $\tau$ swaps  $\mathbb H \leftrightarrow \mathbb H^*$.
\item[\rm (iv)]  The Lie algebra $\mathfrak{sl}_4(\mathbb C)$ is generated by $\mathbb H, \mathbb H^*$.
\end{enumerate}
\end{lemma}
\begin{proof} The matrices $A_1, A_2, A_3$ are linearly independent, so they form a basis for a subspace $\mathbb H$ of $\mathfrak{sl}_4(\mathbb C)$.
Similarly, the matrices $A^*_1, A^*_2, A^*_3$ form a basis for a subspace $\mathbb H^*$ of $\mathfrak{sl}_4(\mathbb C)$.
The subspaces $\mathbb H$, $\mathbb H^*$ satisfy (iii) by Lemma \ref{lem:Tint}.
The subspaces $\mathbb H$, $\mathbb H^*$ satisfy (iv) by construction.
The subspace $\mathbb H^*$ is a Cartan subalgebra of $\mathfrak{sl}_4(\mathbb C)$, because $\mathbb H^*$ consists of the diagonal matrices in $\mathfrak{sl}_4(\mathbb C)$.
The subspace $\mathbb H$ is a Cartan subalgebra of $\mathfrak{sl}_4(\mathbb C)$, because $\mathbb H$ is the image of a Cartan subalgebra $\mathbb H^*$
under an automorphism $\tau$ of $\mathfrak{sl}_4(\mathbb C)$.
\end{proof}

\begin{definition} \rm Let $W$ denote an $\mathfrak{sl}_4(\mathbb C)$-module, and consider the action of
$\mathbb H$ on $W$.
 A common eigenspace for this action is called an {\it $\mathbb H$-weight space} for $W$. An $\mathbb H^*$-weight space
for $W$ is similarly defined.
\end{definition}

\section{An $\mathfrak{sl}_4(\mathbb C)$-action on the polynomial algebra $\mathbb C \lbrack x,y,z,w\rbrack$}

Let $x,y,z,w$ denote mutually commuting indeterminates, and consider the algebra $\mathbb C\lbrack x,y,z,w\rbrack$ of polynomials in
$x,y,z,w$ that have all coefficients in $\mathbb C$. We abbreviate $P=\mathbb C \lbrack x,y,z,w\rbrack$.
The following is a basis for $P$:
\begin{align}
x^r y^s z^t w^u \qquad \qquad r,s,t,u \in \mathbb N. \label{eq:Pbasis}
\end{align}
For $N\in \mathbb N$ let $P_N$ denote the subspace of $P$ consisting of the homogeneous polynomials that have total degree $N$.
We call $P_N$ the {\it $N$th homogeneous component} of $P$. 
The sum $P =\sum_{N \in \mathbb N} P_N$ is direct. For notational convenience, define $P_{-1}=0$ and $P_{-2}=0$.

\begin{definition} \label{def:profiles} \rm For $N \in \mathbb N$ let the set $\mathcal P_N$ consist of the 4-tuples of natural numbers $(r,s,t,u)$ such that
$r+s+t+u=N$. An element of $\mathcal P_N$ is called a {\it profile of degree $N$}.
\end{definition}
\noindent Note that 
\begin{align}
\vert \mathcal P_N \vert = \binom{N+3}{3}. \label{eq:PNsize}
\end{align}

\begin{lemma} \label{lem:PDbasic} For $N \in \mathbb N$
the following is a basis for $P_N$:
\begin{align}
x^r y^s z^t w^u \qquad \qquad (r,s,t,u)  \in \mathcal P_N. \label{eq:PDbasis}
\end{align}
Moreover, $P_N$ has dimension $\binom{N+3}{3}$.
\end{lemma}
\begin{proof} Routine.
\end{proof}

\noindent  Our next goal is to turn $P$ into an $ \mathfrak{sl}_4(\mathbb C)$-module.
To this end, we first consider $P_1$. Note that $x,y,z,w$ is a basis for $P_1$.

\begin{lemma} \label{lem:sl4onP1} 
The vector space $P_1$ becomes 
an $\mathfrak{sl}_4(\mathbb C)$-module that satisfies {\rm (i)--(vi)} below:
\begin{enumerate}
\item[\rm (i)] 
$A_1$ swaps $x \leftrightarrow y$ and $z \leftrightarrow w$;
\item[\rm (ii)] 
$A_2$ swaps $x \leftrightarrow z$ and $y \leftrightarrow w$;
\item[\rm (iii)] 
$A_3$ swaps $x \leftrightarrow w$ and $y \leftrightarrow z$;
\item[\rm (iv)] $A^*_1$ sends
\begin{align*}
x \mapsto x, \qquad y \mapsto y, \qquad z \mapsto -z, \qquad w \mapsto -w;
\end{align*}
\item[\rm (v)] $A^*_2$ sends
\begin{align*}
x \mapsto x, \qquad y \mapsto -y, \qquad z \mapsto z, \qquad w \mapsto -w;
\end{align*}
\item[\rm (vi)] $A^*_3$ sends
\begin{align*}
x \mapsto x, \qquad y \mapsto -y, \qquad z \mapsto -z, \qquad w \mapsto w.
\end{align*}
\end{enumerate}
\end{lemma}
\begin{proof}
By Lemma \ref{lem:sharp} and the comment above Lemma  \ref{lem:sl4Basis}.
\end{proof}

\noindent We will be discussing derivations of $P$.  The Lie algebra $\mathfrak{gl}(P)$ consists of the vector space ${\rm End}(P)$ and Lie bracket
\begin{align*}
\lbrack \varphi, \phi \rbrack = \varphi \phi - \phi \varphi  \qquad \qquad \varphi, \phi \in {\rm End}(P).
\end{align*}
A {\it derivation of $P$} is an element $\mathcal D \in {\rm End}(P)$ such that 
\begin{align}
\mathcal D (fg) = \mathcal D(f) g + f \mathcal D(g)\qquad \qquad f, g \in P.          \label{eq:fg}
\end{align}
Let the set ${\rm Der}(P)$ consist of the derivations of $P$. One checks that
 ${\rm Der}(P)$ is a Lie subalgebra of $\mathfrak{gl}(P)$.
 \begin{lemma} \label{lem:oneD} For  $\mathcal D \in {\rm Der}(P)$ we have $\mathcal D(1)=0$.
 \end{lemma}
 \begin{proof} We have
 \begin{align*}
 \mathcal D(1) = \mathcal D(1^2) = \mathcal D(1) 1 + 1 \mathcal D(1) = 2 \mathcal D(1).
 \end{align*}
 \noindent Therefore $\mathcal D(1)=0$.
 \end{proof}
 
 \begin{lemma} \label{lem:fn} For $\mathcal D \in {\rm Der}(P)$ and $f \in P$ and $n \in \mathbb N$,
 \begin{align*}
 \mathcal D(f^n) = n f^{n-1} \mathcal D(f).
 \end{align*}
 \end{lemma}
 \begin{proof} Use \eqref{eq:fg} and Lemma  \ref{lem:oneD} and induction on $n$.
 \end{proof}

\begin{lemma} \label{lem:derAction} For  $\mathcal D \in {\rm End}(P)$ the following are equivalent:
\begin{enumerate}
\item[\rm (i)] $\mathcal D \in {\rm Der}(P)$;
\item[\rm (ii)] for $r,s,t,u \in \mathbb N$,
\begin{align*}
\mathcal D(x^r y^s z^t w^u) &= r x^{r-1} y^s z^t w^u \mathcal D(x) +
                                                s x^r y^{s-1} z^t w^u \mathcal D(y) \\
                                              &\quad  + t x^r y^s z^{t-1} w^u \mathcal D(z)+
                                                u x^r y^s z^t w^{u-1} \mathcal D(w);
 \end{align*}
 \item[\rm (iii)] for all polynomials $f,g$ in the $P$-basis \eqref{eq:Pbasis},
 \begin{align*}
\mathcal D (fg) = \mathcal D(f) g + f \mathcal D(g).    
\end{align*}
 \end{enumerate}
 \end{lemma}
 \begin{proof} ${\rm (i)} \Rightarrow {\rm (ii)}$   By  \eqref{eq:fg} and Lemma \ref{lem:fn},
 \begin{align*}
\mathcal D(x^r y^s z^t w^u) &=    \mathcal D(x^r) y^s z^t w^u  +
                                                 x^r \mathcal D(y^s) z^t w^u  
                                               +  x^r y^s\mathcal D(z^t) w^u 
                                              +
                                                 x^r y^s z^t  \mathcal D(w^u) \\
                                            &=   r x^{r-1} y^s z^t w^u \mathcal D(x) +
                                                s x^r y^{s-1} z^t w^u \mathcal D(y) 
                                               + t x^r y^s z^{t-1} w^u \mathcal D(z) \\
                                               &\quad +
                                                u x^r y^s z^t w^{u-1} \mathcal D(w).
 \end{align*}
 ${\rm (ii)} \Rightarrow {\rm (iii)}$  Routine. \\
  ${\rm (iii)} \Rightarrow {\rm (i)}$  The map $\mathcal D$ satisfies condition \eqref{eq:fg} since $\mathcal D$ is $\mathbb C$-linear.
 \end{proof}
 
 \begin{lemma} \label{lem:zeroDer} For $\mathcal D \in {\rm Der}(P)$ the following are equivalent:
 \begin{enumerate}
 \item[\rm (i)] $\mathcal D=0$;
 \item[\rm (ii)] $\mathcal D=0$ on $P_1$.
 \end{enumerate}
 \end{lemma}
 \begin{proof} ${\rm (i)} \Rightarrow {\rm (ii)}$ Clear.
 \\
 ${\rm (ii)} \Rightarrow {\rm (i)}$ Each of $\mathcal D(x),\mathcal D(y),\mathcal D(z),\mathcal D(w)$ is zero, so $\mathcal D=0$ by Lemma \ref{lem:derAction}(i),(ii).
 \end{proof}
 
 \begin{lemma}\label{lem:derExt} For a $\mathbb C$-linear map $\mathcal D_1 : P_1 \to P$ there exists a unique
 $\mathcal D \in {\rm Der}(P)$ such that the restriction $\mathcal D \vert_{P_1} = \mathcal D_1$.
 \end{lemma}
 \begin{proof} Concerning existence, by linear algebra there exists $\mathcal D \in {\rm End}(P)$ that acts on the $P$-basis
 \eqref{eq:Pbasis} as follows.
 For $r,s,t,u \in \mathbb N$,
 \begin{align*}
\mathcal D(x^r y^s z^t w^u) &= r x^{r-1} y^s z^t w^u \mathcal D_1(x) +
                                                s x^r y^{s-1} z^t w^u \mathcal D_1(y) \\
                                              &\quad  + t x^r y^s z^{t-1} w^u \mathcal D_1(z)+
                                                u x^r y^s z^t w^{u-1} \mathcal D_1(w).
 \end{align*}
 Note that
 \begin{align*}
 \mathcal D(x) = \mathcal D_1(x), \qquad 
  \mathcal D(y) = \mathcal D_1(y), \qquad 
   \mathcal D(z) = \mathcal D_1(z), \qquad 
    \mathcal D(w) = \mathcal D_1(w).
 \end{align*}
Therefore $\mathcal D \vert_{P_1} = \mathcal D_1$. 
 By these comments and Lemma  \ref{lem:derAction}(i),(ii) we have $\mathcal D \in {\rm Der}(P)$. We have shown that there exists 
 $\mathcal D \in {\rm Der}(P)$ such that $\mathcal D \vert_{P_1} = \mathcal D_1$. The uniqueness of $\mathcal D$ follows
 from Lemma \ref{lem:zeroDer}.
\end{proof}
 
 \begin{proposition} \label{prop:sl4EXT} There exists a unique Lie algebra homomorphism ${\rm der}: \mathfrak{sl}_4(\mathbb C) \to {\rm Der}(P)$
 such that for all $\varphi \in \mathfrak{sl}_4(\mathbb C)$ the following coincide:
 \begin{enumerate}
 \item[\rm (i)] the acton of ${\rm der}(\varphi)$ on $P_1$;
 \item[\rm (ii)] the action of $\varphi$ on $P_1$ from Lemma \ref{lem:sl4onP1}.
 \end{enumerate}
 The map $\rm der$ is injective.
 \end{proposition}
 \begin{proof} For $\varphi \in \mathfrak{sl}_4(\mathbb C)$ we define ${\rm der}(\varphi)$ as follows.
 Consider the action of $\varphi$ on $P_1$ from Lemma \ref{lem:sl4onP1}. By Lemma \ref{lem:derExt},
 there exists a unique $\mathcal D \in {\rm Der}(P)$ such that $ \mathcal D= \varphi$ on $P_1$. Define ${\rm der}(\varphi)=\mathcal D$.
 So far, we have a map ${\rm der}: \mathfrak{sl}_4(\mathbb C) \to {\rm Der}(P)$. The map der is $\mathbb C$-linear by construction.
 We show that der is a Lie algebra homomorphism. For $\varphi, \phi \in \mathfrak{sl}_4(\mathbb C)$ we show that
 \begin{align}
 {\rm der}\,\lbrack \varphi, \phi \rbrack = \lbrack {\rm der}\, \varphi, {\rm der} \,\phi \rbrack.  \label{eq:RScom}
 \end{align}
 Each side of \eqref{eq:RScom} is 
 a derivation of $P$. For each side of  \eqref{eq:RScom} the restriction to $P_1$ coincides with the action of $\lbrack \varphi, \phi \rbrack$ on $P_1$.
 The two sides of  \eqref{eq:RScom} are equal in view of Lemma \ref{lem:derExt}. We have shown that the Lie algebra homomorphism
 der exists. This map is unique by Lemma \ref{lem:derExt}, and injective because $\mathfrak{sl}_4(\mathbb C)$ acts faithfully on $P_1$.
 \end{proof}
 
 \noindent By Proposition \ref{prop:sl4EXT}, the vector space $P$ becomes an $\mathfrak{sl}_4(\mathbb C)$-module on which the elements of $\mathfrak{sl}_4(\mathbb C)$
 act as derivations.

\begin{proposition} \label{lem:ActP} The $\mathfrak{sl}_4(\mathbb C)$-generators
 $A_1, A_2, A_3$ and $ A^*_1, A^*_2, A^*_3$ act on the $P$-basis \eqref{eq:Pbasis} as follows.
For $r,s,t,u \in \mathbb N$,
\begin{enumerate}
\item[\rm (i)] the vector
\begin{align*}
 A_1( x^ry^s z^t w^u)
 \end{align*}
 is a linear combination with the following terms and coefficients:
\begin{align*} 
\begin{tabular}[t]{c|c}
{\rm Term }& {\rm Coefficient} 
 \\
 \hline
   $ x^{r-1} y^{s+1} z^{t} w^{u} $& $r$   \\
 $ x^{r+1} y^{s-1} z^{t} w^{u} $   & $s$\\
  $ x^{r} y^{s} z^{t-1} w^{u+1}  $  & $t$ \\
  $ x^{r} y^{s} z^{t+1} w^{u-1} $& $u$
   \end{tabular}
\end{align*}
\item[\rm (ii)] 
the vector
\begin{align*}
A_2(x^r y^s z^t w^u)     
\end{align*}
 is a linear combination with the following terms and coefficients:
\begin{align*} 
\begin{tabular}[t]{c|c}
{\rm Term }& {\rm Coefficient} 
 \\
 \hline
   $ x^{r-1} y^{s} z^{t+1} w^{u} $& $ r$   \\
 $ x^{r} y^{s-1} z^{t} w^{u+1} $   & $s$\\
  $ x^{r+1} y^{s} z^{t-1} w^{u} $  & $t$ \\
  $ x^{r} y^{s+1} z^{t} w^{u-1}  $& $u $
   \end{tabular}
\end{align*}
\item[\rm (iii)] 
the vector
\begin{align*}
A_3( x^{r} y^{s} z^{t} w^{u} )
\end{align*}
 is a linear combination with the following terms and coefficients:
\begin{align*} 
\begin{tabular}[t]{c|c}
{\rm Term }& {\rm Coefficient} 
 \\
 \hline
    $ x^{r-1} y^{s} z^{t} w^{u+1} $& $r $   \\
 $x^{r} y^{s-1} z^{t+1} w^{u} $   & $s $\\
    $ x^{r} y^{s+1} z^{t-1} w^{u} $  & $t$ \\
  $x^{r+1} y^{s} z^{t} w^{u-1} $& $u $
   \end{tabular}
\end{align*}
\item[\rm (iv)] $A^*_1(x^{r} y^{s} z^{t} w^{u} ) = (r+s-t-u) x^{r} y^{s} z^{t} w^{u} $;
\item[\rm (v)] $A^*_2(x^{r} y^{s} z^{t} w^{u} ) = (r-s+t-u) x^{r} y^{s} z^{t} w^{u} $;
\item[\rm (vi)] $A^*_3(x^{r} y^{s} z^{t} w^{u} ) = (r-s-t+u) x^{r} y^{s} z^{t} w^{u} $.
\end{enumerate}
\end{proposition}
\begin{proof} By Lemmas \ref{lem:sl4onP1} and  \ref{lem:derAction}.
\end{proof} 

\noindent We have a comment.
\begin{lemma} \label{lem:PNAAA} For $N \in \mathbb N$ the subspace $P_N$ is a submodule of
the $\mathfrak{sl}_4(\mathbb C)$-module $P$.
\end{lemma}
\begin{proof} By Proposition \ref{lem:ActP}, we have
\begin{align*}
A_i (P_N) \subseteq P_N, \qquad \qquad A^*_i (P_N) \subseteq P_N
\end{align*}
for $i \in \lbrace 1,2,3 \rbrace$.
\end{proof}

\noindent Referring to Lemma \ref{lem:PNAAA}, the submodule $P_N$ is irreducible by \cite[p.~97]{janzten}.
\medskip

\noindent Let $N \in \mathbb N$, and consider the $\mathfrak{sl}_4(\mathbb C)$-module $P_N$.
By Proposition \ref{lem:ActP}(iv)--(vi), the $P_N$-basis \eqref{eq:PDbasis}  diagonalizes the Cartan subalgebra $\mathbb H^*$.
Consequently,  $P_N$ is the direct sum of its $\mathbb H^*$-weight spaces. 
Our next goal is to describe these weight spaces.

\begin{lemma} \label{lem:rR} For natural numbers $r,s,t,u$ and $R,S,T,U$ we have
\begin{align*}
r=R ,\qquad s=S, \qquad t=T, \qquad u=U
\end{align*}
if and only if
\begin{align*}
r+s+t+u&=R+S+T+U,\\
r+s-t-u&=R+S-T-U,\\
r-s+t-u&=R-S+T-U,\\
r-s-t+u&=R-S-T+U.
\end{align*}
\end{lemma}
\begin{proof} Because the matrix $\Upsilon$ in Definition \ref{def:T} is invertible.
\end{proof}

\begin{lemma} \label{lem:WSdim1} Let $N \in \mathbb N$ and consider the $\mathfrak{sl}_4(\mathbb C)$-module $P_N$.
\begin{enumerate}
\item[\rm (i)] 
 Each $\mathbb H^*$-weight space has dimension one.
 \item[\rm (ii)] The $\mathbb H^*$-weight spaces are in bijection with $\mathcal P_N$.
\item[\rm (iii)] For $(r,s,t,u) \in \mathcal P_N$ the following holds on the 
corresponding $\mathbb H^*$-weight space:
\begin{align*} 
A^*_1 = (r+s-t-u)I, \qquad \quad A^*_2 = (r-s+t-u)I, \qquad \quad A^*_3 = (r-s-t+u)I.
\end{align*}
\end{enumerate}
\end{lemma}
\begin{proof} By  Proposition \ref{lem:ActP}(iv)--(vi) and Lemma \ref{lem:rR}.
\end{proof}

\noindent As an aside, we describe the $\mathbb H^*$-weight spaces from another point of view.
We will make a change of variables.

\begin{lemma} \label{lem:ChangeVar}  Let $N \in \mathbb N$. Pick $(r,s,t,u) \in \mathcal P_N$ and recall that
\begin{align*}
N=r+s+t+u. 
\end{align*}
\noindent Define
\begin{align*}
\lambda = r+s-t-u, \qquad \quad \mu = r-s+t-u, \qquad \quad \nu= r-s-t+u.
\end{align*}
Then
\begin{align*}
&r = \frac{N+\lambda + \mu + \nu}{4}, \qquad \qquad s = \frac{ N+\lambda - \mu - \nu}{4}, \\
&t = \frac{N-\lambda + \mu - \nu}{4}, \qquad \qquad u = \frac{ N-\lambda - \mu + \nu}{4}.
\end{align*}
\end{lemma} 
\begin{proof} Because the matrix $\Upsilon$ in Definition \ref{def:T} satisfies $\Upsilon^2 = I$.
\end{proof}

\begin{definition}\label{def:PnP}\rm For $N \in \mathbb N$, let the set $\mathcal P'_N$ consist of the $3$-tuples $(\lambda, \mu, \nu)$  such that 
\begin{align*}
&\lambda, \mu, \nu \in \lbrace N, N-2, N-4, \ldots, -N\rbrace, \qquad \quad N+\lambda+\mu+\nu \;\; \hbox{\rm is divisible by 4}, \\
& N+\lambda+\mu+\nu\geq 0, \qquad\qquad N+\lambda - \mu - \nu\geq 0, \\
& N-\lambda+\mu-\nu\geq 0, \qquad\qquad N-\lambda - \mu + \nu\geq 0.
\end{align*}
\end{definition}

\begin{lemma} \label{lem:PNVSPNP} For $N \in \mathbb N$, there exists a bijection $\mathcal P_N \to \mathcal P'_N$ that sends
\begin{align*}
(r,s,t,u) \mapsto (r+s-t-u, r-s+t-u, r-s-t+u).
\end{align*}
The inverse bijection sends
\begin{align*}
(\lambda, \mu, \nu) \mapsto \biggl ( 
 \frac{N+\lambda+\mu+\nu}{4},
 \frac{N+\lambda-\mu-\nu}{4},
  \frac{N-\lambda+\mu-\nu}{4},
   \frac{N-\lambda-\mu+\nu}{4}
  \biggr ).
\end{align*}
\end{lemma}
\begin{proof} This is readily checked using Lemma \ref{lem:ChangeVar}.
\end{proof}

\begin{lemma} \label{lem:together} Let $N \in \mathbb N$ and consider the $\mathfrak{sl}_4(\mathbb C)$-module $P_N$. The $\mathbb H^*$-weight spaces 
are in bijection with $\mathcal P'_N$, in such a way that for $(\lambda, \mu,\nu) \in \mathcal P'_N$ the following holds on the corresponding $\mathbb H^*$-weight space:
\begin{align*}
  A^*_1 = \lambda I,  \qquad \quad A^*_2 = \mu I, \qquad \quad A^*_3 = \nu I.
\end{align*}
\end{lemma}
\begin{proof} By Definition \ref{def:PnP} and  Lemmas \ref{lem:WSdim1}, \ref{lem:PNVSPNP}.
\end{proof}

\noindent We have a comment.

\begin{lemma} \label{lem:ASspec} For  $i \in \lbrace 1,2,3\rbrace$ and $N \in \mathbb N$       the eigenvalues of $A^*_i$ on $P_N$ are $\lbrace N-2n\rbrace_{n=0}^N$. For $0 \leq n \leq N$
the $(N-2n)$-eigenspace for $A^*_i$ on $P_N$ has dimension $(n+1)(N-n+1)$. 
\end{lemma}
\begin{proof} We first prove our assertions for $A^*_1$. By Definition \ref{def:PnP} and Lemma  \ref{lem:together},  the eigenvalues of $A^*_1$ on $P_N$ are $\lbrace N-2n\rbrace_{n=0}^N$.
Pick a natural number $n$ at most $N$, and let $W$ denote the $(N-2n)$-eigenspace for $A^*_1$ on $P_N$. The subspace $W$ is a direct sum of $\mathbb H^*$-weight spaces.
The $\mathbb H^*$-weight spaces in question correspond (via Lemma \ref{lem:WSdim1}(ii),(iii)) to the elements $(r,s,t,u) \in \mathcal P_N$ such that $r+s=N-n$ and $t+u=n$.
There are exactly $N-n+1$ nonnegative integer solutions to $r+s=N-n$. There are exactly $n+1$ nonnegative integer solutions to $t+u=n$.
Consequently,  there are exactly $(n+1)(N-n+1)$ elements $(r,s,t,u ) \in \mathcal P_N$ such that $r+s=N-n$ and $t+u=n$. This shows that $W$ has dimension $(n+1)(N-n+1)$.
We have proved our assertions for $A^*_1$. Our assertions for $A^*_2, A^*_3$ are similarly proved.
\end{proof}

\noindent We have been discussing  $\mathbb H^*$. In Section 7, we will have a similar discussion about $\mathbb H$.

\section{Some derivations and multiplication maps}

Recall the polynomial algebra $P= \mathbb C\lbrack x,y,z,w\rbrack$. In the previous section,
we turned $P$ into an
$\mathfrak{sl}_4(\mathbb C)$-module. In this section, we describe the $\mathfrak{sl}_4(\mathbb C)$-module $P$ using four partial derivatives and
four multiplication maps.
\medskip

\noindent  Consider the partial derivatives 
\begin{align*}
D_x=\frac{\partial}{\partial x}, \qquad \quad 
D_y=\frac{\partial}{\partial y}, \qquad \quad 
D_z=\frac{\partial}{\partial z}, \qquad \quad 
D_w=\frac{\partial}{\partial w}.
\end{align*}
These derivatives act as follows on the  $P$-basis \eqref{eq:Pbasis}. For $r,s,t,u \in \mathbb N$,
\begin{align} \label{eq:one}
&D_x \bigl(x^r y^s z^t w^u\bigr) = r x^{r-1} y^s z^t w^u, \qquad \qquad
D_y  \bigl(x^r y^s z^t w^u\bigr) = s x^{r} y^{s-1} z^t w^u, \\
& D_z \bigl(x^r y^s z^t w^u\bigr) = t x^{r} y^s z^{t-1} w^u, \qquad \qquad
D_w  \bigl(x^r y^s z^t w^u\bigr) = u x^{r} y^{s} z^t w^{u-1}. \label{eq:two}
\end{align}
\noindent 
For $N \in \mathbb N$ we have
\begin{align*}
D_x (P_N) = P_{N-1}, \qquad 
D_y (P_N) = P_{N-1}, \qquad 
D_z (P_N) = P_{N-1}, \qquad 
D_w  (P_N) = P_{N-1}.
\end{align*}

\begin{lemma} \label{lem:DerDef} The following {\rm (i)--(iv)} hold:
\begin{enumerate}
\item[\rm (i)] $D_x$ is the unique element in ${\rm Der}(P)$ that sends
\begin{align*}
x \mapsto 1, \qquad \quad y \mapsto 0, \qquad \quad z \mapsto 0, \qquad \quad w \mapsto 0;
\end{align*}
\item[\rm (ii)] $D_y$ is the unique element in ${\rm Der}(P)$ that sends
\begin{align*}
x \mapsto 0, \qquad \quad y \mapsto 1, \qquad \quad z \mapsto 0, \qquad \quad w \mapsto 0;
\end{align*}
\item[\rm (iii)] $D_z$ is the unique element in ${\rm Der}(P)$ that sends
\begin{align*}
x \mapsto 0, \qquad \quad y \mapsto 0, \qquad \quad z \mapsto 1, \qquad \quad w \mapsto 0;
\end{align*}
\item[\rm (iv)] $D_w$ is the unique element in ${\rm Der}(P)$ that sends
\begin{align*}
x \mapsto 0, \qquad \quad y \mapsto 0, \qquad \quad z \mapsto 0, \qquad \quad w \mapsto 1.
\end{align*}
\end{enumerate}
\end{lemma}
\begin{proof} By Lemma \ref{lem:derAction}  and \eqref{eq:one},  \eqref{eq:two}.
\end{proof}

\begin{definition}\label{def:MMMM} \rm We define $M_x, M_y, M_z, M_w \in {\rm End}(P)$ as follows.
For $f \in P$,
\begin{align*}
M_x (f) = x f, \qquad \quad 
M_y(f) = y f, \qquad \quad
M_z (f) = z f, \qquad \quad
M_w (f) = w f.
\end{align*}
\end{definition}
\noindent The maps $M_x, M_y, M_z, M_w$  act as follows on the $P$-basis  \eqref{eq:Pbasis}. For $r,s,t,u \in \mathbb N$,
\begin{align} \label{eq:three}
&M_x \bigl(x^r y^s z^t w^u\bigr) = x^{r+1} y^s z^t w^u, \qquad \qquad
M_y  \bigl(x^r y^s z^t w^u\bigr) = x^{r} y^{s+1} z^t w^u, \\
& M_z \bigl(x^r y^s z^t w^u\bigr) =  x^{r} y^s z^{t+1} w^u, \qquad \qquad
M_w  \bigl(x^r y^s z^t w^u\bigr) =  x^{r} y^{s} z^t w^{u+1}. \label{eq:four}
\end{align}
\noindent  
 For $N \in \mathbb N$ we have
\begin{align*}
M_x (P_N) \subseteq P_{N+1}, \qquad 
M_y (P_N) \subseteq P_{N+1}, \qquad 
M_z (P_N) \subseteq P_{N+1}, \qquad 
M_w  (P_N) \subseteq P_{N+1}.
\end{align*}

\begin{lemma} \label{lem:Weyl}  {\rm (See \cite[p.~550]{rotman}.)} The following relations hold.
\begin{enumerate}
\item[\rm (i)] For $a \in \lbrace x,y,z,w\rbrace$,
\begin{align*}
\lbrack D_a, M_a \rbrack = I.
\end{align*}
\item[\rm (ii)] For distinct $a,b \in \lbrace x,y,z,w\rbrace$,
\begin{align*}
\lbrack D_a, D_b \rbrack=0, \qquad \quad 
\lbrack M_a, M_b \rbrack=0, \qquad \quad 
\lbrack D_a, M_b \rbrack=0.
\end{align*}
\end{enumerate}
\end{lemma}
\begin{proof}  We check that $\lbrack D_x, M_x \rbrack=I$. For $f \in P$, 
\begin{align*}
\lbrack D_x, M_x \rbrack (f) &= D_x M_x (f) - M_x D_x (f) = D_x (x f) - x D_x (f) \\
                                  &= D_x(x) f + x D_x (f) - x D_x(f) = f.
\end{align*}
Therefore  $\lbrack D_x, M_x \rbrack=I$. The remaining assertions are checked in a similar way.
\end{proof}

\begin{remark}\rm  The subalgebra of ${\rm End}(P)$ generated by $D_x, D_y, D_z, D_w$ and $M_x, M_y, M_z, M_w$
is called the   fourth Weyl algebra, see  \cite[p.~550]{rotman}.
\end{remark}

\begin{proposition} \label{lem:A6} On the $\mathfrak{sl}_4(\mathbb C)$-module $P$,
\begin{align*}
A_1 &= M_y D_x + M_x D_y  + M_w D_z + M_z D_w, \\
A_2 &= M_z D_x + M_w D_y  + M_x D_z + M_y D_w, \\
A_3 &= M_w D_x + M_z D_y  + M_y D_z + M_x D_w 
\end{align*}
and also
\begin{align*}
A^*_1 &= M_x D_x + M_y D_y  - M_z D_z - M_w D_w, \\
A^*_2 &= M_x D_x - M_y D_y  + M_z D_z - M_w D_w, \\
A^*_3 &= M_x D_x - M_y D_y  - M_z D_z + M_w D_w.
\end{align*}
\end{proposition}
\begin{proof} To verify these equations, apply each side to a $P$-basis vector from \eqref{eq:Pbasis}, 
and evaluate the result using Proposition  \ref{lem:ActP} along with \eqref{eq:one}--\eqref{eq:four}.
\end{proof}

\section{ A basis for $P$ that diagonalizes $\mathbb H$}

\noindent  We continue to discuss the $\mathfrak{sl}_4(\mathbb C)$-module $P= \mathbb C\lbrack x,y,z,w\rbrack$.
In Lemma \ref{lem:Cartan}  we described the Cartan subalgebras $\mathbb H$, $\mathbb H^*$ of  $\mathfrak{sl}_4(\mathbb C)$.
In Section 5, we displayed a basis for $P$ that diagonalizes $\mathbb H^*$.
In this section, we display a basis for $P$ that  diagonalizes $\mathbb H$.
\medskip

\noindent Recall that $x,y,z,w$ form a basis for $P_1$.

\begin{definition} \label{def:xyzws} \rm We define some  vectors in $P_1$:
\begin{align*}
& x^* = \frac{x+y+z+w}{2}, \qquad \qquad y^* = \frac{x+y-z-w}{2}, \\
& z^*= \frac{x-y+z-w}{2}, \qquad \qquad w^* = \frac{x-y-z+w}{2}.
\end{align*}
\end{definition}

\noindent Recall the matrix $\Upsilon$ from Definition \ref{def:T}.
\begin{lemma} \label{lem:3f} The following {\rm (i)--(iii)} hold:
\begin{enumerate}
\item[\rm (i)] the vectors $x^*, y^*, z^*, w^*$ form a basis for $P_1$;
\item[\rm (ii)] $\Upsilon$ is the transition matrix from the basis $x,y,z,w$ 
to the basis $x^*, y^*, z^*, w^*$;
\item[\rm (iii)]
$\Upsilon$ is the transition matrix from the basis $x^*, y^*, z^*, w^*$ to the basis
$x,y,z,w$.
\end{enumerate}
\end{lemma}
\begin{proof} By Definitions  \ref{def:T},  \ref{def:xyzws} and $\Upsilon^2=I$.
\end{proof}

\begin{lemma} \label{lem:xyzws}  We have
\begin{align*}
& x = \frac{x^*+y^*+z^*+w^*}{2}, \qquad \qquad y = \frac{x^*+y^*-z^*-w^*}{2}, \\
& z = \frac{x^*-y^*+z^*-w^*}{2}, \qquad \qquad w = \frac{x^*-y^*-z^*+w^*}{2}.
\end{align*}
\end{lemma}
\begin{proof}  This is a reformulation of Lemma \ref{lem:3f}(iii).
\end{proof}

\begin{lemma} \label{lem:AAABBBP1} 
Referring to the $\mathfrak{sl}_4(\mathbb C)$-module $P_1$,
\begin{enumerate}
\item[\rm (i)] $A_1$ sends
\begin{align*}
x^* \mapsto x^*, \qquad y^* \mapsto y^*, \qquad z^* \mapsto -z^*, \qquad w^* \mapsto -w^*;
\end{align*}
\item[\rm (ii)] $A_2$ sends
\begin{align*}
x^* \mapsto x^*, \qquad y^* \mapsto -y^*, \qquad z^* \mapsto z^*, \qquad w^* \mapsto -w^*;
\end{align*}
\item[\rm (iii)] $A_3$ sends
\begin{align*}
x^* \mapsto x^*, \qquad y^* \mapsto -y^*, \qquad z^* \mapsto -z^*, \qquad w^* \mapsto w^*;
\end{align*}
\item[\rm (iv)] 
$A^*_1$ swaps $x^* \leftrightarrow y^*$ and $z^* \leftrightarrow w^*$;
\item[\rm (v)] 
$A^*_2$ swaps $x^* \leftrightarrow z^*$ and $y^* \leftrightarrow w^*$;
\item[\rm (vi)] 
$A^*_3$ swaps $x^* \leftrightarrow w^*$ and $y^* \leftrightarrow z^*$.
\end{enumerate}
\end{lemma}
\begin{proof} By Lemmas \ref{lem:ATA} and \ref{lem:3f}.
\end{proof}

\noindent We have been discussing $P_1$. Next we consider $P$.

\begin{lemma} \label{lem:Pdualbasis} The following is a basis for $P$:
\begin{align}
x^{*r} y^{*s} z^{*t} w^{*u} \qquad \qquad r,s,t,u \in \mathbb N. \label{eq:Pdualbasis}
\end{align}
\end{lemma}
\begin{proof} The vectors $x^*, y^*,z^*, w^*$ form a basis for $P_1$.
\end{proof}

\begin{lemma} \label{lem:PDdualbasis} For $N \in \mathbb N$ the following is a basis for $P_N$:
\begin{align}
x^{*r} y^{*s} z^{*t} w^{*u} \qquad \qquad (r,s,t,u) \in \mathcal P_N.    \label{eq:PNDbasis}
\end{align}
\end{lemma}
\begin{proof} By Lemma  \ref{lem:Pdualbasis} and since 
each of  $x^*, y^*,z^*, w^*$ is homogeneous with total degree one.
\end{proof}

\begin{lemma} \label{lem:dualDer}
Each $\mathcal D \in {\rm Der}(P)$ acts on the $P$-basis \eqref{eq:Pdualbasis} as follows. For $r,s,t,u \in \mathbb N$,
\begin{align*}
\mathcal D(x^{*r} y^{*s} z^{*t} w^{*u}) &= r x^{* r-1} y^{*s} z^{*t} w^{*u} \mathcal D(x^*) +
                                                s x^{*r} y^{*s-1} z^{*t} w^{*u} \mathcal D(y^*) \\
                                              &\quad  + t x^{*r} y^{*s} z^{*t-1} w^{*u} \mathcal D(z^*)+
                                                u x^{*r} y^{*s} z^{*t} w^{*u-1} \mathcal D(w^*).
 \end{align*}
\end{lemma}
\begin{proof} Similar to the proof of Lemma  \ref{lem:derAction}$\bigl( {\rm (i)}\Rightarrow {\rm (ii)}\bigr)$.
\end{proof}

\begin{proposition} \label{lem:dbAction}
The $\mathfrak{sl}_4(\mathbb C)$-generators
 $A_1, A_2, A_3$ and $ A^*_1, A^*_2, A^*_3$ act on the $P$-basis \eqref{eq:Pdualbasis} as follows.
For $r,s,t,u \in \mathbb N$,
\begin{enumerate}
\item[\rm (i)] $A_1(x^{*r} y^{*s} z^{*t} w^{*u} ) = (r+s-t-u) x^{*r} y^{*s} z^{*t} w^{*u} $;
\item[\rm (ii)] $A_2(x^{*r} y^{*s} z^{*t} w^{*u} ) = (r-s+t-u) x^{*r} y^{*s} z^{*t} w^{*u} $;
\item[\rm (iii)] $A_3(x^{*r} y^{*s} z^{*t} w^{*u} ) = (r-s-t+u) x^{*r} y^{*s} z^{*t} w^{*u} $;
\item[\rm (iv)] the vector
\begin{align*}
 A^*_1( x^{*r}y^{*s} z^{*t} w^{*u})
 \end{align*}
 is a linear combination with the following terms and coefficients:
\begin{align*} 
\begin{tabular}[t]{c|c}
{\rm Term }& {\rm Coefficient} 
 \\
 \hline
   $ x^{*r-1} y^{*s+1} z^{*t} w^{*u} $& $r$   \\
 $ x^{*r+1} y^{*s-1} z^{*t} w^{*u} $   & $s$\\
  $ x^{*r} y^{*s} z^{*t-1} w^{*u+1}  $  & $t$ \\
  $ x^{*r} y^{*s} z^{*t+1} w^{*u-1} $& $u$
   \end{tabular}
\end{align*}
\item[\rm (v)] 
the vector
\begin{align*}
A^*_2(x^{*r} y^{*s} z^{*t} w^{*u})     
\end{align*}
 is a linear combination with the following terms and coefficients:
\begin{align*} 
\begin{tabular}[t]{c|c}
{\rm Term }& {\rm Coefficient} 
 \\
 \hline
   $ x^{*r-1} y^{*s} z^{*t+1} w^{*u} $& $ r$   \\
 $ x^{*r} y^{*s-1} z^{*t} w^{*u+1} $   & $s$\\
  $ x^{*r+1} y^{*s} z^{*t-1} w^{*u} $  & $t$ \\
  $ x^{*r} y^{*s+1} z^{*t} w^{*u-1}  $& $u $
   \end{tabular}
\end{align*}
\item[\rm (vi)] 
the vector
\begin{align*}
A^*_3( x^{*r} y^{*s} z^{*t} w^{*u} )
\end{align*}
 is a linear combination with the following terms and coefficients:
\begin{align*} 
\begin{tabular}[t]{c|c}
{\rm Term }& {\rm Coefficient} 
 \\
 \hline
    $ x^{*r-1} y^{*s} z^{*t} w^{*u+1} $& $r $   \\
 $x^{*r} y^{*s-1} z^{*t+1} w^{*u} $   & $s $\\
    $ x^{*r} y^{*s+1} z^{*t-1} w^{*u} $  & $t$ \\
  $x^{*r+1} y^{*s} z^{*t} w^{*u-1} $& $u $
   \end{tabular}
\end{align*}
\end{enumerate}
\end{proposition}
\begin{proof} By Lemmas  \ref{lem:AAABBBP1}  and  \ref{lem:dualDer}.
\end{proof}

\noindent Let $N \in \mathbb N$, and consider the $\mathfrak{sl}_4(\mathbb C)$-module $P_N$.
By Proposition \ref{lem:dbAction}(i)--(iii), the $P_N$-basis   \eqref{eq:PNDbasis}  diagonalizes $\mathbb H$.
Consequently,  $P_N$ is the direct sum of its $\mathbb H$-weight spaces.

\begin{lemma} \label{lem:WSdim1dual} Let $N \in \mathbb N$ and consider the $\mathfrak{sl}_4(\mathbb C)$-module $P_N$.
\begin{enumerate}
\item[\rm (i)] 
 Each $\mathbb H$-weight space has dimension one.
 \item[\rm (ii)] The $\mathbb H$-weight spaces are in bijection with $\mathcal P_N$.
\item[\rm (iii)] For $(r,s,t,u) \in \mathcal P_N$ the following holds on the 
corresponding $\mathbb H$-weight space:
\begin{align*} 
A_1 = (r+s-t-u)I, \qquad \quad A_2 = (r-s+t-u)I, \qquad \quad A_3 = (r-s-t+u)I.
\end{align*}
\end{enumerate}
\end{lemma}
\begin{proof}  By Lemma \ref{lem:rR} and Proposition \ref{lem:dbAction}(i)--(iii).
\end{proof}

\begin{lemma} \label{lem:Dualtogether} Let $N \in \mathbb N$ and consider the $\mathfrak{sl}_4(\mathbb C)$-module $P_N$. The $\mathbb H$-weight spaces 
are in bijection with $\mathcal P'_N$, in such a way that for $(\lambda, \mu,\nu) \in \mathcal P'_N$ the following holds on the corresponding $\mathbb H$-weight space:
\begin{align*}
  A_1 = \lambda I,  \qquad \quad A_2 = \mu I, \qquad \quad A_3 = \nu I.
\end{align*}
\end{lemma}
\begin{proof} By Definition \ref{def:PnP} and Lemmas  \ref{lem:PNVSPNP}, \ref{lem:WSdim1dual}.
\end{proof}

\begin{lemma} \label{lem:Aspec} For $i \in \lbrace 1,2,3\rbrace$  and  $N \in \mathbb N$      the eigenvalues of $A_i$ on $P_N$ are $\lbrace N-2n\rbrace_{n=0}^N$. For $0 \leq n \leq N$
the $(N-2n)$-eigenspace for $A_i$ on $P_N$ has dimension $(n+1)(N-n+1)$. 
\end{lemma}
\begin{proof} Similar to the proof of Lemma \ref{lem:ASspec}.
\end{proof}

\begin{proposition} \label{thm:aut2} There exists an automorphism $\sigma$ of $P$ that sends
\begin{align*}
x \leftrightarrow x^*, \qquad \quad
y \leftrightarrow y^*, \qquad \quad
z \leftrightarrow z^*, \qquad \quad
w \leftrightarrow w^*.
\end{align*}
We have $\sigma^2 = {\rm id}$. For $i \in \lbrace 1,2,3\rbrace$ the following holds on the $\mathfrak{sl}_4(\mathbb C)$-module $P$:
\begin{align}
A^*_i = \sigma A_i \sigma^{-1},   \qquad \qquad A_i = \sigma A^*_i \sigma^{-1}.           \label{eq:SRS}
\end{align}
\end{proposition}
\begin{proof}  The vector space $P_1$ has a basis $x,y,z,w$ and a basis $x^*, y^*, z^*, w^*$. Therefore, there exists
an automorphism $\sigma$ of $P$ that sends
\begin{align*}
x \mapsto x^*, \qquad \quad y \mapsto y^*, \qquad \quad z\mapsto z^*, \qquad \quad w \mapsto w^*.
\end{align*}
By Definition  \ref{def:xyzws} and Lemma  \ref{lem:xyzws}, the automorphism $\sigma$ sends
\begin{align*}
x^* \mapsto x, \qquad \quad y^* \mapsto y, \qquad \quad z^*\mapsto z, \qquad \quad w^* \mapsto w.
\end{align*}
\noindent Therefore  $\sigma^2 = {\rm id}$. Let $i \in \lbrace 1,2,3\rbrace$. Comparing Propositions \ref{lem:ActP}, \ref{lem:dbAction}
we find that $A^*_i \sigma = \sigma A_i$ holds on $P$. This yields   \eqref{eq:SRS}.
\end{proof}

\noindent Recall the automorphism $\tau$ of $\mathfrak{sl}_4(\mathbb C)$ from Lemma \ref{lem:tau}. 

\begin{proposition}
For $\varphi \in \mathfrak{sl}_4(\mathbb C)$ the following holds on the $\mathfrak{sl}_4(\mathbb C)$-module $P$:
\begin{align}
\tau(\varphi) = \sigma \varphi \sigma^{-1}.  \label{eq:SRS2}
\end{align}
\end{proposition}
\begin{proof} 

By \eqref{eq:SRS}  and Lemma  \ref{lem:Tint}, the following holds on $P$:  
\begin{align*}
\tau(A_i) = \sigma A_i \sigma^{-1}, \qquad \qquad \tau(A^*_i) = \sigma A^*_i \sigma^{-1} \qquad \quad i \in \lbrace 1,2,3 \rbrace.
\end{align*}
This yields \eqref{eq:SRS2} because the Lie algebra $\mathfrak{sl}_4(\mathbb C)$ is generated by 
\begin{align*}
A_i, \quad A^*_i \qquad \quad i \in \lbrace 1,2,3 \rbrace.
\end{align*}
\end{proof}

\noindent We have a comment about the automorphism $\sigma$ of $P$ from Proposition \ref{thm:aut2}. The map $\sigma$ acts on
the $P$-bases \eqref{eq:Pbasis}  and \eqref{eq:Pdualbasis}  as follows.
For $r,s,t,u \in \mathbb N$ the map $\sigma$ swaps
\begin{align*}
x^r y^s z^t w^u \leftrightarrow x^{*r} y^{*s} z^{*t} w^{*u}.
\end{align*}
Moreover, $\sigma(P_N)=P_N$ for $N \in \mathbb N$.

\section{More derivations and multiplication maps}

We continue to discuss the $\mathfrak{sl}_4(\mathbb C)$-module $P= \mathbb C\lbrack x,y,z,w\rbrack$. 
In Section 6, we investigated the derivations $D_x, D_y, D_z, D_w$ and the multiplication maps
$M_x, M_y, M_z, M_w$. In this section, we introduce the analogous derivations 
$D_{x^*}$, $D_{y^*}$, $D_{z^*}$, $D_{w^*}$ and multiplication maps
$M_{x^*}$, $M_{y^*}$, $M_{z^*}$, $M_{w^*}$.
\medskip

\noindent The next result is motivated by Lemma \ref{lem:DerDef}.

\begin{lemma} \label{lem:DerDefdual} The following {\rm (i)--(iv)} hold.
\begin{enumerate}
\item[\rm (i)] There exists a unique element ${ D_{x^*} \in \rm Der}(P)$ that sends
\begin{align*}
x^* \mapsto 1, \qquad \quad y^* \mapsto 0, \qquad \quad z^* \mapsto 0, \qquad \quad w^* \mapsto 0.
\end{align*}
\item[\rm (ii)] There exists a unique element  $ D_{y^*} \in {\rm Der}(P)$ that sends
\begin{align*}
x^* \mapsto 0, \qquad \quad y^* \mapsto 1, \qquad \quad z^* \mapsto 0, \qquad \quad w^* \mapsto 0.
\end{align*}
\item[\rm (iii)] There exists a unique element  $ D_{z^*}\in {\rm Der}(P)$ that sends
\begin{align*}
x^* \mapsto 0, \qquad \quad y^* \mapsto 0, \qquad \quad z^* \mapsto 1, \qquad \quad w^* \mapsto 0.
\end{align*}
\item[\rm (iv)] There exists a  unique element  $D_{w^*} \in {\rm Der}(P)$ that sends
\begin{align*}
x^*\mapsto 0, \qquad \quad y^* \mapsto 0, \qquad \quad z^* \mapsto 0, \qquad \quad w^* \mapsto 1.
\end{align*}
\end{enumerate}
\end{lemma}
\begin{proof} By Lemma \ref{lem:derExt}.
\end{proof}

\begin{lemma}\label{lem:DDact} The  derivations
$D_{x^*}$, $D_{y^*}$, $D_{z^*}$, $D_{w^*}$
act as follows on the  $P$-basis \eqref{eq:Pdualbasis}. For $r,s,t,u \in \mathbb N$,
\begin{align*}
&D_{x^*} \bigl(x^{*r} y^{*s} z^{*t} w^{*u}\bigr) = r x^{* r-1} y^{*s} z^{*t} w^{*u}, \qquad \quad
D_{y^*}  \bigl(x^{*r} y^{*s} z^{*t} w^{*u}\bigr) = s x^{*r} y^{* s-1} z^{*t} w^{*u}, \\
& D_{z^*} \bigl(x^{*r} y^{*s} z^{*t} w^{*u}\bigr) = t x^{*r} y^{*s} z^{* t-1} w^{*u}, \qquad \quad
D_{w^*}  \bigl(x^{*r} y^{*s} z^{*t} w^{*u}\bigr) = u x^{*r} y^{*s} z^{*t} w^{* u-1}. 
\end{align*}
\end{lemma}
\begin{proof} By Lemmas \ref{lem:dualDer}, \ref{lem:DerDefdual}.
\end{proof}

\begin{lemma} \label{lem:GradeDual}
For $N \in \mathbb N$ we have
\begin{align*}
D_{x^*} (P_N) = P_{N-1}, \qquad 
D_{y^*} (P_N) = P_{N-1}, \qquad 
D_{z^*} (P_N) = P_{N-1}, \qquad 
D_{w^*}  (P_N) = P_{N-1}.
\end{align*}
\end{lemma}
\begin{proof} By Lemmas \ref{lem:PDdualbasis},  \ref{lem:DDact}.
\end{proof}

\begin{lemma}  \label{lem:DDs} We have 
\begin{align*}
& D_{x^*} = \frac{D_x+D_y+D_z+D_w}{2}, \qquad \qquad D_{y^*} = \frac{D_x+D_y-D_z-D_w}{2}, \\
& D_{z^*}= \frac{D_x-D_y+D_z-D_w}{2}, \qquad \qquad D_{w^*} = \frac{D_x-D_y-D_z+D_w}{2}.
\end{align*}
\end{lemma}
\begin{proof} For these equations, each side is a derivation. So by Lemma \ref{lem:zeroDer}, these equations hold  if they hold on $P_1$.
By Definition  \ref{def:xyzws} or Lemma \ref{lem:xyzws}, these equations hold on $P_1$.
\end{proof}

\begin{lemma}  We have
\begin{align*}
& D_x = \frac{D_{x^*}+D_{y^*}+D_{z^*}+D_{w^*}}{2}, \qquad \quad D_y = \frac{D_{x^*}+D_{y^*}-D_{z^*}-D_{w^*}}{2}, \\
& D_z = \frac{D_{x^*}-D_{y^*}+D_{z^*}-D_{w^*}}{2}, \qquad \quad D_w = \frac{D_{x^*}-D_{y^*}-D_{z^*}+D_{w^*}}{2}.
\end{align*}
\end{lemma}
\begin{proof}  By Lemma \ref{lem:DDs} and linear algebra.
\end{proof}

\begin{lemma} \label{lem:sDs} The automorphism $\sigma$ of $P$ satisfies
\begin{align*}
D_{x^*} = \sigma D_x \sigma^{-1}, \qquad \quad
D_{y^*} = \sigma D_y\sigma^{-1}, \qquad \quad
D_{z^*} = \sigma D_z \sigma^{-1}, \qquad \quad
D_{w^*} = \sigma D_w \sigma^{-1}.
\end{align*}
\end{lemma}
\begin{proof}  Compare \eqref{eq:one}, \eqref{eq:two} with Lemma \ref{lem:DDact}
using the comment at the end of Section 7.
\end{proof}

\begin{definition}\label{def:MMMMd} \rm We define $M_{x^*}, M_{y^*}, M_{z^*}, M_{w^*} \in {\rm End}(P)$ as follows.
For $f \in P$,
\begin{align*}
M_{x^*} (f) = x^* f, \qquad \quad 
M_{y^*}(f) = y^* f, \qquad \quad
M_{z^*} (f) = z^* f, \qquad \quad
M_{w^*} (f) = w^* f.
\end{align*}
\end{definition}

\noindent The maps $M_{x^*}, M_{y^*}, M_{z^*}, M_{w^*}$  act as follows on the $P$-basis  \eqref{eq:Pdualbasis}. For $r,s,t,u \in \mathbb N$,
\begin{align*} 
&M_{x^*} \bigl(x^{*r} y^{*s} z^{*t} w^{*u}\bigr) = x^{* r+1} y^{*s} z^{*t} w^{*u}, \qquad \quad
M_{y^*}  \bigl(x^{*r} y^{*s} z^{*t} w^{*u}\bigr) = x^{*r} y^{* s+1} z^{*t} w^{*u}, \\
& M_{z^*} \bigl(x^{*r} y^{*s} z^{*t} w^{*u}\bigr) =  x^{*r} y^{*s} z^{* t+1} w^{*u}, \qquad \quad
M_{w^*}  \bigl(x^{*r} y^{*s} z^{*t} w^{*u}\bigr) =  x^{*r} y^{*s} z^{*t} w^{* u+1}. 
\end{align*}
\noindent  
 For $N \in \mathbb N$ we have
\begin{align*}
M_{x^*} (P_N) \subseteq P_{N+1}, \qquad 
M_{y^*} (P_N) \subseteq P_{N+1}, \qquad 
M_{z^*} (P_N) \subseteq P_{N+1}, \qquad 
M_{w^*}  (P_N) \subseteq P_{N+1}.
\end{align*}

\begin{lemma}  \label{lem:MMs} We have 
\begin{align*}
& M_{x^*} = \frac{M_x+M_y+M_z+M_w}{2}, \qquad \qquad M_{y^*} = \frac{M_x+M_y-M_z-M_w}{2}, \\
& M_{z^*}= \frac{M_x-M_y+M_z-M_w}{2}, \qquad \qquad M_{w^*} = \frac{M_x-M_y-M_z+M_w}{2}.
\end{align*}
\end{lemma}
\begin{proof} By Definitions \ref{def:MMMM},    \ref{def:xyzws}, \ref{def:MMMMd}.
\end{proof}

\begin{lemma}  We have
\begin{align*}
& M_x = \frac{M_{x^*}+M_{y^*}+M_{z^*}+M_{w^*}}{2}, \qquad \quad M_y = \frac{M_{x^*}+M_{y^*}-M_{z^*}-M_{w^*}}{2}, \\
& M_z = \frac{M_{x^*}-M_{y^*}+M_{z^*}-M_{w^*}}{2}, \qquad \quad M_w = \frac{M_{x^*}-M_{y^*}-M_{z^*}+M_{w^*}}{2}.
\end{align*}
\end{lemma}
\begin{proof} 
By Lemma \ref{lem:xyzws} and Definitions \ref{def:MMMM},  \ref{def:MMMMd}.
 \end{proof}

\begin{lemma} \label{lem:sMs} The automorphism $\sigma$ of $P$ satisfies
\begin{align*}
M_{x^*} = \sigma M_x \sigma^{-1}, \qquad \quad
M_{y^*} = \sigma M_y\sigma^{-1}, \qquad \quad
M_{z^*} = \sigma M_z \sigma^{-1}, \qquad \quad
M_{w^*} = \sigma M_w \sigma^{-1}.
\end{align*}
\end{lemma}
\begin{proof}  We prove the first equation. We have $M_{x^*} \sigma = \sigma M_x$ because for all $f \in P$,
\begin{align*}
M_{x^*} \sigma (f) = x^* \sigma(f)  = \sigma(x) \sigma(f)  = \sigma(xf) = \sigma M_x (f).
\end{align*}
The first equation is proved. The remaining equations are similarly proved.
\end{proof}

\begin{lemma} \label{lem:WeylD} The following relations hold.
\begin{enumerate}
\item[\rm (i)] For $a \in \lbrace x^* ,y^* ,z^*,w^*\rbrace$,
\begin{align*}
\lbrack D_a, M_a \rbrack = I.
\end{align*}
\item[\rm (ii)] For distinct $a,b \in \lbrace x^*,y^*,z^*,w^*\rbrace$,
\begin{align*}
\lbrack D_a, D_b \rbrack=0, \qquad \quad 
\lbrack M_a, M_b \rbrack=0, \qquad \quad 
\lbrack D_a, M_b \rbrack=0.
\end{align*}
\end{enumerate}
\end{lemma}
\begin{proof} In  Lemma \ref{lem:Weyl} conjugate each term by $\sigma$, and evaluate the
result using Lemmas \ref{lem:sDs}, \ref{lem:sMs}.
\end{proof}

\begin{proposition} \label{lem:A6D} On the $\mathfrak{sl}_4(\mathbb C)$-module $P$,
\begin{align*}
A^*_1 &= M_{y^*} D_{x^*} + M_{x^*} D_{y^*}  + M_{w^*} D_{z^*} + M_{z^*} D_{w^*}, \\
A^*_2 &= M_{z^*} D_{x^*} + M_{w^*} D_{y^*}  + M_{x^*} D_{z^*} + M_{y^*} D_{w^*}, \\
A^*_3 &= M_{w^*} D_{x^*} + M_{z^*} D_{y^*}  + M_{y^*} D_{z^*} + M_{x^*} D_{w^*} 
\end{align*}
and also
\begin{align*}
A_1 &= M_{x^*} D_{x^*} + M_{y^*} D_{y^*}  - M_{z^*} D_{z^*} - M_{w^*} D_{w^*}, \\
A_2 &= M_{x^*} D_{x^*} - M_{y^*} D_{y^*}  + M_{z^*} D_{z^*} - M_{w^*} D_{w^*}, \\
A_3 &= M_{x^*} D_{x^*} - M_{y^*} D_{y^*}  - M_{z^*} D_{z^*} + M_{w^*} D_{w^*}.
\end{align*}
\end{proposition}
\begin{proof}  In Proposition \ref{lem:A6} conjugate each term  by $\sigma$, and evaluate the result using  \eqref{eq:SRS} along with
 Lemmas \ref{lem:sDs}, \ref{lem:sMs}.
\end{proof}


\section{A Hermitian form on $P$}
 We continue to discuss the $\mathfrak{sl}_4(\mathbb C)$-module $P= \mathbb C\lbrack x,y,z,w\rbrack$.
 In this section, we endow the vector space
$P$ with a Hermitian form $\langle \,,\,\rangle$  with respect to which the $P$-bases \eqref{eq:Pbasis} and \eqref{eq:Pdualbasis} are orthogonal.

\begin{definition}\label{def:Herm} \rm Let $W$ denote a vector space. A {\it Hermitian form on $W$} is a function $\langle \,,\,\rangle: W \times W \to \mathbb C$
such that:
\begin{enumerate}
\item[\rm (i)] $\langle f+g, h \rangle = \langle f, h\rangle + \langle g, h \rangle$ for all $f,g,h \in W$;
\item[\rm (ii)] $\langle \alpha f, g \rangle = \alpha \langle f,g \rangle$ for all $\alpha \in \mathbb C$ and $f,g\in W$;
\item[\rm (iii)] $\langle f,g \rangle = \overline {\langle g,f \rangle}$ for all $f,g \in W$.
\end{enumerate}
\end{definition}
\noindent For a Hermitian form $\langle \,,\,\rangle $ on $W$, we abbreviate $\Vert f \Vert^2 = \langle f,f\rangle$ for all $f \in W$.

\begin{definition} \label{def:bform} \rm We endow the vector space $P$ with a Hermitian form $\langle \,,\,\rangle$ with respect to which
the basis vectors
\begin{align*}
x^r y^s z^t w^u \qquad \qquad r,s,t,u \in \mathbb N
\end{align*}
are mutually orthogonal and
\begin{align}
\Vert x^r y^s z^t w^u \Vert^2 = r!s!t!u!  \qquad \qquad r,s,t,u \in \mathbb N. \label{eq:PNsn}
\end{align}
\end{definition}

\begin{lemma} \label{lem:PDorthog} The homogeneous components $\lbrace P_N \rbrace_{N \in \mathbb N}$ are mutually 
orthogonal with respect to $\langle \,,\,\rangle$.
\end{lemma} 
\begin{proof} By Lemma  \ref{lem:PDbasic} and Definition \ref{def:bform}.
\end{proof}

\begin{lemma} \label{lem:bdb} For $N \in \mathbb N$ the $P_N$-basis 
\begin{align*}
x^r y^s z^t w^u \qquad \qquad (r,s,t,u) \in \mathcal P_N 
\end{align*}
and the $P_N$-basis 
\begin{align*}
\frac{x^r y^s z^t w^u}{r!s!t!u!} \qquad \qquad (r,s,t,u) \in \mathcal P_N
\end{align*}
are dual with respect to $\langle \,,\,\rangle$.
\end{lemma}
\begin{proof}  We invoke Definition \ref{def:bform}. For $(r,s,t,u) \in \mathcal P_N$  and $(R,S,T,U) \in \mathcal P_N$ we have
\begin{align*}
\Bigl \langle \frac{x^r y^s z^t w^u}{r!s!t!u!}, x^R y^S z^T w^U \Bigr \rangle = \delta_{r,R} \delta_{s,S} \delta_{t,T} \delta_{u,U}.
\end{align*}
\end{proof}

\begin{lemma} \label{lem:xnExpand} For $N \in \mathbb N$,
\begin{align} \label{eq:xs}
x^{*N} = \frac{N!}{2^{N}} \sum_{(r,s,t,u) \in \mathcal P_N}   \frac{x^r y^s z^t w^u}{r!s!t!u!}.
\end{align}
\end{lemma}
\begin{proof} By Definition \ref{def:xyzws},
\begin{align*}
x^{*N} = \Bigl ( \frac{x+y+z+w}{2} \Bigr)^N.
\end{align*}
In this equation, expand the right-hand side using the multinomial theorem.
\end{proof}

\begin{lemma}\label{lem:Normal} For $N \in \mathbb N$ and $(r,s,t,u) \in \mathcal P_N$,
\begin{align*}
\Bigl \langle x^r y^s z^t w^u, x^{*N} \Bigr \rangle = \frac{N!}{2^N}.
\end{align*}
\end{lemma}
\begin{proof} To verify the result, eliminate $x^{*N}$ using  \eqref{eq:xs} and evaluate the result using  Definition \ref{def:bform}.
\end{proof}

\begin{lemma} \label{lem:bilDM2} For $f,g \in P$ we have
\begin{align}
&\langle D_x f, g \rangle = \langle f, M_x g\rangle, \qquad \qquad
\langle D_y f, g \rangle = \langle f, M_y g\rangle, \label{eq:w1} \\
&\langle D_z f, g \rangle = \langle f, M_z g\rangle, \qquad \qquad
\langle D_w f, g \rangle = \langle f, M_w g\rangle. \label{eq:w2}
\end{align}
\end{lemma}
\begin{proof} Without loss of generality, we may assume that $f,g$ are contained in
the $P$-basis  \eqref{eq:Pbasis}.  Under this assumption  \eqref{eq:w1}, \eqref{eq:w2}  are routinely 
checked using Definition \ref{def:bform} and 
\eqref{eq:one}--\eqref{eq:four}.
\end{proof}

\begin{lemma} \label{lem:bilDM3} For $f,g \in P$ we have
\begin{align}
&\langle M_x f, g \rangle = \langle f, D_x g\rangle, \qquad \qquad
\langle M_y f, g \rangle = \langle f, D_y g\rangle, \label{eq:w1d} \\
&\langle M_z f, g \rangle = \langle f, D_z g\rangle, \qquad \qquad
\langle M_w f, g \rangle = \langle f, D_w g\rangle. \label{eq:w2d}
\end{align}
\end{lemma}
\begin{proof} By Definition \ref{def:Herm}(iii)  and Lemma \ref{lem:bilDM2}.
\end{proof}

\begin{lemma} \label{lem:inv} For $i \in \lbrace 1,2,3 \rbrace$ and $f,g \in P$,
\begin{align*}
\langle A_i f, g \rangle = \langle f, A_i g\rangle,  \qquad \qquad \langle A^*_i f, g \rangle = \langle f, A^*_i g\rangle.
\end{align*}
\end{lemma}
\begin{proof}  We first show that $\langle A_1 f, g\rangle = \langle f, A_1 g\rangle$.
By Proposition  \ref{lem:A6}, the following holds on $P$:
\begin{align*}
A_1 = M_y D_x + M_x D_y + M_w D_z + M_z D_w.
\end{align*}
Using this and Lemmas  \ref{lem:bilDM2}, \ref{lem:bilDM3} we obtain
\begin{align*}
\langle A_1 f, g\rangle &= \langle  M_y D_x f,g\rangle + \langle M_x D_y f,g\rangle + \langle M_w D_z f,g\rangle +\langle  M_z D_w f,g\rangle \\
 &= \langle  D_x f,  D_yg\rangle + \langle  D_y f, D_x g\rangle + \langle  D_z f, D_w g\rangle +\langle  D_w f, D_z g\rangle \\
  &= \langle   f, M_x D_yg\rangle + \langle  f, M_y D_x g\rangle + \langle   f, M_z D_w g\rangle +\langle  f, M_w D_z g\rangle \\
  & = \langle f, A_1 g\rangle.
\end{align*}
We have shown that $\langle A_1 f, g\rangle = \langle f, A_1 g\rangle$. The other assertions are similarly shown.
\end{proof}

\begin{proposition} \label{lem:bilinear}  With respect to  $\langle \,,\,\rangle$ 
the vectors
\begin{align}
x^{*r} y^{*s} z^{*t} w^{*u} \qquad \qquad r,s,t,u \in \mathbb N \label{eq:dbasis}
\end{align}
are mutually orthogonal and
\begin{align} \label{eq:sn}
\Vert x^{*r} y^{*s} z^{*t} w^{*u} \Vert^2 = r!s!t!u! \qquad \qquad r,s,t,u \in \mathbb N.
\end{align}
\end{proposition}
\begin{proof} We first show that the vectors \eqref{eq:dbasis} are mutually orthogonal.
To this end, let $f,g$ denote vectors from \eqref{eq:dbasis} such that $\langle f,g\rangle \not=0$.
We show that $f=g$. Write
\begin{align*}
f= x^{*r} y^{*s} z^{*t} w^{*u}, \qquad \qquad g=x^{*R}y^{*S}z^{*T}w^{*U}.
\end{align*} 
By Lemma \ref{lem:PDorthog}  and since $\langle f,g\rangle\not=0$,
\begin{align} \label{eq:R1}
r+s+t+u=R+S+T+U.
\end{align}
By Lemma  \ref{lem:inv} and the construction,
\begin{align*}
\frac{\langle A_1 f, g\rangle}{\langle f,g\rangle}
= 
\frac{\langle f, A_1 g\rangle}{\langle f,g\rangle}, \qquad \quad
\frac{\langle A_2 f, g\rangle}{\langle f,g\rangle}
= 
\frac{\langle f, A_2 g\rangle}{\langle f,g\rangle}, \qquad \quad
\frac{\langle A_3 f, g\rangle}{\langle f,g\rangle}
= 
\frac{\langle f, A_3 g\rangle}{\langle f,g\rangle}.
\end{align*}
Evaluating these equations using Proposition  \ref{lem:dbAction}, we obtain
\begin{align} 
r+s-t-u&=R+S-T-U, \label{eq:R2} \\ 
r-s+t-u &=R-S+T-U,  \label{eq:R3}\\
r-s-t+u &= R-S-T+U. \label{eq:R4}
\end{align}
By Lemma \ref{lem:rR} and \eqref{eq:R1}--\eqref{eq:R4}, 
\begin{align*}
r=R, \qquad s= S, \qquad t=T, \qquad u=U.
\end{align*}
Therefore $f=g$. We have shown that the vectors \eqref{eq:dbasis} are mutually orthogonal. Next we prove 
\eqref{eq:sn}. For the rest of this proof, fix $N \in \mathbb N$. We will prove that
\begin{align} \label{eq:sn2}
\Vert x^{*r} y^{*s} z^{*t} w^{*u} \Vert^2 = r!s!t!u! \qquad \qquad  (r,s,t,u) \in \mathcal P_N.
\end{align}
Our proof of \eqref{eq:sn2} is by induction on $s+t+u$. First assume that $s+t+u=0$. Then $r=N$ and $s=t=u=0$.
We must show that $\Vert x^{*N} \Vert^2 = N!$.  To this end, in \eqref{eq:xs}  take the square norm of each side to obtain
\begin{align*}
\Vert x^{*N}\Vert^2
                              &= \Bigl \Vert \frac{N!}{2^N}  \sum_{(r,s,t,u) \in \mathcal P_N }  \frac{x^ry^sz^t w^u}{r!s!t!u!} \Bigr \Vert^2 \\
                              &= \frac{(N!)^2}{4^{N}} \sum_{(r,s,t,u) \in \mathcal P_N }  \frac{ \Vert  x^ry^sz^t w^u \Vert^2}{(r!s!t!u!)^2}  \\
                              &= \frac{N!}{4^{N}} \sum_{(r,s,t,u) \in \mathcal P_N  }    \frac{N!}{r!s!t!u!}  \\
                              &= N! 4^{-N} (1+1+1+1)^N \\
                              &=N!.
\end{align*}
We have proved \eqref{eq:sn2} for $s+t+u=0$. Next assume that $s+t+u\geq 1$. Then $s\geq 1$ or $t\geq 1$ or $u\geq 1$.
We take each case in turn. 
\medskip

\noindent {\bf Case} $s\geq 1$.  Define
\begin{align*}
R=r+1, \qquad S=s-1, \qquad T=t, \qquad U=u
\end{align*}
and note that $S+T+U<s+t+u$.
By Proposition \ref{lem:dbAction}(iv) and Lemma \ref{lem:inv},
\begin{align*}
R \Vert x^{*r} y^{*s} z^{*t} w^{*u} \Vert^2 &=  
\langle  x^{*r} y^{*s} z^{*t} w^{*u}, A^*_1 x^{*R}y^{*S}z^{*T}w^{*U} \rangle  \\
&=\langle A^*_1 x^{*r} y^{*s} z^{*t} w^{*u}, x^{*R}y^{*S}z^{*T}w^{*U} \rangle  \\
&= s \Vert x^{*R} y^{*S} z^{*T} w^{*U} \Vert^2.
\end{align*}
By this and induction,
\begin{align*}
\Vert x^{*r} y^{*s} z^{*t} w^{*u} \Vert^2 &= \frac{s}{R} \Vert x^{*R} y^{*S} z^{*T} w^{*U} \Vert^2 
            = \frac{s}{R} R! S!T!U!
            = r!s! t! u!.
\end{align*}
\noindent {\bf Case} $t\geq 1$.  Define
\begin{align*}
R=r+1, \qquad S=s, \qquad T=t-1, \qquad U=u
\end{align*}
and note that $S+T+U<s+t+u$.
By Proposition \ref{lem:dbAction}(v) and Lemma \ref{lem:inv},
\begin{align*}
R \Vert x^{*r} y^{*s} z^{*t} w^{*u} \Vert^2 &=  
\langle  x^{*r} y^{*s} z^{*t} w^{*u}, A^*_2 x^{*R}y^{*S}z^{*T}w^{*U} \rangle  \\
&=\langle A^*_2 x^{*r} y^{*s} z^{*t} w^{*u}, x^{*R}y^{*S}z^{*T}w^{*U} \rangle  \\
&= t \Vert x^{*R} y^{*S} z^{*T} w^{*U} \Vert^2.
\end{align*}
By this and induction,
\begin{align*}
\Vert x^{*r} y^{*s} z^{*t} w^{*u} \Vert^2 &= \frac{t}{R} \Vert x^{*R} y^{*S} z^{*T} w^{*U} \Vert^2 
            = \frac{t}{R} R! S!T!U!
            = r!s! t! u!.
\end{align*}
\noindent {\bf Case} $u\geq 1$.  Define
\begin{align*}
R=r+1, \qquad S=s, \qquad T=t, \qquad U=u-1
\end{align*}
and note that $S+T+U<s+t+u$.
By Proposition \ref{lem:dbAction}(vi) and Lemma \ref{lem:inv},
\begin{align*}
R \Vert x^{*r} y^{*s} z^{*t} w^{*u} \Vert^2 &=  
\langle  x^{*r} y^{*s} z^{*t} w^{*u}, A^*_3 x^{*R}y^{*S}z^{*T}w^{*U} \rangle  \\
&=\langle A^*_3 x^{*r} y^{*s} z^{*t} w^{*u}, x^{*R}y^{*S}z^{*T}w^{*U} \rangle  \\
&= u \Vert x^{*R} y^{*S} z^{*T} w^{*U} \Vert^2.
\end{align*}
By this and induction,
\begin{align*}
\Vert x^{*r} y^{*s} z^{*t} w^{*u} \Vert^2 &= \frac{u}{R} \Vert x^{*R} y^{*S} z^{*T} w^{*U} \Vert^2 
            = \frac{u}{R} R! S!T!U!
            = r!s! t! u!.
\end{align*}
We have proven \eqref{eq:sn2}, and \eqref{eq:sn} follows.
\end{proof}

\noindent The following result is motivated by Lemma \ref{lem:bdb}.

\begin{lemma} For $N \in \mathbb N$ the $P_N$-basis 
\begin{align*}
x^{*r} y^{*s} z^{*t} w^{*u} \qquad \qquad (r,s,t,u) \in \mathcal P_N
\end{align*}
and the $P_N$-basis 
\begin{align*}
\frac{x^{*r} y^{*s} z^{*t} w^{*u}}{r!s!t!u!} \qquad \qquad (r,s,t,u) \in \mathcal P_N
\end{align*}
are dual with respect to $\langle \,,\,\rangle$.
\end{lemma}
\begin{proof} Use Proposition \ref{lem:bilinear}. 
\end{proof}

\noindent The following result is motivated by Lemma \ref{lem:xnExpand}. Recall the automorphism $\sigma$ of $P$ from Proposition \ref{thm:aut2}.

\begin{lemma} \label{lem:motivated} For $N \in \mathbb N$,
\begin{align*} 
x^{N} = \frac{N!}{2^{ N}} \sum_{(r,s,t,u) \in \mathcal P_N  }      \frac{x^{*r}y^{*s}z^{*t} w^{*u}}{r!s!t!u!}.
\end{align*}
\end{lemma}
\begin{proof}  Apply the automorphism
$\sigma$ to each side of \eqref{eq:xs}.
\end{proof}

\begin{proposition} \label{thm:aut} For $f,g \in P$ we have
\begin{align} \label{eq:orthog}
\langle \sigma f, \sigma g\rangle = \langle f,g \rangle.
\end{align}
\end{proposition}
\begin{proof} Without loss of generality, we may assume that $f,g$ are contained in the $P$-basis \eqref{eq:Pbasis}. Comparing
Definition  \ref{def:bform} and Proposition \ref{lem:bilinear}, we routinely obtain \eqref{eq:orthog}.
\end{proof}

\noindent The following result is motivated by Lemma \ref{lem:Normal}.

\begin{lemma} \label{lem:xBinner} For $N \in \mathbb N$ and $(r,s,t,u) \in \mathcal P_N$,
\begin{align*}
\Bigl \langle x^N, x^{*r} y^{*s}z^{*t} w^{*u} \Bigr \rangle = \frac{N!}{2^N}.
\end{align*}
\end{lemma} 
\begin{proof}
Apply the automorphism $\sigma$
to everything in Lemma \ref{lem:Normal}, and evaluate the result using Proposition \ref{thm:aut}.
\end{proof}

\noindent The next result is motivated by Lemmas  \ref{lem:bilDM2},  \ref{lem:bilDM3}.

\begin{lemma} \label{lem:bilDM2s} For $f,g \in P$ we have
\begin{align*}
&\langle D_{x^*} f, g \rangle = \langle f, M_{x^*} g\rangle, \qquad \qquad
\langle D_{y^*} f, g \rangle = \langle f, M_{y^*} g\rangle, \\
&\langle D_{z^*} f, g \rangle = \langle f, M_{z^*} g\rangle, \qquad \qquad
\langle D_{w^*} f, g \rangle = \langle f, M_{w^*} g\rangle
\end{align*}
and also
\begin{align*}
&\langle M_{x^*} f, g \rangle = \langle f, D_{x^*} g\rangle, \qquad \qquad
\langle M_{y^*} f, g \rangle = \langle f, D_{y*} g\rangle,  \\
&\langle M_{z^*} f, g \rangle = \langle f, D_{z^*} g\rangle, \qquad \qquad
\langle M_{w^*} f, g \rangle = \langle f, D_{w^*} g\rangle. 
\end{align*}
\end{lemma}
\begin{proof} To prove the first equation, observe that
\begin{align*}
\langle D_{x^*} f, g \rangle &=
\langle \sigma D_x \sigma^{-1} f, g\rangle =
\langle  D_x \sigma^{-1}  f, \sigma^{-1} g \rangle \\
&=
\langle  \sigma^{-1} f ,  M_x \sigma^{-1} g \rangle 
=
\langle   f,  \sigma M_x \sigma^{-1} g \rangle =
\langle   f,  M_{x^*} g\rangle.
\end{align*}
The remaining equations are similarly proved.
\end{proof}

\noindent  Let $N \in \mathbb N$. Our next goal is to compute all the inner products between the $P_N$-basis
\begin{align*}
x^r y^s z^t w^u \qquad \qquad (r,s,t,u) \in \mathcal P_N 
\end{align*}
and the $P_N$-basis
\begin{align*}
x^{*r} y^{*s} z^{*t} w^{*u} \qquad \qquad (r,s,t,u) \in \mathcal P_N.
\end{align*}
We will express these inner products in two ways: using
a generating function, and as  hypergeometric sums.

\begin{proposition} \label{thm:cross} Let $N \in \mathbb N$. For $(r,s,t,u) \in \mathcal P_N$ and $(R,S,T,U) \in \mathcal P_N$, the inner product
\begin{align} \label{eq:cross}
\Bigl \langle x^r y^s z^t w^u, x^{*R} y^{*S} z^{*T} w^{*U} \Bigr \rangle
\end{align}
is equal to 
\begin{align*}
\frac{r!s!t!u!}{2^N}
\end{align*}
times the coefficient of $x^ry^s z^t w^u$ in 
\begin{align*}
(x+y+z+w)^R (x+y-z-w)^S (x-y+z-w)^T (x-y-z+w)^U.
\end{align*}
\end{proposition}
\begin{proof} Expand the inner product \eqref{eq:cross} using Definition \ref{def:xyzws},
and evaluate the result using Definition \ref{def:bform}.
\end{proof}

\noindent We bring in some notation. For $\alpha \in \mathbb C$ define
\begin{align*}
(\alpha)_n = \alpha (\alpha+1) \cdots (\alpha+n-1) \qquad \qquad n \in \mathbb N.
\end{align*}
\noindent The following result is a routine application of \cite[Line~(6)]{tanaka}.

\begin{proposition} \label{thm:sumCross} 
Let $N \in \mathbb N$. For $(r,s,t,u) \in \mathcal P_N$ and $(R,S,T,U) \in \mathcal P_N$, 
\begin{align*} 
&\Bigl \langle x^r y^s z^t w^u, x^{*R} y^{*S} z^{*T} w^{*U} \Bigr\rangle \\
&= \frac{N!}{2^N} \sum_{\stackrel{a,b,c,d,e,f\in \mathbb N}{\stackrel{a+b+c+d+e+f\leq N}{}}} \frac{(-s)_{a+b} (-t)_{c+d} (-u)_{e+f} (-S)_{c+e} (-T)_{a+f} (-U)_{b+d} }  {(-N)_{a+b+c+d+e+f}} \,\frac{   2^{a+b+c+d+e+f}}{a! b! c! d! e! f!}
\end{align*}
\end{proposition}
\begin{proof} This is \cite[Line~(6)]{tanaka} applied to the character algebra of the Klein four-group $\mathbb Z_2 \oplus \mathbb Z_2$. For this character algebra the first and
second eigenmatrix is $2 \Upsilon$, where $\Upsilon$ is from Definition \ref{def:T}. The matrix $\tilde \Omega$ mentioned in \cite[Line~(6)]{tanaka} is given by
\begin{align*}
\tilde \Omega = \begin{pmatrix} 0 & 2 & 2 \\
                                                   2 & 0 & 2 \\
                                                   2 & 2 &0
                        \end{pmatrix}.
\end{align*}
\end{proof}

\section{Some polynomials}
\noindent In this section, we consider some polynomials that are  motivated by Proposition \ref{thm:sumCross}.
Background information about this general topic can be found in \cite{DG, citeG, Miz, sasaki}.
\medskip

\noindent Throughout this section, fix $N \in \mathbb N$.
Let $\lambda_1, \lambda_2, \lambda_3, \mu_1, \mu_2, \mu_3$ denote mutually commuting indeterminates.
\medskip

\begin{definition} \label{def:calP} \rm Define the polynomial
\begin{align*}
\mathcal P&(\lambda_1, \lambda_2, \lambda_3 ; \mu_1, \mu_2, \mu_3) \\
&=  \sum_{\stackrel{a,b,c,d,e,f\in \mathbb N}{\stackrel{a+b+c+d+e+f\leq N}{}}} \frac{(-\lambda_1)_{a+b} (-\lambda_2)_{c+d} (-\lambda_3)_{e+f} (-\mu_1)_{c+e} (-\mu_2)_{a+f} (-\mu_3)_{b+d} }  {(-N)_{a+b+c+d+e+f}} \,\frac{   2^{a+b+c+d+e+f}}{a! b! c! d! e! f!}.
\end{align*}
\end{definition}
\begin{note}\rm The polynomial $\mathcal P(\lambda_1, \lambda_2, \lambda_3 ; \mu_1, \mu_2, \mu_3) 
$ is symmetric in $\lambda_1, \lambda_2, \lambda_3$ and symmetric in $\mu_1,\mu_2, \mu_3$. Moreover,
\begin{align}
\mathcal P(\lambda_1, \lambda_2, \lambda_3 ; \mu_1, \mu_2, \mu_3) =
\mathcal P(\mu_1, \mu_2, \mu_3; \lambda_1, \lambda_2, \lambda_3). \label{eq:calPsym}
\end{align}
\end{note}

\begin{lemma} \label{lem:Pmeaning} For $(r,s,t,u) \in \mathcal P_N$ and $(R,S,T,U) \in \mathcal P_N$,
\begin{align*}
 \Bigl \langle x^r y^s z^t w^u, x^{*R} y^{*S} z^{*T} w^{*U} \Bigr \rangle= \frac{N!}{2^N} \mathcal P(s,t,u; S,T,U).
\end{align*}
\end{lemma}
\begin{proof} By Proposition \ref{thm:sumCross} and Definition \ref{def:calP}.
\end{proof} 


\begin{proposition}\label{prop:trans} For $(r,s,t,u) \in \mathcal P_N$ we have
\begin{align}
x^r y^s z^t w^u &= \frac{N!}{2^N} \sum_{(R,S,T,U) \in \mathcal P_N} \frac{\mathcal P(s,t,u; S,T,U)}{R!S!T!U!} x^{*R} y^{*S} z^{*T} w^{*U},    \label{eq:trans1}   \\
x^{*r} y^{*s} z^{*t} w^{*u} &= \frac{N!}{2^N} \sum_{(R,S,T,U) \in \mathcal P_N} \frac{\mathcal P(s,t,u; S,T,U)}{R!S!T!U!} x^{R} y^{S} z^{T} w^{U}.      \label{eq:trans2}
\end{align}
\end{proposition}
\begin{proof}  Use linear algebra, invoking Definition \ref{def:bform} and Proposition \ref{lem:bilinear} and Lemma  \ref{lem:Pmeaning}.
\end{proof}

\noindent Next, we give an orthogonality relation.
The following result is a special case of \cite[Theorem~1.1]{tanaka}.
\begin{proposition} For $(r,s,t,u) \in \mathcal P_N$ and $(r',s',t',u') \in \mathcal P_N$,
\begin{align*}
\sum_{(R,S,T,U) \in \mathcal P_N} \frac{\mathcal P(s,t,u; S,T,U) \mathcal P(s',t',u'; S,T,U)}{R!S!T!U!} = \delta_{s,s'} \delta_{t,t'} \delta_{u,u'} \frac{4^N r!s!t!u!}{(N!)^2}.
\end{align*}
\end{proposition} 
\begin{proof} We consider the inner product
\begin{align}  \label{eq:2View}
\Bigl \langle x^r y^s z^t w^u, x^{r'} y^{s'} z^{t'} w^{u'} \Bigr \rangle 
\end{align}
from two points of view. On one hand, we evaluate \eqref{eq:2View} using Definition \ref{def:bform}.
On the other hand, we eliminate the two arguments in \eqref{eq:2View} using  \eqref{eq:trans1},
and evaluate the result using  Proposition \ref{lem:bilinear}. Comparing the two points of view,
we get the orthogonality relation.
\end{proof}

\noindent 
Next, we give some recurrence relations.

\begin{proposition}\label{prop:4TR} The following {\rm (i)--(iii)} hold for $(r,s,t,u) \in \mathcal P_N$ and $(R,S,T,U) \in \mathcal P_N$.
\begin{enumerate}
\item[\rm (i)] $(R+S-T-U) \mathcal P(s,t,u; S,T,U)$ 
 is a linear combination with the following terms and coefficients:
\begin{align*} 
\begin{tabular}[t]{c|c}
{\rm Term }& {\rm Coefficient} 
 \\
 \hline
   $ \mathcal P(s+1,t,u; S,T,U) $& $r$   \\
 $ \mathcal P(s-1,t,u; S,T,U) $   & $s$\\
  $ \mathcal P(s,t-1,u+1; S,T,U)  $  & $t$ \\
  $ \mathcal P(s,t+1,u-1; S,T,U) $& $u$
   \end{tabular}
\end{align*}
\item[\rm (ii)]   $(R-S+T-U) \mathcal P(s,t,u; S,T,U)$ 
 is a linear combination with the following terms and coefficients:
\begin{align*} 
\begin{tabular}[t]{c|c}
{\rm Term }& {\rm Coefficient} 
 \\
 \hline
   $ \mathcal P(s,t+1,u; S,T,U) $& $ r$   \\
 $ \mathcal P(s-1,t,u+1; S,T,U)$   & $s$\\
  $ \mathcal P(s,t-1,u; S,T,U)$  & $t$ \\
  $ \mathcal P(s+1,t,u-1; S,T,U)  $& $u $
   \end{tabular}
\end{align*}
\item[\rm (iii)]  $(R-S-T+U) \mathcal P(s,t,u; S,T,U)$ 
 is a linear combination with the following terms and coefficients:
\begin{align*} 
\begin{tabular}[t]{c|c}
{\rm Term }& {\rm Coefficient} 
 \\
 \hline
    $\mathcal P(s,t,u+1; S,T,U) $& $r $   \\
 $\mathcal P(s-1,t+1,u; S,T,U) $   & $s $\\
    $ \mathcal P(s+1,t-1,u; S,T,U)$  & $t$ \\
  $\mathcal P(s,t,u-1; S,T,U) $& $u $
   \end{tabular}
\end{align*}
\end{enumerate}
\end{proposition}
\begin{proof} (i) By Lemma \ref{lem:inv}  we have
\begin{align} \label{eq:1step}
\bigl \langle A_1 (x^r y^s z^t w^u), x^{*R}y^{*S}z^{*T} w^{*U} \bigr \rangle =
\bigl \langle  x^r y^s z^t w^u, A_1(x^{*R}y^{*S}z^{*T} w^{*U}) \bigr \rangle.
\end{align}
\noindent Evaluate the left-hand side of \eqref{eq:1step} using Proposition \ref{lem:ActP}(i) and Lemma \ref{lem:Pmeaning}.
Evaluate the right-hand side of \eqref{eq:1step} using Proposition \ref{lem:dbAction}(i) and Lemma \ref{lem:Pmeaning}.
The result follows.
\\
\noindent (ii), (iii) Similar to the proof of (i) above.
\end{proof}

\noindent We just gave some recurrence relations. Next, we simplify these relations using a change of variables.

\begin{definition}\label{def:PCheck} \rm Define the polynomial
\begin{align*}
&\mathcal P^\vee(\lambda_1, \lambda_2, \lambda_3; \mu_1, \mu_2, \mu_3) \\
&\quad =
\mathcal P \biggl( \lambda_1, \lambda_2, \lambda_3; \frac{N+\mu_1-\mu_2 -\mu_3}{4}, \frac{N-\mu_1+\mu_2 -\mu_3}{4}, \frac{N-\mu_1-\mu_2 +\mu_3}{4} \biggr).
\end{align*}
\end{definition}

\begin{lemma} \label{lem:PCheckSym} For $(r,s,t,u) \in \mathcal P_N$ and $(R,S,T,U) \in \mathcal P_N$,
\begin{align*}
\mathcal P(s,t,u; S,T,U) = \mathcal P^\vee (s,t,u; R+S-T-U, R-S+T-U, R-S-T+U).
\end{align*}
\end{lemma}
\begin{proof} By Lemma \ref{lem:ChangeVar} and  Definition \ref{def:PCheck}.
\end{proof}

\noindent Recall the set $\mathcal P'$ from Definition \ref{def:PnP}.

\begin{proposition} \label{prop:4TRec2} The following {\rm (i)--(iii)} hold for  $(r,s,t,u) \in \mathcal P_N$ and $(\lambda, \mu, \nu) \in \mathcal P'_N$.
\begin{enumerate}
\item[\rm (i)] $\lambda  \mathcal P^\vee (s,t,u; \lambda, \mu, \nu)$ 
 is a linear combination with the following terms and coefficients:
\begin{align*} 
\begin{tabular}[t]{c|c}
{\rm Term }& {\rm Coefficient} 
 \\
 \hline
   $ \mathcal P^\vee(s+1,t,u; \lambda, \mu,\nu ) $& $r$   \\
 $ \mathcal P^\vee(s-1,t,u; \lambda, \mu, \nu) $   & $s$\\
  $ \mathcal P^\vee(s,t-1,u+1;\lambda, \mu, \nu)  $  & $t$ \\
  $ \mathcal P^\vee (s,t+1,u-1; \lambda, \mu, \nu) $& $u$
   \end{tabular}
\end{align*}
\item[\rm (ii)]   $\mu  \mathcal P^\vee(s,t,u; \lambda, \mu,\nu)$ 
 is a linear combination with the following terms and coefficients:
\begin{align*} 
\begin{tabular}[t]{c|c}
{\rm Term }& {\rm Coefficient} 
 \\
 \hline
   $ \mathcal P^\vee(s,t+1,u; \lambda, \mu,\nu ) $& $ r$   \\
 $ \mathcal P^\vee(s-1,t,u+1; \lambda, \mu,\nu )$   & $s$\\
  $ \mathcal P^\vee(s,t-1,u; \lambda, \mu,\nu  )$  & $t$ \\
  $ \mathcal P^\vee(s+1,t,u-1; \lambda, \mu,\nu )  $& $u $
   \end{tabular}
\end{align*}
\item[\rm (iii)]  $\nu \mathcal P^\vee(s,t,u; \lambda, \mu,\nu)$ 
 is a linear combination with the following terms and coefficients:
\begin{align*} 
\begin{tabular}[t]{c|c}
{\rm Term }& {\rm Coefficient} 
 \\
 \hline
    $\mathcal P^\vee(s,t,u+1; \lambda, \mu,\nu ) $& $r $   \\
 $\mathcal P^\vee(s-1,t+1,u; \lambda, \mu,\nu ) $   & $s $\\
    $ \mathcal P^\vee(s+1,t-1,u; \lambda, \mu,\nu )$  & $t$ \\
  $\mathcal P^\vee(s,t,u-1; \lambda, \mu,\nu ) $& $u $
   \end{tabular}
\end{align*}
\end{enumerate}
\end{proposition}
\begin{proof} Evaluate Proposition \ref{prop:4TR}(i)--(iii) using Lemma  \ref{lem:PNVSPNP} and Lemma \ref{lem:PCheckSym}.
\end{proof}

\begin{proposition}
\label{prop:PVmeaning} The following hold for  $(r,s,t,u) \in \mathcal P_N$:
\begin{enumerate}
\item[\rm (i)]
$\mathcal P^\vee(s,t,u; A_1, A_2, A_3) x^N = x^r y^s z^t w^u;$
\item[\rm (ii)] $\mathcal P^\vee(s,t,u; A^*_1, A^*_2, A^*_3) x^{*N} = x^{*r} y^{*s} z^{*t} w^{*u}$.
\end{enumerate}
\end{proposition}
\begin{proof} (i) It suffices to show that for $(R,S,T,U) \in \mathcal P_N$,
\begin{align}
\Bigl \langle \mathcal P^\vee (s,t,u; A_1, A_2, A_3) x^N, x^{*R} y^{*S} z^{*T} w^{*U} \Bigr \rangle = \Bigl \langle x^r y^s z^t w^u, x^{*R} y^{*S} z^{*T} w^{*U} \Bigr \rangle.
\end{align}
Using in order Lemma  \ref{lem:inv}, Proposition  \ref{lem:dbAction}(i)--(iii), and Lemmas \ref{lem:PCheckSym},  \ref{lem:xBinner}, \ref{lem:Pmeaning}, 
\begin{align*}
&\Bigl \langle \mathcal P^\vee (s,t,u; A_1, A_2, A_3) x^N, x^{*R} y^{*S} z^{*T} w^{*U} \Bigr \rangle \\
&\quad =\Bigl \langle  x^N,  \mathcal P^\vee (s,t,u; A_1, A_2, A_3) x^{*R} y^{*S} z^{*T} w^{*U} \Bigr \rangle \\
&\quad=\Bigl \langle  x^N, x^{*R} y^{*S} z^{*T} w^{*U} \Bigr \rangle  \mathcal P^\vee (s,t,u; R+S-T-U, R-S+T-U, R-S-T+U)  \\
&\quad=\Bigl \langle  x^N, x^{*R} y^{*S} z^{*T} w^{*U} \Bigr \rangle  \mathcal P(s,t,u; S,T,U) \\
&\quad = \frac{N!}{2^N} \mathcal P(s,t,u; S,T,U) \\
&\quad =  \Bigl \langle x^r y^s z^t w^u, x^{*R} y^{*S} z^{*T} w^{*U} \Bigr \rangle.
\end{align*}
\noindent (ii) Similar to the proof of (i) above.
\end{proof}

\begin{proposition} \label{prop:PNAAAbasis}  The following hold.
\begin{enumerate}
\item[\rm (i)]  $P_N$ has a basis
\begin{align}
A_1^s A_2^t A_3^u x^N \qquad \qquad s,t,u \in \mathbb N, \qquad s+t+u \leq N.    \label{eq:AAAbasis}
\end{align}
\item[\rm (ii)]  $P_N$ has a basis
\begin{align}
A_1^{*s} A_2^{*t} A_3^{*u} x^{*N} \qquad \qquad s,t,u \in \mathbb N, \qquad s+t+u \leq N.    \label{eq:AAAbasis2}
\end{align}
\end{enumerate}
\end{proposition}
\begin{proof} (i) The number of vectors in \eqref{eq:AAAbasis} is equal to $\binom{N+3}{3}$, and this is the dimension of $P_N$.
By linear algebra, it suffices to show that $P_N$ is spanned by the vectors   \eqref{eq:AAAbasis}.
By Lemma \ref{lem:PDbasic} and Proposition \ref{prop:PVmeaning}(i), $P_N$ is spanned by the vectors   \eqref{eq:AAAbasis}.\\
\noindent (ii) Similar to the proof of (i) above.
\end{proof} 

\section{Some $\mathfrak{sl}_2(\mathbb C)$-actions on $P$}
We continue to discuss the $\mathfrak{sl}_4(\mathbb C)$-module  $P= \mathbb C\lbrack x,y,z,w\rbrack$.
Throughout this section we fix distinct $i,j \in \lbrace 1,2,3\rbrace$. In Corollary \ref{cor:6} we saw that  $A_i, A^*_j$ generate a Lie subalgebra of
$\mathfrak{sl}_4(\mathbb C)$ that is isomorphic to $\mathfrak{sl}_2(\mathbb C)$. 
The $\mathfrak{sl}_4(\mathbb C)$-module $P$ becomes an $\mathfrak{sl}_2\mathbb C)$-module, by restricting the
$\mathfrak{sl}_4(\mathbb C)$ action to $\mathfrak{sl}_2\mathbb C)$.
 In the present section we investigate the $\mathfrak{sl}_2(\mathbb C)$-module $P$.
\medskip

\noindent Let us briefly review the theory of finite-dimensional $\mathfrak{sl}_2(\mathbb C)$-modules. The following material can be found in  \cite[Sections~6, 7]{humphreys}.
Each finite-dimensional $\mathfrak{sl}_2(\mathbb C)$-module is a direct sum of irreducible $\mathfrak{sl}_2(\mathbb C)$-modules.
For $N \in \mathbb N$, up to isomorphism there exists a unique irreducible $\mathfrak{sl}_2(\mathbb C)$-module $\mathbb V_N$ of dimension $N+1$.
The $\mathfrak{sl}_2(\mathbb C)$-module $\mathbb V_N$ has a basis $\lbrace v_n \rbrace_{n=0}^N$ such that
\begin{align*}
H v_n &= (N-2n) v_n\qquad \qquad (0 \leq n \leq N), \\
F v_n&= (n+1) v_{n+1} \qquad \qquad (0 \leq n \leq N-1), \qquad F v_N=0, \\
E v_n &= (N-n+1) v_{n-1} \qquad \qquad (1 \leq n \leq N), \qquad E v_0=0.
\end{align*}
\noindent  We mention another basis for $\mathbb V_N$.  Define
\begin{align*}
u_n = v_n \binom{N}{n}^{-1} \qquad \qquad (0 \leq n \leq N).
\end{align*}
The vectors $\lbrace u_n \rbrace_{n=0}^N$ form a basis for $\mathbb V_N$, and
\begin{align*}
H u_n &= (N-2n) u_n\qquad \qquad (0 \leq n \leq N), \\
F u_n&= (N-n) u_{n+1} \qquad \qquad (0 \leq n \leq N-1), \qquad F v_N=0, \\
E u_n &= n u_{n-1} \qquad \qquad (1 \leq n \leq N), \qquad E v_0=0.
\end{align*}
\begin{definition} \label{def:compondent} \rm Let $W$ denote a finite-dimensional $\mathfrak{sl}_2(\mathbb C)$-module.
Decompose  $W$ into a direct sum of irreducible  $\mathfrak{sl}_2(\mathbb C)$-submodules:
\begin{align}
 W=W_1 + W_2 + \cdots + W_m. \label{eq:decomp}
 \end{align}
Pick $N \in \mathbb N$ and consider the irreducible $\mathfrak{sl}_2(\mathbb C)$-module $\mathbb V_N$. 
By the {\it multiplicity with which $\mathbb V_N$ appears in $W$}
 we mean the number of summands in \eqref{eq:decomp} that are isomorphic to $\mathbb V_N$ (this
 number does not depend on the choice of decomposition, see  \cite[Chapter~IV]{CR}).
We say that $\mathbb V_N$ is an {\it irreducible component} of $W$ whenever $\mathbb V_N$ appears in $W$ with nonzero  multiplicity.
\end{definition}

\begin{proposition} For $N \in \mathbb N$  the $\mathfrak{sl}_2(\mathbb C)$-module $P_N$
has the following 
irreducible components:
\begin{align} \label{eq:modulePoss}
{\mathbb V}_{N-2n} \qquad \qquad 0 \leq n \leq \lfloor N/2 \rfloor.
\end{align}
For $0 \leq n \leq \lfloor N/2 \rfloor$ the module $\mathbb V_{N-2n}$ appears in $P_N$  with multiplicity $N-2n+1$.
\end{proposition}
\begin{proof} Pick $M \in \mathbb N$ such that $\mathbb V_M$ is an irreducible component of $P_N$.
The eigenvalues of $A^*_j$ on $P_N$ are $\lbrace N-2n\rbrace_{n=0}^N$.
The integer $M$ is a nonnegative eigenvalue of $A^*_j$ on $\mathbb V_M$, so there exists  a natural number $n  \leq \lfloor N/2 \rfloor$ such that
$M=N-2n$.
By these comments, $\mathbb V_M$ is included in the list \eqref{eq:modulePoss}. For $0 \leq n \leq \lfloor N/2 \rfloor$
let ${\rm mult}({\mathbb V}_{N-2n})$ denote the multiplicity with which ${\mathbb V}_{N-2n}$ appears in $P_N$.
Counting $A^*_j$ eigenspace dimensions and using Lemma \ref{lem:ASspec}, we obtain
\begin{align*}
(n+1)(N-n+1) = \sum_{\ell=0}^n {\rm mult}({\mathbb V}_{N-2 \ell}), \qquad \qquad          0 \leq n \leq \lfloor N/2 \rfloor .
\end{align*}
From these equations we  find that for $0 \leq n \leq \lfloor N/2 \rfloor$,
\begin{align*}
{\rm mult}({\mathbb V}_{N-2n}) &= (n+1)(N-n+1)- n(N-n+2) \\
&= N-2n+1.
\end{align*}
\end{proof}

\section{The maps $L_1, L_2, L_3$ and $R_1, R_2, R_3$}
We continue to discuss the $\mathfrak{sl}_4(\mathbb C)$-module  $P= \mathbb C\lbrack x,y,z,w\rbrack$. In this section we do the following.
For $i \in \lbrace 1,2,3\rbrace$ we introduce  a ``lowering map''  $L_i \in {\rm End}(P)$ and a ``raising map''  $R_i \in {\rm End}(P)$.
For these maps we discuss the
 injectivity and surjectivity.
We also discuss how these maps are related 
 to the  $\mathfrak{sl}_4(\mathbb C)$-generators and the Hermitian form $\langle \,,\,\rangle$. 

\begin{definition} \label{def:L123} \rm Define  $L_1, L_2, L_3 \in {\rm End}(P)$ by
\begin{align*}
L_1 = D_x D_y -D_z D_w, \qquad \quad 
L_2 = D_x D_z -D_w D_y, \qquad \quad 
L_3 =  D_x D_w -D_y D_z.
\end{align*}
\end{definition}

\begin{lemma} \label{lem:LiDual} We have
\begin{align*}
L_1 = D_{x^*} D_{y^*} -D_{z^*} D_{w^*}, \qquad 
L_2 = D_{x^*} D_{z^*} -D_{w^*} D_{y^*}, \qquad  
L_3 =  D_{x^*} D_{w^*} -D_{y^*} D_{z^*}. 
\end{align*}
\end{lemma}
\begin{proof} To verify these equations, eliminate $D_{x^*}, D_{y^*}, D_{z^*}, D_{w^*}$ using Lemma \ref{lem:DDs} and evaluate the results using
Definition \ref{def:L123}.
\end{proof}

\noindent Recall the automorphism $\sigma$ of $P$ from Proposition \ref{thm:aut2}.

\begin{lemma} \label{lem:LScom} For $i \in \lbrace 1,2,3\rbrace$ the  map $L_i$ commutes with $\sigma$.
\end{lemma}
\begin{proof} By Definition \ref{def:L123} and Lemmas \ref{lem:sDs}, \ref{lem:LiDual} we obtain
\begin{align*}
\sigma L_i \sigma^{-1} = \sigma \Bigl( D_x D_y-D_z D_w \Bigr) \sigma^{-1} =  D_{x^*} D_{y^*}-D_{z^*} D_{w^*} = L_i.
\end{align*}
\end{proof}

\begin{lemma} \label{lem:2L}
The following hold:
\begin{enumerate}
\item[\rm (i)] for distinct $i,j \in \lbrace 1,2,3\rbrace$ we have $\lbrack L_i, L_j \rbrack=0$;
\item[\rm (ii)] for $i \in \lbrace 1,2,3\rbrace$ and $N \in \mathbb N$ we have 
$L_i (P_N) \subseteq P_{N-2} $.
\end{enumerate}
\end{lemma}
\begin{proof} (i) By Definition \ref{def:L123} and since $D_x, D_y, D_z, D_w$ mutually commute. \\
\noindent (ii)  By Definition \ref{def:L123} and the comments above Lemma  \ref{lem:DerDef}.    
\end{proof}

\noindent Next, we describe how $L_1, L_2, L_3$ act on the $P$-basis \eqref{eq:Pbasis}.

\begin{lemma} \label{lem:LLLact}  For $r,s,t,u \in \mathbb N$ we have
\begin{align*}
L_1(x^r y^s z^t w^u) &= rs x^{r-1}y^{s-1} z^t w^u - tu x^r y^s z^{t-1} w^{u-1}, \\
L_2(x^r y^s z^t w^u) &= rt x^{r-1}y^{s} z^{t-1} w^u - su x^r y^{s-1} z^{t} w^{u-1}, \\
L_3(x^r y^s z^t w^u) &= ru x^{r-1}y^{s} z^t w^{u-1} - st x^r y^{s-1} z^{t-1} w^{u}.
\end{align*}
\end{lemma}
\begin{proof} By \eqref{eq:one}, \eqref{eq:two} and Definition \ref{def:L123}.
\end{proof}

\noindent Next, we describe how $L_1, L_2, L_3$ act on the $P$-basis \eqref{eq:Pdualbasis}.

\begin{lemma} \label{lem:LLLDact}  For $r,s,t,u \in \mathbb N$ we have
\begin{align*}
L_1(x^{*r} y^{*s} z^{*t} w^{*u}) &= rs x^{* r-1}y^{* s-1} z^{*t} w^{*u} - tu x^{*r} y^{*s} z^{*t-1} w^{*u-1}, \\
L_2(x^{*r} y^{*s} z^{*t} w^{*u}) &= rt x^{*r-1}y^{*s} z^{*t-1} w^{*u} - su x^{*r} y^{*s-1} z^{*t} w^{*u-1}, \\
L_3(x^{*r} y^{*s} z^{*t} w^{*u}) &= ru x^{*r-1}y^{*s} z^{*t} w^{*u-1} - st x^{*r} y^{*s-1} z^{*t-1} w^{*u}.
\end{align*}
\end{lemma}
\begin{proof} By Lemmas \ref{lem:DDact}, \ref{lem:LiDual}.
\end{proof}

\begin{proposition} \label{prop:comL}
The following {\rm (i)--(iii)} hold on $P$.
\begin{enumerate}
\item[\rm (i)] $L_1$ commutes with
\begin{align*}
A_2, \qquad A_3, \qquad A^*_2, \qquad A^*_3.
\end{align*}
\item[\rm (ii)] $L_2$ commutes with
\begin{align*}
A_3, \qquad A_1, \qquad A^*_3, \qquad A^*_1.
\end{align*}
\item[\rm (iii)] $L_3$ commutes with
\begin{align*}
A_1, \qquad A_2, \qquad A^*_1, \qquad A^*_2.
\end{align*}
\end{enumerate}
\end{proposition}
\begin{proof} First we prove that $L_1, A_2$ commute.  By Lemma \ref{lem:Weyl}, Proposition \ref{lem:A6}, and Definition \ref{def:L123},
\begin{align*}
\lbrack L_1, A_2 \rbrack &= \lbrack D_x D_y - D_z D_w,   M_x D_z + M_y D_w + M_z D_x +M_w D_y   \rbrack \\
&= \lbrack D_x D_y, M_x D_z \rbrack
+
\lbrack D_x D_y, M_y D_w \rbrack
-
\lbrack D_z D_w, M_z D_x \rbrack
-
\lbrack D_z D_w, M_w D_y\rbrack \\
&= \lbrack D_x, M_x \rbrack D_y D_z +
 \lbrack D_y, M_y \rbrack D_x D_w -
  \lbrack D_z, M_z \rbrack D_x D_w -
   \lbrack D_w, M_w \rbrack D_y D_z  \\
   &= D_y D_z + D_x D_w -D_x D_w -D_y D_z \\
   &=0.
\end{align*}
Next we prove that $L_1, A^*_2$ commute.  We have 
\begin{align*}
\lbrack L_1, A^*_2 \rbrack &= \lbrack D_x D_y - D_z D_w,  M_x D_x - M_y D_y + M_z D_z - M_w D_w\rbrack \\
&= \lbrack D_x D_y, M_x D_x \rbrack
- \lbrack D_x D_y, M_y D_y \rbrack
- \lbrack D_z D_w, M_z D_z \rbrack
+ \lbrack D_z D_w, M_w D_w \rbrack \\
&= \lbrack D_x, M_x \rbrack D_x D_y -
 \lbrack D_y, M_y \rbrack D_x D_y -
  \lbrack D_z, M_z \rbrack D_z D_w +
   \lbrack D_w, M_w \rbrack D_z D_w  \\
   &= D_x D_y - D_x D_y -D_z D_w +D_z D_w \\
   &=0.
\end{align*}
The remaining assertions are similarly proven.
\end{proof}

\begin{lemma} \label{lem:LLLsurj} Each of  $L_1, L_2, L_3$ is surjective  but not injective.
\end{lemma}
\begin{proof} We first show that $L_1$ is surjective. To do this, we show that
\begin{align}
x^r y^s z^t w^u \in L_1(P) \qquad \qquad r,s,t,u \in \mathbb N. \label{eq:LsurjTOSHOW}
\end{align}
\noindent We will prove \eqref{eq:LsurjTOSHOW} by induction on $u$. First assume that $u=0$. By Lemma \ref{lem:LLLact},
\begin{align*}
x^r y^s z^t  = \frac{L_1\bigl( x^{r+1} y^{s+1} z^t \bigr)}{(r+1)(s+1)} \in L_1 (P).
\end{align*}
Next assume that $u\geq 1$. By Lemma \ref{lem:LLLact},
\begin{align*}
x^r y^s z^t w^u = \frac{L_1\bigl( x^{r+1} y^{s+1} z^t w^u  \bigr) + tux^{r+1} y^{s+1} z^{t-1} w^{u-1} }{(r+1)(s+1)}.
\end{align*}
In the above fraction, the numerator term on the right is contained in $L_1(P)$ by induction. By these comments,
$x^r y^s z^t w^u \in L_1(P)$.
We have shown that $L_1$ is surjective. Observe that $L_1$ is not injective,
because $L_1$ sends $P_0\to 0$ and $P_1 \to 0$. We have proved our claims about $L_1$. The claims about $L_2, L_3$
are similarly proven.
\end{proof}
\noindent We make an observation for later use. For $N \in \mathbb N$,
\begin{align} \label{eq:BNBNS}
\binom{N+3}{3} = \binom{N+1}{3} + (N+1)^2.
\end{align}

\begin{lemma} \label{prop:cap4} The following hold for  $i \in \lbrace 1,2,3\rbrace$ and $N \in \mathbb N$:
\begin{enumerate} 
\item[\rm (i)]  $L_i (P_N)=P_{N-2}$;
\item[\rm (ii)] ${\rm Ker}(L_i) \cap P_N$ has dimension $(N+1)^2$.
\end{enumerate}
\end{lemma}
\begin{proof}(i) By Lemma \ref{lem:2L}(ii) and Lemma \ref{lem:LLLsurj}. \\
\noindent (ii) By linear algebra along with \eqref{eq:BNBNS}
and (i) above,
\begin{align*}
&{\rm dim} \Bigl( {\rm Ker}(L_i) \cap P_N \Bigr) =
{\rm dim}\,P_N  - {\rm dim} \,P_{N-2} \\
&\qquad = \binom{N+3}{3} - \binom{N+1}{3} 
= (N+1)^2.
\end{align*}
\end{proof}

\begin{definition} \label{def:RRR} \rm Define $R_1, R_2, R_3 \in {\rm End}(P)$ by
 \begin{align*}
 R_1 =  M_x M_y  - M_z M_w,      \qquad \quad R_2  = M_x M_z - M_w M_y,  \qquad \quad R_3 = M_x M_w - M_y M_z.
 \end{align*}
\end{definition}

\noindent We clarify Definition \ref{def:RRR}. For $f \in P$,
\begin{align}
\label{eq:RRRf}
R_1(f) = (xy-zw)f, \qquad \quad R_2(f)= (xz-wy) f, \qquad \quad R_3(f) = (xw-yz) f.
\end{align}

\begin{lemma} \label{lem:sxs}  We have
\begin{align*}
xy-zw=x^* y^* - z^* w^*, \qquad  xz-wy = x^* z^* - w^* y^*, \qquad x w-yz = x^* w^*-y^* z^*.
\end{align*}
In other words, the automorphism $\sigma$ fixes each of
\begin{align*}
xy-zw, \qquad xz-wy, \qquad xw-yz.
\end{align*}
\end{lemma}
\begin{proof} Use Definition \ref{def:xyzws}.
\end{proof}

\begin{lemma} \label{def:RRRD}  We have
 \begin{align*}
 R_1 =  M_{x^*} M_{y^*}  - M_{z^*} M_{w^*},      \qquad  R_2  = M_{x^*} M_{z^*} - M_{w^*} M_{y^*},  \qquad  R_3 = M_{x^*} M_{w^*} - M_{y^*} M_{z^*}.
 \end{align*}
\end{lemma}
\begin{proof} By \eqref{eq:RRRf} and Lemma \ref{lem:sxs}.
\end{proof} 

\begin{lemma} \label{lem:DScom} For $i \in \lbrace 1,2,3\rbrace$ the  map $R_i$ commutes with $\sigma$.
\end{lemma}
\begin{proof} Similar to the proof of Lemma \ref{lem:LScom}.
\end{proof}

\begin{lemma} \label{lem:2R} The following hold:
\begin{enumerate}
\item[\rm (i)]
for distinct $i,j \in \lbrace 1,2,3\rbrace$ we have $\lbrack R_i, R_j \rbrack=0$;
\item[\rm (ii)] 
for $i \in \lbrace 1,2,3\rbrace$ and $N \in \mathbb N$ we have
$R_i (P_N) \subseteq P_{N+2} $.
\end{enumerate}
\end{lemma}
\begin{proof} (i) By Definition \ref{def:RRR} and since $M_x, M_y, M_z, M_w$ mutually commute. \\
\noindent (ii)  By Lemma \ref{lem:PDbasic} and \eqref{eq:RRRf}.    
\end{proof}

\noindent Next, we describe how $R_1, R_2, R_3$ act on the $P$-basis  \eqref{eq:Pbasis}.

\begin{lemma} \label{lem:RRRact}  For $r,s,t,u \in \mathbb N$ we have
\begin{align*}
R_1(x^r y^s z^t w^u) &=  x^{r+1}y^{s+1} z^t w^u - x^r y^s z^{t+1} w^{u+1}, \\
R_2(x^r y^s z^t w^u) &=  x^{r+1}y^{s} z^{t+1} w^u -  x^r y^{s+1} z^{t} w^{u+1}, \\
R_3(x^r y^s z^t w^u) &=   x^{r+1}y^{s} z^t w^{u+1} -  x^r y^{s+1} z^{t+1} w^{u}.
\end{align*}
\end{lemma}
\begin{proof} By  \eqref{eq:RRRf}.
\end{proof}

\noindent Next, we describe how $R_1, R_2, R_3$ act on the $P$-basis  \eqref{eq:Pdualbasis}.

\begin{lemma} \label{lem:RRRDact}  For $r,s,t,u \in \mathbb N$ we have
\begin{align*}
R_1(x^{*r} y^{*s} z^{*t} w^{*u}) &=  x^{*r+1}y^{*s+1} z^{*t} w^{*u} - x^{*r} y^{*s} z^{*t+1} w^{*u+1}, \\
R_2(x^{*r} y^{*s} z^{*t} w^{*u}) &=  x^{*r+1}y^{*s} z^{*t+1} w^{*u} -  x^{*r} y^{*s+1} z^{*t} w^{*u+1}, \\
R_3(x^{*r} y^{*s} z^{*t} w^{*u}) &=   x^{*r+1}y^{*s} z^{*t} w^{*u+1} -  x^{*r} y^{*s+1} z^{*t+1} w^{*u}.
\end{align*}
\end{lemma}
\begin{proof} By  Lemma  \ref{def:RRRD}.
\end{proof}

\begin{proposition} \label{prop:comR}
The following {\rm (i)--(iii)} hold on $P$.
\begin{enumerate}
\item[\rm (i)] $R_1$ commutes with
\begin{align*}
A_2, \qquad A_3, \qquad A^*_2, \qquad A^*_3.
\end{align*}
\item[\rm (ii)] $R_2$ commutes with
\begin{align*}
A_3, \qquad A_1, \qquad A^*_3, \qquad A^*_1.
\end{align*}
\item[\rm (iii)] $R_3$ commutes with
\begin{align*}
A_1, \qquad A_2, \qquad A^*_1, \qquad A^*_2.
\end{align*}
\end{enumerate}
\end{proposition}
\begin{proof} First we prove that $R_1, A_2$ commute. By Lemma \ref{lem:Weyl}, Proposition \ref{lem:A6},  and Definition  \ref{def:RRR},
\begin{align*}
\lbrack A_2,R_1\rbrack &= \lbrack M_z D_x + M_w D_y+ M_x D_z + M_y D_w, M_x M_y - M_z M_w \rbrack \\
&= \lbrack M_z D_x, M_x M_y \rbrack+ 
 \lbrack M_w D_y, M_x M_y \rbrack-
  \lbrack M_x D_z, M_z M_w \rbrack-
   \lbrack M_y D_w, M_z M_w \rbrack\\
&= M_y M_z \lbrack D_x, M_x \rbrack +M_x M_w\lbrack D_y, M_y\rbrack - M_x M_w\lbrack D_z, M_z \rbrack- M_y M_z\lbrack D_w, M_w \rbrack \\
&= M_y M_z +  M_x M_w -  M_x M_w - M_y M_z \\
&= 0.
\end{align*}
Next we prove that $R_1, A^*_2$ commute. We have
\begin{align*}
\lbrack A^*_2,R_1\rbrack &= \lbrack M_x D_x - M_y D_y+ M_z D_z - M_w D_w, M_x M_y - M_z M_w \rbrack \\
&= \lbrack M_x D_x, M_x M_y \rbrack
- \lbrack M_y D_y, M_x M_y\rbrack
- \lbrack M_z, D_z, M_z M_w \rbrack
+ \lbrack M_w D_w, M_z M_w \rbrack \\
&= M_x M_y \lbrack D_x, M_x \rbrack  -  M_x M_y \lbrack D_y, M_y\rbrack  - M_z M_w \lbrack D_z, M_z \rbrack  + M_z M_w \lbrack D_w, M_w \rbrack \\
&= M_x M_y -  M_x M_y -  M_z M_w + M_z M_w \\
&= 0.
\end{align*}
The remaining assertions are similarly proven.
\end{proof}

\begin{lemma} \label{lem:RRRinj} Each of $R_1, R_2, R_3$ is injective but not surjective.
\end{lemma}
\begin{proof} The maps are injective, by \eqref{eq:RRRf} and since $P$ is an integral domain. The maps are not surjective, because their images do not contain
$P_0, P_1$ by   Lemma \ref{lem:2R}(ii).
\end{proof}

\begin{lemma} \label{lem:RLbil} For  $f,g \in P$ we have
\begin{align*} 
\langle L_1 f, g \rangle = \langle f, R_1 g \rangle, \qquad \quad
\langle L_2 f, g \rangle = \langle f, R_2 g \rangle, \qquad \quad
\langle L_3 f, g \rangle = \langle f, R_3 g \rangle.
\end{align*}
\end{lemma}
\begin{proof}  We first show that $\langle L_1 f, g \rangle = \langle f, R_1 g \rangle$.  Using
Lemmas \ref{lem:Weyl}(ii), \ref{lem:bilDM2} and
 Definitions  \ref{def:L123}, \ref{def:RRR} we obtain
\begin{align*}
\langle L_1 f, g \rangle &=\langle D_x D_y f, g \rangle - \langle D_z D_w f,g \rangle \\
&=
 \langle  D_y f, M_x g \rangle - \langle D_w f,  M_zg \rangle \\
 & =\langle   f, M_y M_x g \rangle - \langle  f, M_w M_zg \rangle  \\
 &=\langle   f, M_x M_y g \rangle - \langle  f, M_z M_w g \rangle  \\
&= 
 \langle f, R_1 g \rangle.
\end{align*}
The remaining assertions are similarly shown.
\end{proof}

\begin{lemma} \label{lem:LRbil} For  $f,g \in P$ we have
\begin{align*} 
\langle R_1 f, g \rangle = \langle f, L_1 g \rangle, \qquad \quad
\langle R_2 f, g \rangle = \langle f, L_2 g \rangle, \qquad \quad
\langle R_3 f, g \rangle = \langle f, L_3 g \rangle.
\end{align*}
\end{lemma}
\begin{proof} By Definition \ref{def:Herm}(iii) and Lemma \ref{lem:RLbil}.
\end{proof}

\begin{proposition} \label{prop:RPvsKerL}
For $i \in \lbrace 1,2,3\rbrace$ and $N \in \mathbb N$ the following sum is orthogonal and direct:
\begin{align}
P_N = R_i (P_{N-2}) + {\rm Ker} (L_i) \cap P_N. \label{eq:DSPPL}
\end{align}
\end{proposition}
\begin{proof}  Referring to \eqref{eq:DSPPL}, we consider the dimensions of the three terms.
By 
 Lemma  \ref{lem:PDbasic},  the dimension of $P_N$ is equal to $\binom{N+3}{3}$. By Lemma \ref{lem:PDbasic} and since
the map $R_i$ is injective by Lemma \ref{lem:RRRinj},  the dimension of $R_i (P_{N-2})$ is equal to $\binom{N+1}{3}$.
By Lemma    \ref{prop:cap4}(ii) the dimension of  ${\rm Ker} (L_i) \cap P_N$ is equal to $(N+1)^2$.
By these comments, for \eqref{eq:DSPPL} 
the dimension of the left-hand side is equal to the sum of the dimensions of the two terms on the right-hand side.
It remains to show that these two terms are orthogonal. For $f \in P_{N-2}$ and $g \in {\rm Ker}(L_i) \cap P_N$,
\begin{align*}
\langle R_i f, g \rangle = \langle f, L_i g \rangle = \langle f, 0 \rangle = 0.
\end{align*}
The result follows.
\end{proof}

\begin{proposition} \label{prop:Psum1} For  $i \in \lbrace 1,2,3\rbrace$ and $N \in \mathbb N$ the following sum is direct:
\begin{align} \label{eq:PsumPiece}
P_N = \sum_{\ell =0}^{\lfloor N/2\rfloor} R^\ell_i \Bigl( {\rm Ker}(L_i) \cap P_{N-2\ell}\Bigr).
\end{align}
\end{proposition}
\begin{proof} By Proposition \ref{prop:RPvsKerL} and induction on $N$, together with the fact that  $R_i$ is injective.
\end{proof}
\noindent Shortly we will show that the sum \eqref{eq:PsumPiece} is orthogonal.

\begin{proposition} \label{prop:RPExp1}  For $i\in \lbrace 1,2,3\rbrace$
the following sum is direct:
\begin{align}
\label{eq:RPExp}
     P = \sum_{N \in \mathbb N} \sum_{\ell \in \mathbb N} R_i^\ell \Bigl( {\rm Ker}(L_i ) \cap P_N\Bigr).
\end{align}
\end{proposition}
\begin{proof}  In the direct sum $P=\sum_{N \in \mathbb N} P_N$ eliminate the summands using
Proposition \ref{prop:Psum1}. Evaluate the result using a change of variables $N-2 \ell \rightarrow N$.
\end{proof}
\noindent Shortly we will show that the sum \eqref{eq:RPExp} is orthogonal. We will also describe the
summands from various points of view.

\begin{lemma} \label{lem:KPinv} The following holds for $N, \ell \in \mathbb N$:
\begin{enumerate}
\item[\rm (i)]  $R_1^\ell \Bigl( {\rm Ker}(L_1) \cap P_N \Bigr) $ is invariant under each of
         \begin{align*} A_2, \qquad A_3, \qquad A^*_2, \qquad A^*_3, \qquad \sigma.
         \end{align*} 
\item[\rm (ii)] $R_2^\ell \Bigl( {\rm Ker}(L_2) \cap P_N \Bigr) $ is invariant under each of
         \begin{align*} A_3, \qquad A_1, \qquad A^*_3, \qquad A^*_1, \qquad \sigma.
         \end{align*} 
\item[\rm (iii)] $R_3^\ell \Bigl( {\rm Ker}(L_3) \cap P_N \Bigr) $ is invariant under each of
         \begin{align*} A_1, \qquad A_2, \qquad A^*_1, \qquad A^*_2, \qquad \sigma.
         \end{align*} 
\end{enumerate}
\end{lemma}
\begin{proof} (i) Each of  $L_1,  R_1$ commutes with each of 
$ A_2,  A_3,  A^*_2,  A^*_3$ by  Propositions \ref{prop:comL}, \ref{prop:comR}. 
Each of $L_1, R_1$ commutes with $\sigma $ by Lemmas \ref{lem:LScom}, \ref{lem:DScom}.
By Lemma  \ref{lem:PNAAA}, the subspace $P_N$ is invariant under $ A_2,  A_3,  A^*_2,  A^*_3$. We mentioned at the end of Section 7 that
$\sigma(P_N)=P_N$.
The result follows.
\\
\noindent (ii), (iii) Similar to the proof of (i).
\end{proof}

\section{More actions of $\mathfrak{sl}_2(\mathbb C)$  on $P$}
We continue to discuss the $\mathfrak{sl}_4(\mathbb C)$-module $P= \mathbb C\lbrack x,y,z,w\rbrack$. 
Let $i \in \lbrace 1,2,3 \rbrace$. 
 In this section, we use $L_i, R_i$ to construct an $\mathfrak{sl}_2(\mathbb C)$-action on $P$.
 We decompose $P$ into an orthogonal direct sum of irreducible $\mathfrak{sl}_2(\mathbb C)$-submodules.
 These  irreducible $\mathfrak{sl}_2(\mathbb C)$-submodules are infinite-dimensional.
 We investigate how the decomposition is related to the sum in \eqref{eq:RPExp}.

\begin{definition}\label{def:H} \rm
Define a map $\Omega \in {\rm End}(P)$ such that for $N \in \mathbb N$, the subspace $P_N$
is an eigenspace for $\Omega$ with eigenvalue $N$.
\end{definition}

\begin{lemma} \label{lem:OmegaC} The map $\Omega$ commutes with $\sigma$ and everything in $\mathfrak{sl}_4(\mathbb C)$.
\end{lemma}
\begin{proof} For $N \in \mathbb N$ the subspace $P_N$ is invariant under $\sigma$ and everything in  $\mathfrak{sl}_4(\mathbb C)$.
\end{proof}

\begin{lemma} \label{lem:HDDDD} We have
\begin{align*}
\Omega &= M_x D_x + M_y D_y +M_z D_z + M_w D_w, \\
     \Omega   &= M_{x^*} D_{x^*} + M_{y^*} D_{y^*} +M_{z^*} D_{z^*} + M_{w^*} D_{w^*}. 
\end{align*}
\end{lemma}
\begin{proof} To verify the first  equation, apply each side to a $P$-basis vector from \eqref{eq:Pbasis}.
To verify the second equation, apply each side to a $P$-basis vector from \eqref{eq:Pdualbasis}.
\end{proof}

\begin{lemma} \label{lem:Hbil} We have 
\begin{align*}
\langle \Omega f, g\rangle = \langle f, \Omega g\rangle \qquad \qquad f, g \in P.
\end{align*}
\end{lemma}
\begin{proof} By Lemma \ref{lem:PDorthog} and Definition \ref{def:H}.  
\end{proof}

\begin{proposition} \label{prop:HLR} For $i \in \lbrace 1,2,3 \rbrace$ we have
\begin{align*}
\lbrack L_i, R_i \rbrack = \Omega +2I, \qquad \quad 
\lbrack \Omega, R_i \rbrack = 2 R_i, \qquad \quad
\lbrack \Omega, L_i \rbrack = -2 L_i.
\end{align*}
\end{proposition}
\begin{proof} We check that $\lbrack L_1, R_1 \rbrack=\Omega+2I$. Using Definitions \ref{def:L123},  \ref{def:RRR} we find
\begin{align*}
\lbrack L_1, R_1 \rbrack &= \lbrack D_x D_y -D_z D_w, M_x M_y-M_z M_w \rbrack \\
& = \lbrack D_x D_y, M_x M_y\rbrack 
- \lbrack D_x D_y, M_z M_w\rbrack 
-\lbrack D_z D_w, M_x M_y\rbrack 
+ \lbrack D_z D_w, M_z M_w\rbrack.
\end{align*}
Using Lemma \ref{lem:Weyl}  we obtain
\begin{align*}
&\lbrack D_x D_y, M_x M_y\rbrack = M_x D_x + M_y D_y + I, \qquad \qquad 
 \lbrack D_x D_y, M_z M_w\rbrack = 0, \\
&\lbrack D_z D_w, M_x M_y\rbrack =0, \qquad \qquad 
 \lbrack D_z D_w, M_z M_w\rbrack = M_z D_z + M_w D_w + I.
\end{align*}
By these comments and Lemma \ref{lem:HDDDD}, we obtain  $\lbrack L_1, R_1 \rbrack=\Omega+2I$.
The remaining equations are similarly checked.
\end{proof}


\noindent Recall the Lie algebra $\mathfrak{sl}_2(\mathbb C)$ from Example \ref{ex:sl2}.

\begin{lemma} \label{lem:sl2P} For $i \in \lbrace 1,2,3\rbrace$ the vector space $P$ becomes an $\mathfrak{sl}_2(\mathbb C)$-module on which $E, F, H$ act as follows:
\begin{align*} 
\begin{tabular}[t]{c|ccc}
{\rm element $\varphi$ }&$E$ &$F$&$H$ 
 \\
 \hline 
 {\rm action of $\varphi$ on $P$} & $-L_i$ &$R_i$ &$-\Omega-2I$ 
   \end{tabular}
\end{align*}
\end{lemma}
\begin{proof} Compare the relations in Example \ref{ex:sl2} and Proposition \ref{prop:HLR}.
\end{proof}

\noindent We recall the Casimir operator for $\mathfrak{sl}_2(\mathbb C)$, see  \cite[p.~238]{carter} or \cite[p.~118]{humphreys}.
The Casimir operator $C$ is the following element in the universal enveloping algebra of $\mathfrak{sl}_2(\mathbb C)$:
\begin{align*}
C = EF+FE+H^2/2.
\end{align*} 
The operator $C$ looks as follows in terms of $A, A^*$:
\begin{align*}
C = \frac{4 A^2 + 4 A^{*2} - (A A^*-A^* A)^2}{8}.
\end{align*}
The operator $C$ generates the center of the universal enveloping algebra of $\mathfrak{sl}_2(\mathbb C)$. In particular,
$C$ commutes with each of $E,F,H, A, A^*$. On the $\mathfrak{sl}_2(\mathbb C)$-module $\mathbb V_N$,
\begin{align}
C = \frac{N(N+2)}{2} I. \label{eq:CasAct}
\end{align}

\begin{proposition} \label{prop:CCC1} On the $\mathfrak{sl}_4(\mathbb C)$-module $P$,
\begin{align*}
\frac{(\Omega+2I)^2}{2} - L_1 R_1 - R_1 L_1  &= \frac{4 A^2_2 + 4 A^{*2}_3 - (A_2 A^*_3 - A^*_3 A_2)^2 }{8} \\
 &= \frac{4 A^{*2}_2 + 4 A^{2}_3 - (A^{*}_2 A_3 - A_3 A^{*}_2)^2 }{8}.
\end{align*}
Call this common value $C_1$. Then $C_1$ commutes with
\begin{align*}
\Omega, \qquad L_1, \qquad R_1, \qquad A_2, \qquad A^*_2, \qquad A_3, \qquad A^*_3, \qquad \sigma.
\end{align*}
\end{proposition}
\begin{proof} The equations hold, because each of the three given expressions is equal to
\begin{align*}
2 D_x D_y M_z M_w + 2 D_z D_w M_x M_y - 2 M_x M_y D_x D_y - 2 M_z M_w D_z D_w + \Omega (\Omega +2 I)/2.
\end{align*}
This equality is checked using Lemma \ref{lem:Weyl}  along with 
Proposition \ref{lem:A6}  and Definitions \ref{def:L123}, \ref{def:RRR}  and Lemma \ref{lem:HDDDD}. The last assertion in the proposition statement follows from 
our comments about the Casimir operator and the fact that $\sigma$ commutes with each of $\Omega, L_1, R_1$.
\end{proof}

\begin{proposition} \label{prop:CCC2} On the $\mathfrak{sl}_4(\mathbb C)$-module $P$,
\begin{align*}
\frac{(\Omega+2I)^2}{2} - L_2 R_2 - R_2 L_2  &= \frac{4 A^2_3 + 4 A^{*2}_1 - (A_3 A^*_1 - A^*_1 A_3)^2 }{8} \\
 &= \frac{4 A^{*2}_3 + 4 A^{2}_1 - (A^{*}_3A_1 - A_1 A^{*}_3)^2 }{8}.
\end{align*}
Call this common value $C_2$. Then $C_2$ commutes with
\begin{align*}
\Omega, \qquad L_2, \qquad R_2, \qquad A_3, \qquad A^*_3, \qquad A_1, \qquad A^*_1, \qquad \sigma.
\end{align*}
\end{proposition}
\begin{proof} Similar to the proof of Proposition \ref{prop:CCC1}.
\end{proof}

\begin{proposition} \label{prop:CCC3} On the $\mathfrak{sl}_4(\mathbb C)$-module $P$,
\begin{align*}
\frac{(\Omega+2I)^2}{2} - L_3 R_3 - R_3L_3 &= \frac{4 A^2_1 + 4 A^{*2}_2 - (A_1 A^*_2 - A^*_2 A_1)^2 }{8} \\
 &= \frac{4 A^{*2}_1 + 4 A^{2}_2 - (A^{*}_1A_2 - A_2 A^{*}_1)^2 }{8}.
\end{align*}
Call this common value $C_3$. Then $C_3$ commutes with
\begin{align*}
\Omega, \qquad L_3, \qquad R_3, \qquad A_1, \qquad A^*_1, \qquad A_2, \qquad A^*_2, \qquad \sigma.
\end{align*}
\end{proposition}
\begin{proof} Similar to the proof of Proposition \ref{prop:CCC1}.
\end{proof}

\begin{lemma}\label{prop:Cinv} For  $f,g \in P$ we have
\begin{align*}
\langle C_1 f, g \rangle = \langle f, C_1 g\rangle, \qquad  \quad
\langle C_2 f, g \rangle = \langle f, C_2 g\rangle, \qquad  \quad
\langle C_3 f, g \rangle = \langle f, C_3 g\rangle.
\end{align*}
\end{lemma}
\begin{proof}  The result for $C_1$ follows from Lemma \ref{lem:inv} and the definition of $C_1$ in Proposition \ref{prop:CCC1}.
The results for $C_2, C_3$ are similarly obtained.
\end{proof}

\begin{lemma}\label{lem:PNCinv} For $N \in \mathbb N$,
\begin{align*}
C_1 (P_N) \subseteq P_N, \qquad \quad 
C_2 (P_N) \subseteq P_N, \qquad \quad 
C_3 (P_N) \subseteq P_N.
\end{align*}
\end{lemma} 
\begin{proof} Since each of $C_1, C_2, C_3$ commutes with $\Omega$.
\end{proof}

\begin{proposition}  The maps $C_1, C_2, C_3$ act as follows on the $P$-basis  \eqref{eq:Pbasis}.
For $N \in \mathbb N$ and $(r,s,t,u) \in {\mathcal P}_N$,
\begin{enumerate}
\item[\rm (i)] the vector
\begin{align*}
 C_1( x^ry^s z^t w^u)
 \end{align*}
 is a linear combination with the following terms and coefficients:
\begin{align*} 
\begin{tabular}[t]{c|c}
{\rm Term }& {\rm Coefficient} 
 \\
 \hline
   $ x^{r-1} y^{s-1} z^{t+1} w^{u+1} $& $2rs$   \\
 $ x^{r} y^{s} z^{t} w^{u} $   & $N(N+2)/2-2rs-2tu$\\
  $ x^{r+1} y^{s+1} z^{t-1} w^{u-1}  $  & $2tu$
   \end{tabular}
\end{align*}
\item[\rm (ii)] 
the vector
\begin{align*}
C_2(x^r y^s z^t w^u)     
\end{align*}
 is a linear combination with the following terms and coefficients:
\begin{align*} 
\begin{tabular}[t]{c|c}
{\rm Term }& {\rm Coefficient} 
 \\
 \hline
   $ x^{r-1} y^{s+1} z^{t-1} w^{u+1} $& $ 2rt$   \\
 $ x^{r} y^{s} z^{t} w^{u} $   & $N(N+2)/2-2rt-2su$\\
  $ x^{r+1} y^{s-1} z^{t+1} w^{u-1} $  & $2su$ 
   \end{tabular}
\end{align*}
\item[\rm (iii)] 
the vector
\begin{align*}
C_3( x^{r} y^{s} z^{t} w^{u} )
\end{align*}
 is a linear combination with the following terms and coefficients:
\begin{align*} 
\begin{tabular}[t]{c|c}
{\rm Term }& {\rm Coefficient} 
 \\
 \hline
    $ x^{r-1} y^{s+1} z^{t+1} w^{u-1} $& $2ru $   \\
 $x^{r} y^{s} z^{t} w^{u} $   & $N(N+2)/2-2ru-2st $\\
    $ x^{r+1} y^{s-1} z^{t-1} w^{u+1} $  & $2st$
   \end{tabular}
\end{align*}
\end{enumerate}
To get the action of $C_1, C_2, C_3$ on the $P$-basis  \eqref{eq:Pdualbasis}, replace $x,y,z,w$ by $x^*, y^*, z^*, w^*$ respectively in {\rm (i)--(iii)} above.
\end{proposition}
\begin{proof} (i) By the definition of $C_1$ in Proposition \ref{prop:CCC1}, along with Proposition \ref{lem:ActP}  or Lemma \ref{lem:LLLact}, Lemma \ref{lem:RRRact}, 
Definition \ref{def:H}. \\
\noindent (ii), (iii) Similar to the proof of (i) above. \\
\noindent The last assertion of the proposition statement holds because $\sigma$ commutes with each of $C_1, C_2, C_3$.
\end{proof}

\noindent  In Section 11 we discussed the finite-dimensional irreducible $\mathfrak{sl}_2(\mathbb C)$-modules.
We now discuss a more general type of  $\mathfrak{sl}_2(\mathbb C)$-module, said to be highest-weight.
The following discussion is distilled from \cite[Chapter~VI]{humphreys}.
Let $W$ denote an $\mathfrak{sl}_2(\mathbb C)$-module. A {\it highest-weight vector} in $W$ is a nonzero $v \in W$ such that $v$ is an eigenvector
for $H$ and $Ev=0$. The $\mathfrak{sl}_2(\mathbb C)$-module $W$ is said to be {\it highest-weight} whenever $W$ 
contains a highest-weight vector $v$ such that $W$
is spanned by $\lbrace F^\ell v \rbrace_{\ell \in \mathbb N}$. Assume that $W$ is highest-weight. Then the previously mentioned 
vector $v$ is unique up to multiplication by a nonzero scalar in $\mathbb C$.
By the {\it highest-weight of $W$}, we mean the $H$-eigenvalue associated with $v$. A pair of irreducible highest-weight
$\mathfrak{sl}_2(\mathbb C)$-modules are isomorphic if and only if  they have the same highest-weight \cite[p.~109]{humphreys}.

\begin{lemma} \label{lem:Wprelim} Pick $i \in \lbrace 1,2,3\rbrace$ and consider the corresponding $\mathfrak{sl}_2(\mathbb C)$-module structure on $P$ from Lemma \ref{lem:sl2P}.
Let $N \in \mathbb N$. Pick $0 \not= v \in {\rm Ker}(L_i) \cap P_N$ and define
$v_\ell = R^\ell_i v/{\ell !} $ for $\ell \in \mathbb N$.
Then:
\begin{align*}
    & v_\ell \in P_{N+2\ell}, \qquad \quad \Omega v_\ell = (N+2\ell) v_\ell \qquad \qquad \ell \in \mathbb N,  \\
    & R_i v_\ell = (\ell+1) v_{\ell+1} \qquad \qquad \ell \in \mathbb N,  \\
    &  L_i v_\ell = (N+\ell+1) v_{\ell-1} \qquad \quad \ell \geq 1, \qquad \quad L_i v_0=0. 
      \end{align*}
     The vectors $\lbrace v_\ell \rbrace_{\ell \in \mathbb N}$ form a basis for an $\mathfrak{sl}_2(\mathbb C)$-submodule of $P$. Denote this $\mathfrak{sl}_2(\mathbb C)$-submodule 
      by $W$.
      Then $W$ is irreducible and highest-weight, with highest-weight $-N-2$. On $W$,
      \begin{align*}
     C_i =  \frac{N(N+2)}{2} I.    
     \end{align*}
         \end{lemma}
      \begin{proof} Routine application of the $\mathfrak{sl}_2(\mathbb C)$ representation theory, see \cite[Chapter~VI]{humphreys}.
      \end{proof}

\noindent  We refer to the  $\mathfrak{sl}_2(\mathbb C)$-module $P$ in Lemma \ref{lem:Wprelim}. In that lemma, we 
constructed some  irreducible  $\mathfrak{sl}_2(\mathbb C)$-submodules of $P$. In the next result, we show that every
irreducible  $\mathfrak{sl}_2(\mathbb C)$-submodule of $P$ comes from the construction.

\begin{lemma} \label{lem:PdsW} Pick $i \in \lbrace 1,2,3\rbrace$ and consider the corresponding $\mathfrak{sl}_2(\mathbb C)$-module structure on $P$ from Lemma \ref{lem:sl2P}.
Then the following hold.
\begin{enumerate}
\item[\rm (i)] Each irreducible $\mathfrak{sl}_2(\mathbb C)$-submodule of $P$ is highest-weight.
\item[\rm (ii)] Let $W$ denote an irreducible $\mathfrak{sl}_2(\mathbb C)$-submodule of $P$, with highest-weight $\zeta$. Then $\zeta$ is an integer at most  $-2$.
Write $\zeta=-N-2$. Then there exists 
$0 \not= v \in {\rm Ker}(L_i) \cap P_N$ such that $\lbrace R^\ell_i v/{\ell !}\rbrace_{\ell \in \mathbb N}$ is a basis for $W$.
      \end{enumerate}
      \end{lemma}
      \begin{proof}     By Definition \ref{def:H},  $\Omega$ is diagonalizable on $P$ with eigenvalues $0,1,2,\ldots$
       Let $W$ denote an irreducible   $\mathfrak{sl}_2(\mathbb C)$-submodule of $P$.
   The action of $\Omega$
     on $W$ is diagonalizable with each eigenvalue contained in $\mathbb N$.  By this and Lemma \ref{lem:sl2P}, $H$ is diagonalizable on $W$
     with each eigenvalue  an integer at most $-2$. Let $\zeta$ denote the maximal eigenvalue of $H$ on $W$. Let $0 \not=v \in W$
     denote an eigenvector for $H$ with eigenvalue $\zeta $. 
     We have $Ev=0$; otherwise $Ev \in W$ is an eigenvector for $H$ with eigenvalue $\zeta+2$, contradicting the
     maximality of $\zeta$.  By these comments, $v$ is a highest-weight vector for $W$. 
     Write $\zeta=-N-2$ and note that $v \in {\rm Ker}(L_i) \cap P_N$. Using $v$ we define vectors $\lbrace v_\ell\rbrace_{\ell \in \mathbb N}$ as in Lemma \ref{lem:Wprelim}.
   By Lemma \ref{lem:Wprelim} the vectors $\lbrace v_\ell \rbrace_{\ell \in \mathbb N}$ form a basis for
     a highest-weight $\mathfrak{sl}_2(\mathbb C)$-submodule of $W$, which must equal $W$ by the irreducibility of $W$.   
     \end{proof}
    
   \begin{lemma} \label{lem:Breakdown} 
  Pick $i \in \lbrace 1,2,3\rbrace$ and consider the corresponding $\mathfrak{sl}_2(\mathbb C)$-module structure on $P$ from Lemma \ref{lem:sl2P}.
  For $N \in \mathbb N$ the following are the same:
  \begin{enumerate}
  \item[\rm (i)] the span of the irreducible $\mathfrak{sl}_2(\mathbb C)$-submodules of $P$ that have highest-weight $-N-2$;
  \item[\rm (ii)]  $ \sum_{\ell \in \mathbb N} R_i^\ell \bigl( {\rm Ker}(L_i ) \cap P_N\bigr)$;
  \item[\rm (iii)] the eigenspace of $C_i$ on $P$ with eigenvalue $N(N+2)/2$.
  \end{enumerate}
  \end{lemma}
  \begin{proof} Let $P\lbrack N \rbrack, P'\lbrack N \rbrack, P''\lbrack N \rbrack$ denote the subspaces in (i), (ii), (iii) respectively.
By Lemmas \ref{lem:Wprelim}, \ref{lem:PdsW} we have $P\lbrack N \rbrack = P' \lbrack N \rbrack$ and on this common value 
$C_i=N(N+2)/2\, I$. By this and Proposition \ref{prop:RPExp1},  we obtain $P\lbrack N \rbrack= P'\lbrack N \rbrack= P''\lbrack N \rbrack$.
  \end{proof}
      
 \begin{corollary} \label{lem:Breakdown2} 
  Pick $i \in \lbrace 1,2,3\rbrace$ and consider the corresponding $\mathfrak{sl}_2(\mathbb C)$-module structure on $P$ from Lemma \ref{lem:sl2P}.
  For $N \in \mathbb N$ the subspace 
   $ \sum_{\ell \in \mathbb N} R_i^\ell \bigl( {\rm Ker}(L_i ) \cap P_N\bigr)$ is an $\mathfrak{sl}_2(\mathbb C)$-submodule of $P$.
  \end{corollary}
      \begin{proof} By Lemma \ref{lem:Breakdown}(i),(ii).
      \end{proof}

    \begin{lemma} \label{lem:CEV} For $i \in \lbrace 1,2,3\rbrace$ the following hold:
    \begin{enumerate}
    \item[\rm (i)]  the map $C_i$ is diagonalizable on $P$;
    \item[\rm (ii)] the eigenvalues of $C_i$ on $P$ are 
    \begin{align*}
    N(N+2)/2 \qquad \qquad N \in \mathbb N;
    \end{align*}
    \item[\rm (iii)] the eigenspaces of $C_i$ on $P$ are mutually orthogonal.
    \end{enumerate}
    \end{lemma}      
    \begin{proof} (i), (ii)  By Proposition   \ref{prop:RPExp1} and  Lemma \ref{lem:Breakdown}(ii),(iii). \\
    \noindent (iii) By  Lemma \ref{prop:Cinv} and  (ii) above.
    \end{proof}

   \noindent Recall the direct sum decompositions   \eqref{eq:PsumPiece} and \eqref{eq:RPExp}.
      \begin{proposition}  \label{prop:sl2sl2MOrth} The following hold for $ i \in \lbrace 1,2,3\rbrace$:
     \begin{enumerate}
     \item[\rm (i)] the summands in \eqref{eq:RPExp} are mutually orthogonal;
  \item[\rm (ii)] for $N \in \mathbb N$ the summands in \eqref{eq:PsumPiece} are mutually orthogonal.
      \end{enumerate}
      \end{proposition}
      \begin{proof}  (i) By Lemma \ref{lem:PDorthog}  along with Lemmas \ref{lem:Breakdown}(ii),(iii) and \ref{lem:CEV}(iii). \\
      \noindent (ii) By (i) above.
      \end{proof}

\noindent  Pick $i \in \lbrace 1,2,3\rbrace$ and consider the corresponding $\mathfrak{sl}_2(\mathbb C)$-module structure on $P$ from Lemma \ref{lem:sl2P}.
Our next general goal is to show that the $\mathfrak{sl}_2(\mathbb C)$-module $P$ is an orthogonal direct sum of irreducible $\mathfrak{sl}_2(\mathbb C)$-submodules.

\begin{lemma} \label{lem:INL}  For  $i \in \lbrace 1,2,3\rbrace$ and $N, \ell \in \mathbb N$ the
following holds on  $R_i^\ell \bigl( {\rm Ker}(L_i ) \cap P_N\bigr)$:        
\begin{align*}
L_i R_i = (\ell+1)(N+\ell+2) I, \qquad \qquad R_i L_i = \ell (N+\ell +1) I.
\end{align*}
\end{lemma}
\begin{proof} Let $0 \not=v \in {\rm Ker}(L_i) \cap P_N$ and define $v_\ell = R^\ell_i v/{\ell !}$. By Lemma \ref{lem:Wprelim} we obtain
\begin{align*}
L_i R_i v_\ell = (\ell+1)(N+\ell+2) v_\ell, \qquad \qquad R_i L_i v_\ell= \ell (N+\ell +1) v_\ell.
\end{align*}
The result follows.
\end{proof}

\begin{lemma} \label{lem:INP}  Let  $i \in \lbrace 1,2,3\rbrace$ and $N, \ell \in \mathbb N$.  Then for
$f,g \in R_i^\ell \bigl( {\rm Ker}(L_i ) \cap P_N\bigr)$ we have       
\begin{align*}
\langle R_i f,  R_i g \rangle = (\ell+1)(N+\ell+2) \langle f, g\rangle, \qquad \qquad \langle L_i f,  L_i g \rangle = \ell (N+\ell +1) \langle f, g \rangle.
\end{align*}
\end{lemma}
\begin{proof} By Lemma \ref{lem:LRbil} and Lemma \ref{lem:INL}, 
\begin{align*}
\langle R_i f,  R_i g \rangle &= \langle f, L_i R_i g \rangle = (\ell+1)(N+\ell+2) \langle f,g \rangle .
\end{align*}
By Lemma \ref{lem:RLbil} and Lemma \ref{lem:INL}, 
\begin{align*}
\langle L_i f,  L_i g \rangle &= \langle f, R_i L_i  g \rangle = \ell (N+\ell+1) \langle f, g \rangle .
\end{align*}
\end{proof}

\begin{corollary}\label{cor:Orth} Let  $i \in \lbrace 1,2,3\rbrace$ and $N \in \mathbb N$. Pick
$u,v \in {\rm Ker}(L_i ) \cap P_N$ such that $\langle u,v\rangle =0$. Then
$\langle R_i^\ell u, R_i^\ell v\rangle = 0$ for $\ell \in \mathbb N$.
\end{corollary}
\begin{proof} By Lemma \ref{lem:INP} and induction on $\ell$.
\end{proof}

\begin{proposition} \label{prop:ODSC}   Pick $i \in \lbrace 1,2,3\rbrace$ and consider the corresponding $\mathfrak{sl}_2(\mathbb C)$-module structure on $P$ from Lemma \ref{lem:sl2P}.
  Then for $N \in \mathbb N$ the $\mathfrak{sl}_2(\mathbb C)$-submodule
   $ \sum_{\ell \in \mathbb N} R_i^\ell \bigl( {\rm Ker}(L_i ) \cap P_N\bigr)$ is an orthogonal direct sum of irreducible $\mathfrak{sl}_2(\mathbb C)$-submodules.
\end{proposition}
\begin{proof} By Lemma \ref{prop:cap4}(ii),  the subspace  ${\rm Ker}(L_i) \cap P_N$ has dimension $(N+1)^2$. Abbreviate $d=(N+1)^2$.
 Let $\lbrace v_j \rbrace_{j=1}^d$ denote an orthogonal basis for ${\rm Ker}(L_i) \cap P_N$.
For $1 \leq j \leq d$ we use $v_j$ and Lemma \ref{lem:Wprelim}  to construct an irreducible $\mathfrak{sl}_2(\mathbb C)$-submodule $W_j$ of $P$.
By Lemma \ref{lem:Wprelim} and the construction,
\begin{align}
\sum_{\ell \in \mathbb N} R_i^\ell \bigl( {\rm Ker}(L_i ) \cap P_N\bigr) = \sum_{j=1}^d W_j.       \label{eq:WWW}
\end{align}
By Corollary \ref{cor:Orth} the subspaces $\lbrace W_j \rbrace_{j=1}^d $ are mutually orthogonal.
The result follows.
\end{proof}

\begin{proposition} \label{prop:ODSP}  Pick $i \in \lbrace 1,2,3\rbrace$ and consider the corresponding $\mathfrak{sl}_2(\mathbb C)$-module structure on $P$ from Lemma \ref{lem:sl2P}.
Then the $\mathfrak{sl}_2(\mathbb C)$-module  $P$ is an orthogonal direct sum of irreducible $\mathfrak{sl}_2(\mathbb C)$-submodules.
\end{proposition}
\begin{proof} By Propositions \ref{prop:RPExp1},     \ref{prop:sl2sl2MOrth}(i),  \ref{prop:ODSC}.
\end{proof}
      \section{The  Lie algebra $\mathfrak{sl}_2(\mathbb C) \oplus \mathfrak{sl}_2(\mathbb C)$, revisited}
      
      We continue to discuss the $\mathfrak{sl}_4(\mathbb C)$-module $P= \mathbb C\lbrack x,y,z,w\rbrack$. 
Let $i \in \lbrace 1,2,3 \rbrace$.  By Propositions \ref{prop:RPExp1}, \ref{prop:sl2sl2MOrth}(i) we have an orthogonal direct sum
  \begin{align} \label{eq:SSKsum}
     P = \sum_{N \in \mathbb N} \sum_{\ell \in \mathbb N} R_i^\ell \Bigl( {\rm Ker}(L_i ) \cap P_N\Bigr).
     \end{align}
 In this section, we show how each summand in \eqref{eq:SSKsum} becomes an irreducible module for the Lie algebra $\mathfrak{sl}_2(\mathbb C) \oplus \mathfrak{sl}_2(\mathbb C)$.
\medskip

    \noindent We mentioned  $\mathfrak{sl}_2(\mathbb C) \oplus \mathfrak{sl}_2(\mathbb C)$ in Corollary \ref{cor:LL} and  Definition \ref{def:sl2sl2sub}.
    We now have some more comments about $\mathfrak{sl}_2(\mathbb C) \oplus \mathfrak{sl}_2(\mathbb C)$.
    For $N,M\in \mathbb N$ the vector space ${\mathbb V}_N \otimes {\mathbb V}_M$ is an
    $\mathfrak{sl}_2(\mathbb C) \oplus \mathfrak{sl}_2(\mathbb C)$ module with the following action. Let $a,b \in \mathfrak{sl}_2(\mathbb C)$
   and consider  the element $(a,b)$ in  $\mathfrak{sl}_2(\mathbb C) \oplus \mathfrak{sl}_2(\mathbb C)$.
    For $u \in {\mathbb V}_N$
    and $v \in {\mathbb V}_M$, the element $(a,b)$  sends
    \begin{align*}
    u \otimes v \mapsto (au) \otimes v + u \otimes (bv).
    \end{align*}
   It is routine to check (or see \cite[Section~3.8]{pretel}) that up to isomorphism the finite-dimensional irreducible modules for $\mathfrak{sl}_2(\mathbb C) \oplus \mathfrak{sl}_2(\mathbb C)$
   are
   \begin{align} \label{eq:VNVM}
  {\mathbb  V}_N \otimes {\mathbb V}_M \qquad \qquad N, M \in \mathbb N.
   \end{align}

     \begin{proposition} \label{prop:sl2sl2} The following hold for $ i \in \lbrace 1,2,3\rbrace$.
     \begin{enumerate}
              \item[\rm (i)]  Each summand in \eqref{eq:SSKsum} is an irreducible submodule for the $i$th Lie subalgebra of $\mathfrak{sl}_4(\mathbb C)$ isomorphic to $\mathfrak{sl}_2(\mathbb C)\oplus \mathfrak{sl}_2(\mathbb C)$.
     \item[\rm (ii)] For $N \in \mathbb N$ and $\ell \in \mathbb N$ the corresponding summand in \eqref{eq:SSKsum} is isomorphic to ${\mathbb V}_N \otimes {\mathbb V}_N$
     as a module for the $i$th Lie subalgebra of $\mathfrak{sl}_4(\mathbb C)$ isomorphic to $\mathfrak{sl}_2(\mathbb C)\oplus \mathfrak{sl}_2(\mathbb C)$.
         \end{enumerate}
     \end{proposition}  
\begin{proof} 
 Let $j,k$ denote the elements in $\lbrace 1,2,3\rbrace\backslash \lbrace i \rbrace$. Let $\mathcal L$ denote the Lie
subalgebra of $\mathfrak{sl}_4(\mathbb C)$ generated by $A_j, A_k, A^*_j, A^*_k$. By Corollary  \ref{cor:LL}  the Lie algebra $\mathcal L$ is 
isomorphic to  $\mathfrak{sl}_2(\mathbb C) \oplus \mathfrak{sl}_2(\mathbb C)$. By Definition \ref{def:sl2sl2sub}, $\mathcal L$ is  the $i$th Lie subalgebra of $\mathfrak{sl}_4(\mathbb C)$ isomorphic to $\mathfrak{sl}_2(\mathbb C)\oplus \mathfrak{sl}_2(\mathbb C)$.
For notational convenience, throughout this proof we identify the Lie algebras $\mathcal L$ and $\mathfrak{sl}_2(\mathbb C) \oplus \mathfrak{sl}_2(\mathbb C)$ via the isomorphism in
Lemma \ref{lem:sl2sl2}.
The $\mathfrak{sl}_4(\mathbb C)$-module $P$ becomes an $\mathcal L$-module by restricting the $\mathfrak{sl}_4(\mathbb C)$ action to $\mathcal L$.
Let $N,\ell \in \mathbb N$ be given, and let $W$ denote the corresponding summand in \eqref{eq:SSKsum}.
We show that $W$ is an irreducible $\mathcal L$-submodule of $P$ that is isomorphic to ${\mathbb V}_N \otimes {\mathbb V}_N$.
The subspace $W$ is invariant under
$A_j, A_k, A^*_j, A^*_k$ by Lemma \ref{lem:KPinv},     
 so $W$ is an $\mathcal L$-submodule of $P$. Let $\mathcal W$
denote an irreducible $\mathcal L$-submodule of $W$. 
By the discussion around  \eqref{eq:VNVM},
the $\mathcal L$-module $\mathcal W$ is isomorphic to ${\mathbb V}_r \otimes {\mathbb V}_s$ for some $r,s\in \mathbb N$.
Viewing $\mathcal W$ as a module for the copy of $\mathfrak{sl}_2(\mathbb C)$ generated by $A_j, A^*_k$, we find that $\mathcal W$ is a direct sum of $s+1$ irreducible
 $\mathfrak{sl}_2(\mathbb C)$-submodules, each isomorphic to  ${\mathbb V}_r$.
Viewing $\mathcal W$ as a module for the copy of $\mathfrak{sl}_2(\mathbb C)$  generated by $A^*_j, A_k$, we find that $\mathcal W$ is a direct sum of $r+1$ irreducible
 $\mathfrak{sl}_2(\mathbb C)$-submodules, each isomorphic to 
 ${\mathbb V}_s$.
For both of these $\mathfrak{sl}_2(\mathbb C)$-actions, the Casimir operator acts as $C_i$ by Propositions \ref{prop:CCC1}--\ref{prop:CCC3}, and by Lemma  \ref{lem:Breakdown}(ii),(iii)
 this operator acts on $\mathcal W$ as $N(N+2)/2$ times the identity. By these comments and the discussion around  \eqref{eq:CasAct}, we obtain
 $r=s=N$. Thus the $\mathcal L$-module
 $\mathcal W$ is isomorphic to ${\mathbb V}_N \otimes {\mathbb V}_N$.
The dimension of ${\mathbb V}_N \otimes {\mathbb V}_N$ is $(N+1)^2$, and this is the dimension of $W$ in view of Lemma \ref{prop:cap4}(ii) and the injectivity of $R_i$. By these comments, $\mathcal W=W$.
We have shown that $W$ is an irreducible $\mathcal L$-submodule of $P$ that is isomorphic to ${\mathbb V}_N \otimes {\mathbb V}_N$.
\end{proof}

\noindent Let  $i \in \lbrace 1,2,3 \rbrace$ and $N \in \mathbb N$.
By Propositions \ref{prop:Psum1},  \ref{prop:sl2sl2MOrth}(ii) we have an orthogonal direct sum
\begin{align} \label{eq:odsREMIND}
P_N = \sum_{\ell =0}^{\lfloor N/2\rfloor} R^\ell_i \Bigl( {\rm Ker}(L_i) \cap P_{N-2\ell}\Bigr). 
\end{align}

\begin{proposition} \label{lem:comment} The following hold for $i \in \lbrace 1,2,3\rbrace$ and $N \in \mathbb N$.
  \begin{enumerate}
          \item[\rm (i)]  Each summand in  \eqref{eq:odsREMIND}  is an irreducible submodule for the $i$th Lie subalgebra of $\mathfrak{sl}_4(\mathbb C)$ isomorphic to $\mathfrak{sl}_2(\mathbb C)\oplus \mathfrak{sl}_2(\mathbb C)$.
          \item[\rm (ii)] For $0 \leq \ell \leq \lfloor N/2\rfloor$ the $\ell$-summand in    \eqref{eq:odsREMIND}   has dimension $(N-2\ell+1)^2$.
     \item[\rm (iii)] For $0 \leq \ell \leq \lfloor N/2\rfloor$ the $\ell$-summand in    \eqref{eq:odsREMIND}     is isomorphic to ${\mathbb V}_{N-2\ell} \otimes {\mathbb V}_{N-2\ell}$
     as a module for the $i$th Lie subalgebra of $\mathfrak{sl}_4(\mathbb C)$ isomorphic to $\mathfrak{sl}_2(\mathbb C)\oplus \mathfrak{sl}_2(\mathbb C)$.
     \item[\rm (iv)] For $0 \leq \ell \leq \lfloor N/2 \rfloor$ the $\ell$-summand in     \eqref{eq:odsREMIND}    is an eigenspace for the action of $C_i$ on $P_N$; the  eigenvalue is
\begin{align*}
\frac{(N-2\ell)(N-2\ell+2)}{2}.
\end{align*}
         \end{enumerate}
\end{proposition}
\begin{proof} (i)--(iii) By Proposition \ref{prop:sl2sl2}. \\
\noindent (iv) By Lemma \ref{lem:Breakdown}(ii),(iii).
\end{proof}

\noindent Let $i \in \lbrace 1,2,3\rbrace$ and $N \in \mathbb N$. We finish this section with a summary of how
$C_i$ acts on $P_N$.

\begin{corollary}\label{prop:PNCi} The following hold for $i \in \lbrace 1,2,3 \rbrace$ and $N \in \mathbb N$.
\begin{enumerate}
\item[\rm (i)] The action of $C_i$ on $P_N$ is diagonalizable.
\item[\rm (ii)] For the action of $C_i$ on $P_N$ the eigenvalues are 
\begin{align*}
\frac{(N-2\ell)(N-2\ell+2)}{2}, \qquad \qquad 0 \leq \ell \leq \lfloor N/2 \rfloor.
\end{align*}
\item[\rm (iii)] For $0 \leq \ell \leq \lfloor N/2 \rfloor$ the $(N-2\ell)(N-2\ell+2)/2$-eigenspace for $C_i$ on $P_N$ has dimension $(N-2\ell+1)^2$.
\item[\rm (iv)]   the eigenspaces of $C_i$ on $P_N$ are mutually orthogonal.
\end{enumerate}
\end{corollary}
\begin{proof} By  Propositions \ref{prop:sl2sl2MOrth}(ii) and  \ref{lem:comment}(ii),(iv).
\end{proof}

\section{Some  bases for the vector space $R_i^\ell \bigl( {\rm Ker}(L_i) \cap P_N\bigr) $}

\noindent We continue to discuss the $\mathfrak{sl}_4(\mathbb C)$-module $P= \mathbb C\lbrack x,y,z,w\rbrack$. 
Pick  $i \in \lbrace 1,2,3\rbrace$ and recall the orthogonal direct sum  \eqref{eq:SSKsum}.
In this section, we find some bases for each summand.

\begin{lemma} \label{lem:cap4}  For  $i \in \lbrace 1,2,3\rbrace$ and $N \in \mathbb N$ the
subspace ${\rm Ker}(L_i) \cap P_N$ contains $x^N$ and $x^{*N}$.
\end{lemma}
\begin{proof} We first consider $x^N$.
We have $x^N \in P_N$ since $x^N$ is homogeneous with total degree $N$. We have $x^N \in {\rm Ker}(L_i)$ by
Definition \ref{def:L123} and since
\begin{align*}
D_y (x^N)=0, \qquad \quad D_z(x^N) = 0, \qquad \quad D_w(x^N) = 0.
\end{align*}
By these comments, $ x^N \in {\rm Ker}(L_i) \cap P_N$. We have $ x^{*N} \in {\rm Ker}(L_i) \cap P_N$ because $\sigma(x^N)=x^{*N}$ and
$ {\rm Ker}(L_i) \cap P_N$ is invariant under $\sigma$ by Lemma \ref{lem:KPinv}.
\end{proof}

\begin{definition} \label{def:fi} \rm Let $\eta$ denote an indeterminate. For $N \in \mathbb N$ we define some polynomials $\lbrace f_n\rbrace_{n=0}^{N+1}$
in $\mathbb C \lbrack \eta \rbrack$ such that $f_0=1$ and
\begin{align*}
&\eta f_n= n f_{n-1} + (N-n) f_{n+1} \qquad \quad (0 \leq n \leq N-1), \\
& \eta f_N = N f_{N-1} + f_{N+1},
\end{align*}
where $f_{-1}=0$.
The polynomial $f_n$ has degree $n$ for $0 \leq n \leq N+1$.
The  $\lbrace f_n\rbrace_{n=0}^{N+1}$ are  Krawtchouk polynomials, see \cite[Section~9.11]{KLS},  \cite{KrawSL2}, \cite[Section~6]{madrid}.
\end{definition}
\noindent We refer to Definition \ref{def:fi}.
By \cite[Lemma~4.8]{KrawSL2} we have
\begin{align} \label{eq:minPolyA}
f_{N+1}(\eta)= \frac{(\eta-N)(\eta-N+2)(\eta-N+4) \cdots (\eta+N)}{N!}.
\end{align}
\noindent The polynomials $\lbrace f_n \rbrace_{n=0}^{N+1}$ are related to the irreducible $\mathfrak{sl}_2(\mathbb C)$-module ${\mathbb V}_N$ in the following way.
Above Definition \ref{def:compondent} we discussed the basis $\lbrace u_n \rbrace_{n=0}^N$ for ${\mathbb V}_N$.  By that discussion and $A=E+F$ we obtain
\begin{align*}
A u_n &= n u_{n-1} + (N-n) u_{n+1} \qquad \quad (0 \leq n \leq N-1), \\
A u_N &= N u_{N-1},
\end{align*}
where $u_{-1}=0$.
Comparing this recurrence with the one in Definition \ref{def:fi}, we find
\begin{align*}
f_n(A) u_0 = u_n \qquad \qquad (0 \leq n \leq N).
\end{align*}
We also find that $N! f_{N+1}$ is the minimal polynomial of $A$  on ${\mathbb V}_N$. See \cite{KrawSL2} for more information about
the Krawtchouk polynomials and $\mathfrak{sl}_2(\mathbb C)$.

\begin{lemma} \label{lem:Ni} For $i \in \lbrace 1,2,3\rbrace$ and $N \in \mathbb N$ the following hold on $P_N$:
\begin{align}
A_i  f_n (A_i ) &= n f_{n-1}(A_i) + (N-n) f_{n+1}(A_i) \qquad \qquad (0 \leq n \leq N),  \label{eq:rec1} \\
f_{N+1}(A_i) &=0.     \label{eq:rec2}
\end{align}
\end{lemma} 
\begin{proof} By Lemma \ref{lem:Aspec} and \eqref{eq:minPolyA} we obtain \eqref{eq:rec2}.
From this and Definition \ref{def:fi} we obtain  \eqref{eq:rec1}.
\end{proof}

\begin{lemma} \label{lem:fMeaning} For $N \in \mathbb N$  and $0 \leq n \leq N$ we have
\begin{align*}
&f_n(A_1) x^N = x^{N-n} y^n, \qquad \quad
f_n(A_2) x^N = x^{N-n} z^n, \qquad \quad
f_n(A_3) x^N = x^{N-n} w^n.
\end{align*}
\end{lemma}
\begin{proof} We first prove our assertions involving $A_1$. Define 
\begin{align*}
\xi_n = x^{N-n} y^n \qquad (0 \leq n \leq N), \qquad \qquad \xi_{N+1}=0.
\end{align*}
By Proposition \ref{lem:ActP}(i),
\begin{align}
A_1 \xi_n = n \xi_{n-1} + (N-n) \xi_{n+1} \qquad \qquad (0 \leq n \leq N). \label{eq:hrec}
\end{align}
Comparing  \eqref{eq:rec1},
\eqref{eq:hrec}  we see that the  sequences $\lbrace f_n(A_1) x^N \rbrace_{n=0}^{N}$ and $\lbrace \xi_n \rbrace_{n=0}^{N}$ satisfy the same recurrence. These sequences have
the same initial condition, since $f_0(A_1)x^N =x^N= \xi_0$. Therefore $f_n(A_1)x^N=\xi_n$ for $0 \leq n \leq N$.
We have proven our assertions involving $A_1$. The remaining assertions are similarly proven.
\end{proof}

\noindent Let $N \in \mathbb N$. In the next result, we give an orthogonal basis for  ${\rm Ker}(L_1) \cap P_N$. Similar orthogonal bases exist for 
 ${\rm Ker}(L_2) \cap P_N$ and ${\rm Ker}(L_3) \cap P_N$.

\begin{proposition} \label{lem:KLbasis} For $N \in \mathbb N$  the
vectors
\begin{align}
v_{j,k}=f_j(A_2) f_k(A_3) x^N        \qquad \qquad 0 \leq j,k\leq N  \label{eq:vij}
\end{align}
form an orthogonal basis for ${\rm Ker}(L_1) \cap P_N$. Moreover
 for $0 \leq j,k\leq N$,
\begin{align}
A_2 v_{j,k} &= j v_{j-1,k} + (N-j) v_{j+1,k}, \label{eq:w1A} \\
A_3 v_{j,k} &= k v_{j,k-1} + (N-k) v_{j,k+1},       \label{eq:w2A} \\
A^*_2 v_{j,k} &= (N-2k) v_{j,k},       \label{eq:w3A}  \\
A^*_3 v_{j,k} &= (N-2j) v_{j,k},        \label{eq:w4A}   \\
\Vert v_{j,k} \Vert^2 &= N! \binom{N}{j}^{-1} \binom{N}{k}^{-1} . \label{eq:w5}
\end{align}
In the above lines, we interpret $v_{a,b} =0$ unless $0 \leq a,b\leq N$.
\end{proposition}
\begin{proof}  
By Lemma  \ref{lem:KPinv}(i) the subspace ${\rm Ker}(L_1) \cap P_N$ is invariant under $A_2$ and $A_3$. By this and
Lemma \ref{lem:cap4},    the vectors \eqref{eq:vij} are contained in ${\rm Ker}(L_1) \cap P_N$.
The vectors \eqref{eq:vij} satisfy the recurrence \eqref{eq:w1A}  by Lemma  \ref{lem:Ni}.
The vectors \eqref{eq:vij} satisfy the recurrence \eqref{eq:w2A}  by  Lemma  \ref{lem:Ni}  and 
$\lbrack A_2, A_3 \rbrack=0$. 
Concerning \eqref{eq:w3A}, we use Proposition \ref{lem:ActP}(v), Lemma \ref{lem:fMeaning},
and 
$\lbrack A_2, A^*_2 \rbrack=0$  to obtain
\begin{align*}
A^*_2 v_{j,k} &= A^*_2 f_j(A_2) f_k(A_3) x^N 
                     = f_j(A_2) A^*_2 f_k(A_3) x^N 
                     = f_j(A_2) A^*_2 x^{N-k} w^k \\
                   &  =(N-2k)f_j(A_2) x^{N-k}w^k 
                     = (N-2k) f_j(A_2) f_k(A_3) x^N 
                     = (N-2k) v_{j,k}.
\end{align*}
Concerning \eqref{eq:w4A}, we use Proposition \ref{lem:ActP}(vi), Lemma \ref{lem:fMeaning},
and  $\lbrack A_2, A_3\rbrack=0 = \lbrack A^*_3, A_3\rbrack$ to obtain
\begin{align*}
A^*_3 v_{j,k} &= A^*_3 f_j(A_2) f_k(A_3) x^N 
                     = f_k(A_3) A^*_3 f_j(A_2) x^N 
                     = f_k(A_3) A^*_3 x^{N-j} z^j \\
                   &  =(N-2j)f_k(A_3) x^{N-j}z^j
                     = (N-2j)f_k(A_3) f_j(A_2) x^N 
                                          = (N-2j) v_{j,k}.
\end{align*}
The vectors \eqref{eq:vij} are mutually orthogonal by \eqref{eq:w3A}, \eqref{eq:w4A} and Lemma \ref{lem:inv}.
Next we show \eqref{eq:w5}. Assume for the moment that $j\geq 1$. By Lemma \ref{lem:inv} and \eqref{eq:w1A},
\begin{align}
(N-j+1)\Vert v_{j,k} \Vert^2 = \langle A_2 v_{j-1,k}, v_{j,k} \rangle = \langle v_{j-1,k}, A_2 v_{j,k} \rangle = j \Vert v_{j-1,k} \Vert^2. \label{eq:Ind1}
\end{align}
Assume for the moment that $k\geq 1$. By Lemma \ref{lem:inv} and \eqref{eq:w2A},
\begin{align}
(N-k+1)\Vert v_{j,k} \Vert^2 = \langle A_3 v_{j,k-1}, v_{j,k} \rangle = \langle v_{j,k-1}, A_3 v_{j,k} \rangle = k \Vert v_{j,k-1} \Vert^2. \label{eq:Ind2}
\end{align}
 By \eqref{eq:Ind1}, \eqref{eq:Ind2} and induction on $j+k$, 
\begin{align*}
\Vert v_{j,k} \Vert^2 = \Vert v_{0,0} \Vert^2  \binom{N}{j}^{-1} \binom{N}{k}^{-1}.
\end{align*}
This and  $\Vert v_{0,0}\Vert^2=\Vert x^N\Vert^2 = N!$ yields \eqref{eq:w5}.
By \eqref{eq:w5} the vectors  \eqref{eq:vij} are nonzero.
The vectors  \eqref{eq:vij} are linearly independent, because they are nonzero and mutually orthogonal.
These vectors form a basis for ${\rm Ker}(L_1) \cap P_N$ in view of Proposition  \ref{prop:cap4}(ii).
\end{proof}

\begin{proposition} \label{lem:insert} The following hold for $N, \ell \in \mathbb N$:
\begin{enumerate}
\item[\rm (i)] 
$R_1^\ell \bigl( {\rm Ker}(L_1) \cap P_N \bigr)$  has an orthogonal basis
\begin{align*}
R_1^\ell  f_j(A_2) f_k(A_3)x^N \qquad \qquad 0 \leq j,k\leq N;
\end{align*}
\item[\rm (ii)] 
$R_2^\ell \bigl( {\rm Ker}(L_2) \cap P_N \bigr) $  has an orthogonal basis
\begin{align*}
R_2^\ell    f_j(A_3) f_k(A_1)x^N      \qquad \qquad 0 \leq j,k\leq N;
\end{align*}
\item[\rm (iii)] 
$R_3^\ell \bigl( {\rm Ker}(L_3) \cap P_N\bigr) $  has an orthogonal basis
\begin{align*}
R_3^\ell   f_j(A_1) f_k(A_2)x^N \qquad \qquad 0 \leq j,k\leq N.
\end{align*}
\end{enumerate}
\end{proposition} 
\begin{proof} (i) By Lemma \ref{lem:RRRinj},  Corollary \ref{cor:Orth},  and Proposition  \ref{lem:KLbasis}.
\\
\noindent (ii), (iii) Similar to the proof of (i) above.
\end{proof}

\begin{proposition} \label{prop:RPNbasis3} The following hold for $N, \ell \in \mathbb N$:
\begin{enumerate}
\item[\rm (i)] 
$ R_1^\ell \bigl( {\rm Ker}(L_1) \cap P_N \bigr) $  has a basis
\begin{align*}
R_1^\ell A^j_2 A^k_3 x^N \qquad \qquad 0 \leq j,k\leq N;
\end{align*}
\item[\rm (ii)] 
$R_2^\ell \bigl( {\rm Ker}(L_2) \cap P_N \bigr) $  has a basis
\begin{align*}
R_2^\ell A^j_3 A^k_1 x^N \qquad \qquad 0 \leq j,k\leq N;
\end{align*}
\item[\rm (iii)] 
$R_3^\ell \bigl( {\rm Ker}(L_3) \cap P_N \bigr) $  has a basis
\begin{align*}
R_3^\ell A^j_1 A^k_2 x^N \qquad \qquad 0 \leq j,k\leq N.
\end{align*}
\end{enumerate}
\end{proposition}
\begin{proof} By  Proposition \ref{lem:insert} and since the polynomial $f_n$ has degree $n$ for $0 \leq n \leq N$. 
\end{proof}

\begin{corollary} \label{cor:15p8} The following hold:
\begin{enumerate}
\item[\rm (i)] 
$P$  has a basis
\begin{align*}
R_1^\ell A^j_2 A^k_3 x^{N}  \qquad \qquad N, \ell \in \mathbb N,  \qquad \quad 0 \leq j,k\leq N;
\end{align*}
\item[\rm (ii)] 
$P$  has a basis
\begin{align*}
R_2^\ell A^j_3 A^k_1 x^{N} \qquad \qquad N, \ell \in \mathbb N,  \qquad \quad  0 \leq j,k\leq N;
\end{align*}
\item[\rm (iii)] 
$P$ has a basis
\begin{align*}
R_3^\ell A^j_1 A^k_2 x^{N}  \qquad \qquad N, \ell \in \mathbb N, \qquad \quad  0 \leq j,k\leq N.
\end{align*}
\end{enumerate}
\end{corollary}
\begin{proof} By Proposition \ref{prop:RPNbasis3}  and since the sum  \eqref{eq:SSKsum}  is direct.
\end{proof} 

\begin{proposition} \label{rem:AboutDual}  Lemmas \ref{lem:Ni}, \ref{lem:fMeaning}  and Propositions \ref{lem:KLbasis}, \ref{lem:insert}, \ref{prop:RPNbasis3}
 and Corollary \ref{cor:15p8}  all remain valid if
we replace $x,y,z,w$ by $x^*, y^*,z^*,w^*$ respectively and swap $A_i \leftrightarrow A^*_i $ for $i \in \lbrace 1,2,3\rbrace$.
\end{proposition}
\begin{proof} Apply $\sigma$ throughout the listed results, and use
Propositions \ref{thm:aut2}, \ref{thm:aut} along with Lemmas \ref{lem:LScom}, \ref{lem:DScom}. 
\end{proof}

\section{Some bases for the vector space $P_N$ }

\noindent We continue to discuss the $\mathfrak{sl}_4(\mathbb C)$-module $P= \mathbb C\lbrack x,y,z,w\rbrack$. 
Pick $i \in \lbrace 1,2,3 \rbrace$ and $N \in \mathbb N$. Recall the direct sum decomposition \eqref{eq:odsREMIND}.
 In Proposition \ref{prop:RPNbasis3}, we obtained a basis for each summand. 
In this section, we use these bases to obtain some bases for $P_N$.

\begin{proposition} \label{lem:PNalt} The following hold for $N \in \mathbb N$:
\begin{enumerate}
\item[\rm (i)] 
$P_N$  has a basis
\begin{align*}
R_1^\ell A^j_2 A^k_3 x^{N-2\ell}  \qquad \qquad 0 \leq \ell \leq \lfloor N/2 \rfloor, \qquad \quad 0 \leq j,k\leq N-2\ell;
\end{align*}
\item[\rm (ii)] 
$P_N$  has a basis
\begin{align*}
R_2^\ell A^j_3 A^k_1 x^{N-2\ell} \qquad \qquad 0 \leq \ell \leq \lfloor N/2 \rfloor, \qquad \quad  0 \leq j,k\leq N-2\ell;
\end{align*}
\item[\rm (iii)] 
$P_N$  has a basis
\begin{align*}
R_3^\ell A^j_1 A^k_2 x^{N-2\ell}  \qquad \qquad 0 \leq \ell \leq \lfloor N/2 \rfloor, \qquad \quad  0 \leq j,k\leq N-2\ell.
\end{align*}
\end{enumerate}
\end{proposition}
\begin{proof} By Proposition  \ref{prop:RPNbasis3}  and since the sum \eqref{eq:odsREMIND} is direct.
\end{proof}

\begin{proposition} Proposition \ref{lem:PNalt} remains valid if we replace $x$ by $x^*$ and $A_i$ by $A^*_i$ for $i \in \lbrace 1,2,3\rbrace$.
\end{proposition}
\begin{proof}  Apply $\sigma$ throughout Proposition \ref{lem:PNalt}, and use 
Proposition \ref{thm:aut2} along with Lemma\ref{lem:DScom} and the fact that $\sigma(P_N)=P_N$.
\end{proof}


\section{The hypercube $H(N,2)$}
\noindent 

\noindent 
We turn our attention to graph theory. For us, a graph is understood to be finite and undirected, without loops or multiple edges.
For the rest of this paper, fix  $N \in \mathbb N$. 
We define a graph $H(N,2)$ as follows. The vertex set $X$ consists of the
$N$-tuples of elements taken from the set $\lbrace 1,-1\rbrace$. Vertices $x,y \in X$
are adjacent whenever they differ in exactly one coordinate. The graph $H(N,2)$ 
is called the {\it $N$-cube} or a {\it hypercube} or a {\it binary Hamming graph}. 
 The graph $H(N,2)$ is distance-regular in the sense of \cite[Chapter~1]{bcn}. Background information about $H(N,2)$
 can be found in \cite{bbit, banIto, bcn, dkt, go}.
 Going forward, we will assume
 that the reader is generally familiar with $H(N,2)$. In the next few paragraphs, we recall from \cite{go} some features of $H(N,2)$ for later use.
 \medskip
 
 \noindent We have $\vert X \vert = 2^N$.
 For $x,y \in X$ let $\partial(x,y)$ denote the path-length distance between $x,y$.
 Then $\partial(x,y)$ is equal to the number of coordinates at which $x,y$ differ.
 The diameter of $H(N,2)$ is  $N$.
 The graph $H(N,2)$ is a bipartite  antipodal $2$-cover.
 The intersection numbers $c_i, b_i$ of $H(N,2)$ satisfy
\begin{align*}
c_i = i, \qquad \qquad b_i = N-i \qquad \qquad (0 \leq i \leq N).
\end{align*}
The valencies $k_i$ of $H(N,2)$ satisfy
\begin{align*}
k_i = \binom{N}{i} \qquad \qquad (0 \leq i \leq N).
\end{align*}
\begin{definition}\rm \label{def:standard} Let $V$ denote the vector space  with basis $X$. We call $V$
the {\it standard module} associated with $H(N,2)$.
\end{definition}

\begin{definition}\rm \label{def:HermHN2} \rm We endow $V$ with a Hermitian form $\langle \,,\,\rangle$ with respect to which the basis $X$ is orthonormal.
\end{definition}
\noindent Let us abbreviate  $\Gamma=H(N,2)$. For $x \in X$ let the set $\Gamma(x)$ 
 consist of the vertices in $X$ that are adjacent to $x$. Note that $\vert \Gamma(x) \vert = N$.
 \begin{definition} \label{def:adj} \rm
Define ${\sf A} \in {\rm End}(V)$ such that 
\begin{align}
{\sf A} x = \sum_{\xi \in \Gamma(x)} \xi, \qquad \qquad x \in X.         \label{eq:AdjMap}
\end{align}
We call $\sf A$ the {\it adjacency map} for $H(N,2)$.
\end{definition}
\begin{lemma} \label{lem:Abil} For $x,y \in X$ we have
\begin{align*}
\langle {\sf A}x,y\rangle  = \langle x,{\sf A}y\rangle  =  \begin{cases} 1 & {\mbox{\rm if $\partial(x,y)=1$}}; \\
0, & {\mbox{\rm if $\partial(x,y)\not= 1$}}.
\end{cases} 
\end{align*}
\end{lemma}
\begin{proof} By Definition \ref{def:adj}.
\end{proof}
\noindent By \cite[p.~45]{bcn} and \cite[Lemma~3.5]{go}, the map $\sf A$ is diagonalizable with eigenvalues
\begin{align}
\theta_i = N-2i \qquad \quad (0 \leq i \leq N). \label{eq:eigvalHN2}
\end{align}
For $0 \leq i \leq N$  let ${\sf E}_i$ denote the primitive idempotent of $\sf A$ associated with $\theta_i$.
By \cite[Lemma~3.5]{go} the eigenspace ${\sf E}_iV$ has dimension $m_i = \binom{N}{i}$.

\begin{lemma} \label{lem:Ebil} For $0 \leq i \leq N$ and $x,y \in X$,
\begin{align*}
\langle {\sf E}_i x, y\rangle = \langle x, {\sf E}_iy\rangle = \langle {\sf E}_i y, x\rangle = \langle y, {\sf E}_ix \rangle.
\end{align*}
\end{lemma}
\begin{proof} By \eqref{eq:Eprod} and Lemma  \ref{lem:Abil}, and since the eigenvalues \eqref{eq:eigvalHN2}
 are real.
\end{proof}

\noindent The ordering $\lbrace {\sf E}_i \rbrace_{i=0}^N$ is $Q$-polynomial in the sense of \cite[Section~12]{go}. For this ordering
the corresponding dual eigenvalue sequence $\lbrace \theta^*_i \rbrace_{i=0}^N$ is defined in \cite[Definition~3.6]{go}. By \cite[Lemma 3.7]{go} we have
\begin{align}
\theta^*_i = N-2i \qquad \quad (0 \leq i \leq N). \label{eq:deigvalHN2}
\end{align}
By \cite[p.~194]{bcn}, the above $Q$-polynomial structure is formally  self-dual in the sense of \cite[Section~2.3]{bcn}.

\begin{definition} \rm \label{def:V3} We define the vector space $V^{\otimes 3} = V \otimes V \otimes V$ and the set
\begin{align*}
X^{\otimes 3} = \lbrace x \otimes y \otimes z\vert x,y,z \in X \rbrace.
\end{align*}
Observe that $X^{\otimes 3}$ is a basis for $V^{\otimes 3}$.
\end{definition}

\begin{lemma} \label{lem:Herm2} The following hold.
\begin{enumerate}
\item[\rm (i)]
There exists a unique Hermitian form $\langle\,,\,\rangle$ on $V^{\otimes 3}$
with respect to which the basis $X^{\otimes 3}$  is orthonormal.
\item[\rm (ii)]  For $u,v,w,u',v',w' \in V$ we have
\begin{align*} 
\langle u \otimes v \otimes w, u' \otimes v' \otimes w' \rangle = \langle u,u'\rangle \langle v, v' \rangle \langle w, w' \rangle.
\end{align*}
\end{enumerate}
\end{lemma}
\begin{proof} Item (i) is clear. Item (ii) is routinely checked.
\end{proof}

\noindent Let $G$ denote the
automorphism group of $H(N,2)$.
By \cite[Theorem~9.2.1]{bcn} the group $G$ is isomorphic to the wreath product of the symmetric groups $S_N$ and
$S_2$.
The elements of $S_N$ permute the vertex coordinates $\lbrace 1,2,\ldots, N\rbrace$
and the elements of $S_2$ permute the set $\lbrace 1,-1\rbrace$.
By \cite[p.~207]{banIto} 
the graph $H(N,2)$ is distance-transitive in the sense of \cite[p.~189]{banIto}. 
\medskip

\noindent Next, we describe how  $V$ becomes a $G$-module.
Pick $v \in V$ and write $v = \sum_{x \in X} v_x x$ $(v_x \in \mathbb C)$. For all $g \in G$,
\begin{align*}
g(v) = \sum_{x \in X} v_x g(x).
\end{align*} 

\begin{lemma} \label{lem:gAcom} For $g \in G$ the following hold on $V$:
\begin{enumerate}
\item[\rm (i)] $g{\sf A}={\sf A}g$;
\item[\rm (ii)] $g {\sf E}_i = {\sf E}_i g$ $(0 \leq i \leq N)$.
\end{enumerate}
\end{lemma}
\begin{proof} (i) Because $g$ respects adjacency in $H(N,2)$. \\
\noindent (ii) By (i) and since ${\sf E}_i$ is a polynomial in $\sf A$.
\end{proof}

\noindent Next, we describe how $V^{\otimes 3}$ becomes a $G$-module. For $g \in G$ and $u,v,w \in V$,
\begin{align}
g(u \otimes v \otimes w) = g(u) \otimes g(v) \otimes g(w).        \label{eq:gVVV}
\end{align}

\begin{definition}\rm \label{def:Fix} Define the subspace
\begin{align*}
{\rm Fix}(G) = \lbrace v \in V^{\otimes 3} \vert g(v)=v \;\forall g \in G\rbrace.
\end{align*}
\end{definition}

\noindent Our next general goal is to obtain a basis for ${\rm Fix}(G)$. To reach this goal, we consider the action of $G$ on the 
set $X^{\otimes 3}$. We will describe the partition of $X^{\otimes 3}$ into $G$-orbits.
\medskip

\noindent Recall the set of profiles $\mathcal P_N$ from Definition \ref{def:profiles}.

\begin{definition} \label{def:xyzprofile} \rm For $x \otimes y \otimes z \in X^{\otimes 3}$ we define its profile as follows.
Write
\begin{align*}
x = (x_1, x_2,\ldots, x_N), \qquad \quad
y=(y_1,y_2,\ldots, y_N), \qquad \quad
z = (z_1, z_2,\ldots, z_N).
\end{align*}

\noindent Define
\begin{align*}
r &= \bigl \vert \lbrace i \vert 1 \leq i \leq N, \; x_i = y_i = z_i \rbrace \bigr \vert, \\
s &= \bigl \vert \lbrace i \vert 1 \leq i \leq N, \; x_i \not= y_i = z_i \rbrace \bigr \vert, \\
t &= \bigl \vert \lbrace i \vert 1 \leq i \leq N, \; y_i \not= z_i = x_i \rbrace \bigr \vert, \\
u &= \bigl \vert \lbrace i \vert 1 \leq i \leq N, \; z_i \not= x_i = y_i \rbrace \bigr \vert.
\end{align*}
Note that $(r,s,t,u) \in \mathcal P_N$. We call $(r,s,t,u)$ the {\it profile of $x \otimes y \otimes z$}.
\end{definition}

\begin{lemma} \label{lem:profileSize} For a profile $(r,s,t,u) \in \mathcal P_N$ the number of elements in $X^{\otimes 3}$ with this profile is equal to
\begin{align} \label{eq:pcount}
\frac{N! 2^N}{r!s!t!u!}.
\end{align}
\end{lemma}
\begin{proof} By combinatorial counting.
\end{proof}

\begin{lemma} \label{lem:orbitprofile} A pair of elements in $X^{\otimes 3}$
are in the same $G$-orbit if and only if they have the same profile.
\end{lemma}
\begin{proof} This is routinely checked.
\end{proof}

\noindent As an aside, we interpret the profile concept using the distance function $\partial$.

\begin{lemma} \label{lem:profmean} Let $x \otimes y \otimes z \in X^{\otimes 3}$ with profile
$(r,s,t,u)$.  Then
\begin{align*}
\partial(x,y)= s+t, \qquad \quad \partial(y,z)=t+u, \qquad \quad \partial(z,x) = u+s.
\end{align*}
Moreover
\begin{align*}
&r = \frac{ 2N - \partial(x,y) - \partial(y,z)-\partial(z,x)}{2}, \qquad \quad s = \frac{\partial(z,x)+\partial(x,y)-\partial(y,z)}{2}, \\
&t = \frac{\partial(x,y)+\partial(y,z)-\partial(z,x)}{2}, \qquad \quad u = \frac{\partial(y,z)+\partial(z,x)-\partial(x,y)}{2}.
\end{align*}
\end{lemma}
\begin{proof} By Definition \ref{def:xyzprofile}.
\end{proof}

\begin{definition} \label{def:Php} \rm Let the set $\mathcal P''_N$ consist of the $3$-tuples of integers $(h,i,j)$ such that
\begin{align*}
&0 \leq h,i,j\leq N, \qquad  h+i+j \;\hbox{\rm is even}, \qquad  h+i+j \leq 2N, \\
&h \leq i+j, \quad \qquad \qquad i \leq j+h, \quad \qquad \qquad j \leq h+i.
\end{align*}
\end{definition}

\begin{lemma} \label{lem:PPbij} There exists a bijection $\mathcal P_N \to \mathcal P''_N$ that sends
\begin{align*}
(r,s,t,u) \mapsto (t+u,u+s,s+t).
\end{align*}
The inverse bijection  $\mathcal P''_N \to \mathcal P_N$ sends
\begin{align*}
(h,i,j) \mapsto \biggl( \frac{2N-h-i-j}{2}, \frac{i+j-h}{2}, \frac{j+h-i}{2}, \frac{h+i-j}{2}\biggr).
\end{align*}
\end{lemma}
\begin{proof} This is readily checked.
\end{proof}

\begin{lemma} \label{lem:3wayBij} For  $0 \leq h,i,j\leq N$ the following {\rm (i)--(iii)}
are equivalent:
\begin{enumerate}
\item[\rm (i)] there exists $x \otimes y \otimes z \in X^{\otimes 3}$  such that
\begin{align*}
h=\partial(y,z), \qquad \qquad i=\partial(z,x), \qquad \qquad
j = \partial(x,y);
\end{align*}
\item[\rm (ii)]  there exists $(r,s,t,u) \in \mathcal P_N$ such that
\begin{align*}
 h=t+u, \qquad \qquad i = u+s, \qquad \qquad j= s+t;
 \end{align*}
\item[\rm (iii)]  $(h,i,j) \in \mathcal P''_N$.
\end{enumerate}
\noindent Assume that {\rm (i)--(iii)} hold. Then $(r,s,t,u)$ is the profile of $x\otimes y \otimes z$.
\end{lemma}
\begin{proof} ${\rm (i)} \Leftrightarrow {\rm (ii)}$ By Lemmas  \ref{lem:profileSize} \ref{lem:profmean}. \\
${\rm (ii)} \Leftrightarrow {\rm (iii)}$ By Lemma \ref{lem:PPbij}. \\
The last assertion follows from Lemmas \ref{lem:profmean}, \ref{lem:PPbij}.
\end{proof}

\begin{remark}\label{rem:IntKrein}
\rm For $H(N,2)$ and $0 \leq h,i,j\leq N$ 
there are some parameters called the intersection number $p^h_{i,j}$ \cite[p.~401]{go} and  Krein parameter $q^h_{i,j}$ \cite[p.~402]{go}.
We use these parameters in a minimal way, but for the sake of completeness let us discuss them briefly. We have
$p^h_{i,j} = p^h_{j,i}$  \cite[p.~401]{go} and  $q^h_{i,j}=q^h_{j,i}$ \cite[p.~402]{go}.
We have $p^h_{i,j} =q^h_{i,j}$  \cite[Lemma~22]{nicholson}. This common value is nonzero if and only if $(h,i,j) \in \mathcal P''_N$; see
 \cite[Corollary~28]{nicholson}. By \cite[Proposition~12]{nicholson}, for  $(h,i,j) \in \mathcal P''_N$ we have
 \begin{align*}
 k_h p^h_{i,j} = k_i p^i_{j,h} = k_j p^j_{h,i} =
 m_h q^h_{i,j} = m_i q^i_{j,h} = m_j q^j_{h,i} =\frac{N!}{r!s!t!u!},
 \end{align*}
 where 
 \begin{align*}
 h=t+u, \qquad \qquad i = u+s, \qquad \qquad j= s+t.
 \end{align*}
\end{remark}

\begin{definition} \label{def:Brstu} \rm For a profile $(r,s,t,u) \in \mathcal P_N$ we define a vector
\begin{align*}
B(r,s,t,u)  =  \sum_{x \otimes y \otimes z} x \otimes y \otimes z,
\end{align*}
where the sum is over the elements $x \otimes y \otimes z$ in $X^{\otimes 3}$ with profile $(r,s,t,u)$.
\end{definition}

\begin{example} \label{ex:BN000} \rm We have
\begin{align*}
B(N,0,0,0) = \sum_{x \in X}   x \otimes x \otimes x.
\end{align*}
\end{example}

\noindent We have a remark about  notation.
\begin{remark} \label{def:Phij} \rm Let  $0 \leq h,i,j\leq N$. In \cite[Definition~9.9]{S3}  we defined a vector
\begin{align*}
P_{h,i,j}  =  \sum_{x \otimes y \otimes z} x \otimes y \otimes z,
\end{align*}
where the sum is over the elements $x \otimes y \otimes z$ in $X^{\otimes 3}$ such that
\begin{align*}
h= \partial(y,z), \qquad \qquad i = \partial(z,x), \qquad \qquad j = \partial(x,y).
\end{align*}
By Lemma \ref{lem:3wayBij},
$P_{h,i,j} \not=0$ if and only if $(h,i,j) \in \mathcal P''_N$. In this case
\begin{align*}
P_{h,i,j}= B(r,s,t,u),
\end{align*}
where 
\begin{align*}
h=t+u, \qquad \qquad i = u+s, \qquad \qquad j = s+t.
\end{align*}
\end{remark}

\begin{lemma} \label{lem:orthogB} The vectors
\begin{align*} 
B(r,s,t,u) \qquad \qquad (r,s,t,u) \in \mathcal P_N
\end{align*}
are mutually orthogonal and
\begin{align*}
\Vert B(r,s,t,u) \Vert^2 = \frac{N! 2^N}{r!s!t!u!} \qquad \qquad (r,s,t,u) \in \mathcal P_N.
\end{align*}
\end{lemma}
\begin{proof} 
The vectors $X^{\otimes 3}$ are orthonormal 
 with respect to $\langle \,,\,\rangle$. The result follows in view of Lemma  \ref{lem:profileSize} and Definition  \ref{def:Brstu}.
\end{proof}

\begin{proposition} \label{lem:FixBasis} The  vectors
\begin{align} 
B(r,s,t,u) \qquad \qquad (r,s,t,u) \in \mathcal P_N \label{eq:FixBasis}
\end{align}
form a  basis for ${\rm Fix}(G)$.
\end{proposition}
\begin{proof} The vectors \eqref{eq:FixBasis} are linearly independent by Lemma \ref{lem:orthogB}.
These vectors span ${\rm Fix}(G)$ by Lemma \ref{lem:orbitprofile} and  Definition \ref{def:Brstu}.
\end{proof}

\begin{corollary}\label{cor:FixGdim} We have
\begin{align*}
{\rm dim}\, {\rm Fix}(G) = \binom{N+3}{3}.
\end{align*}
\end{corollary}
\begin{proof} By  \eqref{eq:PNsize} and Proposition \ref{lem:FixBasis}.
\end{proof}

\noindent Next, we describe the basis for ${\rm Fix}(G)$ that is dual to the one in Proposition \ref{lem:FixBasis},
with respect to the Hermitian form $\langle \,,\,\rangle$.

\begin{definition} \label{lem:Bdual} \rm For a profile $(r,s,t,u) \in \mathcal P_N$ we define a vector
\begin{align*}
\tilde B(r,s,t,u) = \frac{r!s!t!u!}{N! 2^N} B(r,s,t,u).
\end{align*}
\end{definition}

\begin{proposition} \label{lem:FixBasisd} The  vectors
\begin{align*} 
\tilde B(r,s,t,u) \qquad \qquad (r,s,t,u) \in \mathcal P_N 
\end{align*}
form a  basis for ${\rm Fix}(G)$.
\end{proposition}
\begin{proof} By Proposition \ref{lem:FixBasis} and Definition \ref{lem:Bdual}.
\end{proof}

\begin{lemma} \label{lem:Fixdual} The ${\rm Fix}(G)$-basis
\begin{align*}
B(r,s,t,u) \qquad \qquad (r,s,t,u) \in \mathcal P_N 
\end{align*}
and the 
${\rm Fix}(G)$-basis
\begin{align*}
\tilde B(r,s,t,u) \qquad \qquad (r,s,t,u) \in \mathcal P_N 
\end{align*}
are dual with respect to $\langle \,,\,\rangle$.
\end{lemma}
\begin{proof} By Lemma  \ref{lem:orthogB} and Definition  \ref{lem:Bdual}.
\end{proof}

\noindent Recall the vector space $P_N$ from Lemma \ref{lem:PDbasic}.

\begin{lemma} \label{lem:vee} There exists a vector space isomorphism $\ddag: P_N \to {\rm Fix}(G)$ that sends
\begin{align*}
x^r y^s z^t w^u \mapsto    \bigl( N!2^N\bigr)^{1/2}     \tilde B(r,s,t,u) \qquad \qquad (r,s,t,u) \in \mathcal P_N.
\end{align*}
\end{lemma} 
\begin{proof} By Lemma \ref{lem:PDbasic} and Proposition  \ref{lem:FixBasisd}.
\end{proof}

\begin{theorem}  \label{lem:HermComp} The Hermitian forms on $P_N$ and ${\rm Fix}(G)$ are related as follows:
\begin{align*}
\langle f,g \rangle = \langle f^\ddag, g^\ddag \rangle \qquad \qquad f,g \in P_N.
\end{align*}
\end{theorem}
\begin{proof}
By Definitions \ref{def:bform},  \ref{lem:Bdual}   and Lemmas \ref{lem:Fixdual}, \ref{lem:vee}.
\end{proof}

\noindent Our next goal is to turn ${\rm Fix}(G)$ into an $\mathfrak{sl}_4(\mathbb C)$-module.

\begin{definition}\rm \label{def:mapsAAA} (See \cite[Definition~6.2]{S3}.) 
We define $A^{(1)}, A^{(2)}, A^{(3)} \in {\rm End}(V^{\otimes 3})$ as follows. For $x\otimes y\otimes z \in X^{\otimes 3}$,
\begin{align*}
A^{(1)} (x\otimes y \otimes z) &= \sum_{\xi \in \Gamma(x)} \xi \otimes y \otimes z, \\
A^{(2)} (x\otimes y \otimes z) &= \sum_{\xi \in \Gamma(y)} x \otimes \xi \otimes z, \\
A^{(3)} (x\otimes y \otimes z) &= \sum_{\xi \in \Gamma(z)} x \otimes y \otimes \xi.
\end{align*}
\end{definition}

\noindent In the next result, we clarify Definition \ref{def:mapsAAA}.

\begin{lemma} \label{lem:Atensor} For $u,v,w \in V$ we have
\begin{align*}
&A^{(1)} (u\otimes v \otimes w) = {\sf A}u\otimes v \otimes w, \\
&A^{(2)} (u\otimes v \otimes w) = u\otimes {\sf A}v \otimes w, \\
&A^{(3)} (u\otimes v \otimes w) = u\otimes v \otimes {\sf A}w,
\end{align*}
where $\sf A$ is the adjacency map from Definition \ref{def:adj}.
\end{lemma}
\begin{proof} By Definitions \ref{def:adj}, \ref{def:mapsAAA}.
\end{proof}

\begin{definition}\rm \label{def:mapsAAAs} (See \cite[Definition~7.1]{S3}.)
We define $A^{*(1)}, A^{*(2)}, A^{*(3)} \in {\rm End}(V^{\otimes 3})$ as follows. For $x\otimes y\otimes z \in X^{\otimes 3}$,
\begin{align*}
A^{*(1)} (x\otimes y \otimes z) &=  x \otimes y \otimes z\,\theta^*_{\partial(y,z)}, \\
A^{*(2)} (x\otimes y \otimes z) &=  x \otimes y\otimes z\, \theta^*_{\partial(z,x)}  , \\
A^{*(3)} (x\otimes y \otimes z) &=  x \otimes y \otimes z\, \theta^*_{\partial(x,y)}.
\end{align*}
\end{definition}

\begin{proposition} \label{lem:HAactionC} For a profile $(r,s,t,u) \in \mathcal P_N$ the following {\rm (i)--(vi)} hold:
\begin{enumerate}
\item[\rm (i)] the vector
\begin{align*}
A^{(1)} \tilde B(r,s,t,u)
\end{align*}
 is a linear combination with the following terms and coefficients:
\begin{align*} 
\begin{tabular}[t]{c|c}
{\rm Term }& {\rm Coefficient} 
 \\
 \hline
   $ \tilde B(r-1,s+1,t,u)$& $r$   \\
 $\tilde B(r+1,s-1,t,u)$   & $s$\\
  $ \tilde B(r,s,t-1,u+1)$  & $t$ \\
  $ \tilde B(r,s,t+1,u-1) $& $u$
   \end{tabular}
\end{align*}
\item[\rm (ii)] 
the vector
\begin{align*}
A^{(2)} \tilde B(r,s,t,u) 
\end{align*}
 is a linear combination with the following terms and coefficients:
\begin{align*} 
\begin{tabular}[t]{c|c}
{\rm Term }& {\rm Coefficient} 
 \\
 \hline
   $ \tilde B(r-1,s,t+1,u) $& $ r$   \\
 $ \tilde B(r,s-1,t,u+1) $   & $s$\\
  $ \tilde B(r+1,s,t-1,u)$  & $t$ \\
  $ \tilde B(r,s+1,t,u-1) $& $u $
   \end{tabular}
\end{align*}
\item[\rm (iii)] 
the vector
\begin{align*}
A^{(3)} \tilde B(r,s,t,u)
\end{align*}
 is a linear combination with the following terms and coefficients:
\begin{align*} 
\begin{tabular}[t]{c|c}
{\rm Term }& {\rm Coefficient} 
 \\
 \hline
    $ \tilde B(r-1,s,t,u+1)$& $r $   \\
 $ \tilde B(r,s-1,t+1,u)  $   & $s $\\
    $ \tilde B(r,s+1,t-1,u) $  & $t$ \\
  $ \tilde  B(r+1,s,t,u-1) $& $u $
   \end{tabular}
\end{align*}
\item[\rm (iv)] $A^{*(1)} \tilde B(r,s,t,u) = (r+s-t-u) \tilde B(r,s,t,u)$;
\item[\rm (v)] $A^{*(2)} \tilde B(r,s,t,u) = (r-s+t-u) \tilde B(r,s,t,u)$;
\item[\rm (vi)] $A^{*(3)} \tilde B(r,s,t,u)  = (r-s-t+u) \tilde B(r,s,t,u)$.
\end{enumerate}
\end{proposition}
\begin{proof} (i)--(iii) By combinatorial counting. \\
\noindent (iv)--(vi)  By \eqref{eq:deigvalHN2} and Lemma  \ref{lem:profmean} along
with Definitions \ref{def:Brstu}, \ref{lem:Bdual}, \ref{def:mapsAAAs}.
\end{proof} 

\begin{lemma} \label{lem:FixInv} The subspace ${\rm Fix}(G)$ is invariant under the maps
\begin{align*}
A^{(i)}, \quad A^{*(i)} \qquad \quad i \in \lbrace 1,2,3 \rbrace.
\end{align*}
\end{lemma}
\begin{proof} By Propositions \ref{lem:FixBasisd}, \ref{lem:HAactionC}.  
\end{proof}

\begin{theorem} \label{prop:GenAction} The subspace ${\rm Fix}(G)$ becomes an $\mathfrak{sl}_4(\mathbb C)$-module on which
\begin{align*}
A_i = A^{(i)}, \qquad \quad A^*_i = A^{*(i)} \qquad \qquad i \in \lbrace 1,2,3\rbrace.
\end{align*}
Moreover, the map $\ddagger: P_N \to {\rm Fix}(G)$ is an isomorphism of  $\mathfrak{sl}_4(\mathbb C)$-modules.
\end{theorem}
\begin{proof} Compare Propositions \ref{lem:ActP}, \ref{lem:HAactionC} using Lemma \ref{lem:vee}.
\end{proof}

\medskip
\noindent Our next general goal
is to explain what the $\mathfrak{sl}_4(\mathbb C)$-module isomorphism $\ddag: P_N \to {\rm Fix}(G)$
does to the $P_N$-basis given in Lemma \ref{lem:PDdualbasis}.

\begin{definition}\rm \label{def:Qhij} (See \cite[Definition~9.14]{S3}.)         For $0 \leq h,i,j\leq N$ we define a vector
\begin{align*}
Q_{h,i,j} = 2^N \sum_{x \in X} {\sf E}_h x \otimes {\sf E}_i x \otimes {\sf E}_j x.
\end{align*}
\end{definition}

\begin{lemma} \label{lem:Qcondition}  For $0 \leq h,i,j\leq N$ the following hold:
\begin{enumerate}
\item[\rm (i)] $\Vert Q_{h,i,j} \Vert^2 = \Vert P_{h,i,j} \Vert^2$, where $P_{h,i,j}$ is from Remark \ref{def:Phij};
\item[\rm (ii)] $Q_{h,i,j} \not=0$ if and only if $(h,i,j) \in \mathcal P''_N$.
\end{enumerate}
\end{lemma}
\begin{proof}  (i) By \cite[Lemmas~9.11,~9.16]{S3} and  Remark \ref{rem:IntKrein}. \\
\noindent (ii) By Remark \ref{def:Phij} and (i) above.
\end{proof} 

\begin{definition}\label{lem:Bs} \rm For a profile $(r,s,t,u) \in \mathcal P_N$ we define a vector
\begin{align*}
B^*(r,s,t,u) = Q_{h,i,j},
\end{align*}
\noindent where
\begin{align}
h = t+u, \qquad \qquad i = u+s, \qquad \qquad j = s+t.       \label{eq:hij}
\end{align}
\end{definition}

\begin{lemma} \label{lem:BsOrth} The  vectors 
\begin{align}
B^*(r,s,t,u) \qquad \qquad (r,s,t,u) \in \mathcal P_N          \label{eq:step1}
\end{align}
are mutually orthogonal and
\begin{align}
\Vert B^*(r,s,t,u) \Vert^2 = \frac{N! 2^N}{r!s!t!u!} \qquad \qquad (r,s,t,u) \in \mathcal P_N.       \label{eq:step2}
\end{align}
\end{lemma} 
\begin{proof} The  vectors \eqref{eq:step1} are mutually orthogonal by  \cite[Lemma~9.16]{S3} and Definition \ref{lem:Bs}.
For $(r,s,t,u) \in \mathcal P_N$ and $(h,i,j)$ from \eqref{eq:hij}, the following holds
by Definition \ref{lem:Bs}, Lemma \ref{lem:Qcondition}(i), Remark \ref{def:Phij}, Lemma  \ref{lem:orthogB}:
\begin{align*}
\Vert B^*(r,s,t,u) \Vert^2 = \Vert Q_{h,i,j} \Vert^2 = \Vert P_{h,i,j}\Vert^2 =   \Vert B(r,s,t,u) \Vert^2 =\frac{N! 2^N}{r!s!t!u!}.
\end{align*}
\end{proof}

\begin{proposition} \label{lem:BsBasis}  The  vectors 
\begin{align}
B^*(r,s,t,u) \qquad \qquad (r,s,t,u) \in \mathcal P_N  \label{eq:GV}
\end{align}
form a basis for ${\rm Fix}(G)$.
\end{proposition}
\begin{proof} We first show that the vectors \eqref{eq:GV} are contained in ${\rm Fix}(G)$.
Pick $g \in G$ and $(r,s,t,u) \in \mathcal P_N$. Let $(h,i,j)$ satisfy   \eqref{eq:hij}.
Using Lemma  \ref{lem:gAcom}(ii) and  \eqref{eq:gVVV} and Definitions \ref{def:Qhij}, \ref{lem:Bs}  we obtain
\begin{align*}
g \bigl(B^*(r,s,t,u)\bigr) &= g \bigl( Q_{h,i,j} \bigr) \\
                 &=  g \Biggl(2^N \sum_{x \in X}  {\sf E}_h x \otimes {\sf E}_i x \otimes {\sf E}_j x \Biggr) \\
                        &= 2^N   \sum_{x \in X}  g({\sf E}_h x) \otimes g({\sf E}_i x) \otimes g({\sf E}_j x) \\
                          &= 2^N   \sum_{x \in X}  {\sf E}_h g(x) \otimes {\sf E}_i g(x) \otimes {\sf E}_j g(x) \\
                           &= 2^N \sum_{y \in X}  {\sf E}_h y \otimes {\sf E}_i y \otimes {\sf E}_j y \\ 
                           &= B^*(r,s,t,u).
                          \end{align*}
Therefore $B^*(r,s,t,u) \in  {\rm Fix}(G)$.
We have shown that the vectors \eqref{eq:GV} are contained in ${\rm Fix}(G)$.
The vectors \eqref{eq:GV} are linearly independent by Lemma  \ref{lem:BsOrth}. The result follows in view of
 \eqref{eq:PNsize} and Corollary  \ref{cor:FixGdim}.
\end{proof}

\noindent Next, we describe the basis for ${\rm Fix}(G)$ that is dual to the one in Proposition \ref{lem:BsBasis},
with respect to the Hermitian form $\langle \,,\,\rangle$.

\begin{definition} \label{lem:Bsdual} \rm For a profile $(r,s,t,u) \in \mathcal P_N$ we define a vector
\begin{align*}
\tilde B^*(r,s,t,u) = \frac{r!s!t!u!}{N! 2^N} B^*(r,s,t,u).
\end{align*}
\end{definition}

\begin{proposition} \label{lem:FixBSasisd} The  vectors
\begin{align*} 
\tilde B^*(r,s,t,u) \qquad \qquad (r,s,t,u) \in \mathcal P_N 
\end{align*}
form a  basis for ${\rm Fix}(G)$.
\end{proposition}
\begin{proof} By Proposition \ref{lem:BsBasis} and Definition \ref{lem:Bsdual}.
\end{proof}

\begin{lemma} \label{lem:FixBSdual} The ${\rm Fix}(G)$-basis
\begin{align*}
B^*(r,s,t,u) \qquad \qquad (r,s,t,u) \in \mathcal P_N 
\end{align*}
and the 
${\rm Fix}(G)$-basis
\begin{align*}
\tilde B^*(r,s,t,u) \qquad \qquad (r,s,t,u) \in \mathcal P_N 
\end{align*}
are dual with respect to $\langle \,,\,\rangle$.
\end{lemma}
\begin{proof} By Lemma  \ref{lem:BsOrth} and Definition \ref{lem:Bsdual}.
\end{proof}

\noindent The following result is a variation on  \cite[Lemma~9.18]{S3}.
\begin{lemma} \label{lem:SumBs}  We have
\begin{align*}
B(N,0,0,0)= 2^{-N} \sum_{(r,s,t,u) \in \mathcal P_N} B^*(r,s,t,u).
\end{align*}
\end{lemma}
\begin{proof} We have
\begin{align*}
B(N,0,0,0) &= \sum_{x \in X} x \otimes x \otimes x \qquad \qquad \qquad \qquad \quad\qquad\qquad \hbox{\rm by Example \ref{ex:BN000}} \\
                & = \sum_{x \in X} \Biggl(\sum_{h=0}^N {\sf E}_h x \Biggr) \otimes \Biggl( \sum_{i=0}^N {\sf E}_i x \Biggr) \otimes \Biggl( \sum_{j=0}^N {\sf E}_j x \Biggr) \qquad \hbox{\rm by $I= \sum_{\ell=0}^N {\sf E}_\ell$} \\
                & = \sum_{x \in X} \sum_{h=0}^N \sum_{i=0}^N \sum_{j=0}^N {\sf E}_h x \otimes {\sf E}_i x \otimes {\sf E}_j x \\
                & = \sum_{h=0}^N \sum_{i=0}^N \sum_{j=0}^N \sum_{x \in X} {\sf E}_h x \otimes {\sf E}_i x \otimes {\sf E}_j x \\
                &=  2^{-N} \sum_{h=0}^N \sum_{i=0}^N \sum_{j=0}^N Q_{h,i,j} \qquad \qquad \quad \; \;\qquad \qquad \hbox{\rm by Definition  \ref{def:Qhij}} \\
                & = 2^{-N} \sum_{(h,i,j) \in \mathcal P''_N} Q_{h,i,j} \qquad \qquad \quad \qquad \qquad \qquad \hbox{\rm by Lemma \ref{lem:Qcondition}(ii) }\\
                &= 2^{-N} \sum_{(r,s,t,u) \in \mathcal P_N} B^*(r,s,t,u) \qquad  \quad  \hbox{\rm by Lem. \ref{lem:PPbij} and Def. \ref{lem:Bs}}.
\end{align*}
\end{proof}

\begin{lemma} \label{lem:AactBs} For a profile $(r,s,t,u) \in \mathcal P_N$ the following {\rm (i)--(iii)} hold:
\begin{enumerate}
\item[\rm (i)] 
$A^{(1)} B^*(r,s,t,u) = (r+s-t-u) B^*(r,s,t,u)$;
\item[\rm (ii)] 
$A^{(2)} B^*(r,s,t,u) = (r-s+t-u) B^*(r,s,t,u)$;
\item[\rm (iii)] 
$A^{(3)} B^*(r,s,t,u) = (r-s-t+u) B^*(r,s,t,u)$.
\end{enumerate}
\end{lemma}
\begin{proof} Use  \eqref{eq:eigvalHN2} and Lemma \ref{lem:Atensor}  along with Definitions \ref{def:Qhij}, \ref{lem:Bs}.
\end{proof}

\noindent Recall the map  $\ddag: P_N \to {\rm Fix}(G)$ from Lemma \ref{lem:vee}.

\begin{proposition} \label{prop:Main2} The map $\ddag: P_N \to {\rm Fix}(G)$ sends
\begin{align*}
x^{*r} y^{*s} z^{*t} w^{*u} \mapsto \bigl(N! 2^N\bigr)^{1/2} \tilde B^*(r,s,t,u) \qquad \qquad (r,s,t,u) \in \mathcal P_N.
\end{align*}
\end{proposition}
\begin{proof} For each profile $(r,s,t,u) \in \mathcal P_N$ we define a vector
\begin{align*}
\Delta(r,s,t,u) = \bigl(x^{*r} y^{*s} z^{*t} w^{*u}\bigr)^\ddag - \bigl(N! 2^N\bigr)^{1/2} \tilde B^*(r,s,t,u).
\end{align*}
We show that $\Delta(r,s,t,u)=0$. By Proposition  \ref{lem:dbAction}(i)--(iii), Theorem \ref{prop:GenAction}, Definition  \ref{lem:Bsdual}, and Lemma \ref{lem:AactBs}, 
\begin{align*}
A^{(1)} \Delta(r,s,t,u) &= (r+s-t-u) \Delta(r,s,t,u), \\
A^{(2)} \Delta(r,s,t,u) &= (r-s+t-u) \Delta(r,s,t,u), \\
A^{(3)} \Delta(r,s,t,u) &= (r-s-t+u) \Delta(r,s,t,u).
\end{align*}
Therefore, $\Delta(r,s,t,u)$ is contained in an $\mathbb H$-weight space of the $\mathfrak{sl}_4(\mathbb C)$-module ${\rm Fix}(G)$.
The vector space ${\rm Fix}(G)$ is the direct sum of its $\mathbb H$-weight spaces. Therefore,
the nonzero vectors among 
\begin{align*}
\Delta(r,s,t,u) \qquad \qquad (r,s,t,u) \in \mathcal P_N
\end{align*}
are linearly independent.
To finish the proof, it suffices to show that
\begin{align} \label{eq:toshowDelta}
0 = \sum_{(r,s,t,u) \in \mathcal P_N} \frac{\Delta(r,s,t,u)}{r!s!t!u!}.
\end{align}
\noindent We have
\begin{align*}
\sum_{(r,s,t,u) \in \mathcal P_N} \frac{ \bigl(x^{*r} y^{*s} z^{*t} w^{*u}\bigr)^\ddag}{r!s!t!u!} &= \Biggl( \sum_{(r,s,t,u) \in \mathcal P_N} \frac{ x^{*r} y^{*s} z^{*t} w^{*u}}{r!s!t!u!} \Biggr)^\ddag \\
&= \frac{2^N}{N!} \bigl( x^N \bigr)^\ddag  \qquad \qquad \qquad \qquad \qquad\hbox{\rm by Lemma \ref{lem:motivated}} \\
&= \frac{2^N}{N!} \bigl(N! 2^N\bigr)^{1/2} \tilde B(N,0,0,0) \qquad \qquad \hbox{\rm by Lemma \ref{lem:vee}} \\
&=  \frac{ \bigl(N! 2^N\bigr)^{1/2}}{N!} B(N,0,0,0) \qquad \qquad \hbox{\rm by Definition  \ref{lem:Bdual}}.
\end{align*}
\noindent We also have
\begin{align*}
\sum_{(r,s,t,u) \in \mathcal P_N} \frac{\tilde B^*(r,s,t,u)}{r!s!t!u!} &= \frac{1}{N! 2^N} \sum_{(r,s,t,u) \in \mathcal P_N} B^*(r,s,t,u) \qquad \hbox{\rm by Definition  \ref{lem:Bsdual}} \\
&= \frac{B(N,0,0,0)}{N!} \qquad \qquad \qquad \qquad \hbox{\rm by Lemma \ref{lem:SumBs}.}
\end{align*}
We may now argue
\begin{align*}
\sum_{(r,s,t,u) \in \mathcal P_N} \frac{\Delta(r,s,t,u)}{r!s!t!u!} &= \sum_{(r,s,t,u) \in \mathcal P_N} \frac{ \bigl(x^{*r} y^{*s} z^{*t} w^{*u}\bigr)^\ddag}{r!s!t!u!}
           -  \bigl(N! 2^N\bigr)^{1/2} \sum_{(r,s,t,u) \in \mathcal P_N} \frac{\tilde B^*(r,s,t,u)}{r!s!t!u!} \\
&=\frac{ \bigl( N! 2^N \bigr)^{1/2}}{N!} B(N,0,0,0) -  \frac{ \bigl( N! 2^N \bigr)^{1/2}}{N!} B(N,0,0,0) \\
 &=0.
\end{align*}
We have shown \eqref{eq:toshowDelta}, and the result follows.
\end{proof}

\noindent In the next four results, we give some comments and tie up some loose ends.

\begin{proposition} \label{lem:HAactionCd} For a profile $(r,s,t,u) \in \mathcal P_N$ the following {\rm (i)--(vi)} hold:
\begin{enumerate}
\item[\rm (i)] $A^{(1)} \tilde B^*(r,s,t,u) = (r+s-t-u) \tilde B^*(r,s,t,u)$;
\item[\rm (ii)] $A^{(2)} \tilde B^*(r,s,t,u) = (r-s+t-u) \tilde B^*(r,s,t,u)$;
\item[\rm (iii)] $A^{(3)} \tilde B^*(r,s,t,u)  = (r-s-t+u) \tilde B^*(r,s,t,u)$;
\item[\rm (iv)] the vector
\begin{align*}
A^{*(1)} \tilde B^*(r,s,t,u)
\end{align*}
 is a linear combination with the following terms and coefficients:
\begin{align*} 
\begin{tabular}[t]{c|c}
{\rm Term }& {\rm Coefficient} 
 \\
 \hline
   $ \tilde B^*(r-1,s+1,t,u)$& $r$   \\
 $\tilde B^*(r+1,s-1,t,u)$   & $s$\\
  $ \tilde B^*(r,s,t-1,u+1)$  & $t$ \\
  $ \tilde B^*(r,s,t+1,u-1) $& $u$
   \end{tabular}
\end{align*}
\item[\rm (v)] 
the vector
\begin{align*}
A^{*(2)} \tilde B^*(r,s,t,u) 
\end{align*}
 is a linear combination with the following terms and coefficients:
\begin{align*} 
\begin{tabular}[t]{c|c}
{\rm Term }& {\rm Coefficient} 
 \\
 \hline
   $ \tilde B^*(r-1,s,t+1,u) $& $ r$   \\
 $ \tilde B^*(r,s-1,t,u+1) $   & $s$\\
  $ \tilde B^*(r+1,s,t-1,u)$  & $t$ \\
  $ \tilde B^*(r,s+1,t,u-1) $& $u $
   \end{tabular}
\end{align*}
\item[\rm (vi)] 
the vector
\begin{align*}
A^{*(3)} \tilde B^*(r,s,t,u)
\end{align*}
 is a linear combination with the following terms and coefficients:
\begin{align*} 
\begin{tabular}[t]{c|c}
{\rm Term }& {\rm Coefficient} 
 \\
 \hline
    $ \tilde B^*(r-1,s,t,u+1)$& $r $   \\
 $ \tilde B^*(r,s-1,t+1,u)  $   & $s $\\
    $ \tilde B^*(r,s+1,t-1,u) $  & $t$ \\
  $ \tilde  B^*(r+1,s,t,u-1) $& $u $
   \end{tabular}
\end{align*}
\end{enumerate}
\end{proposition}
\begin{proof} Apply the map $\ddag $ to everything in Proposition  \ref{lem:dbAction},
and evaluate the result using Theorem \ref{prop:GenAction} and Proposition \ref{prop:Main2}.
\end{proof}

\begin{lemma} \label{lem:dualSumBs}  We have
\begin{align*}
B^*(N,0,0,0)= 2^{-N} \sum_{(r,s,t,u) \in \mathcal P_N} B(r,s,t,u).
\end{align*}
\end{lemma}
\begin{proof}  Apply the map $\ddag$ to everything in Lemma  \ref{lem:xnExpand}.
In the resulting equation,
evaluate the left-hand side using Definition  \ref{lem:Bsdual} and Proposition \ref{prop:Main2}.
Evaluate the right-hand side using Definition \ref{lem:Bdual} and Lemma \ref{lem:vee}. The result follows.
\end{proof}

\begin{lemma} \label{lem:FGWSdim1dual}  For the $\mathfrak{sl}_4(\mathbb C)$-module ${\rm Fix}(G)$, each $\mathbb H$-weight space has dimension one
and each $\mathbb H^*$-weight space has dimension one.
\end{lemma}
\begin{proof} By Lemmas  \ref{lem:WSdim1}, \ref{lem:WSdim1dual} and since the $\mathfrak{sl}_4(\mathbb C)$-modules $P_N$, ${\rm Fix}(G)$ are isomorphic.
\end{proof}

\begin{lemma} \label{lem:invDual} For the $\mathfrak{sl}_4(\mathbb C)$-module ${\rm Fix}(G)$, the 
following holds for
$i \in \lbrace 1,2,3 \rbrace$:
\begin{align*}
\langle A_i u, v \rangle = \langle u, A_i v\rangle,  \qquad \quad \langle A^*_i u, v \rangle = \langle u, A^*_i v\rangle \qquad \quad u,v \in {\rm Fix}(G).
\end{align*}
\end{lemma}
\begin{proof} Apply the map $\ddag$ to everything in Lemma \ref{lem:inv}, and evaluate the result using Theorem \ref{lem:HermComp} and the fact that
$\ddag$ is an isomorphism of $\mathfrak{sl}_4(\mathbb C)$-modules.
\end{proof}

\noindent  We comment about notation. Earlier in this section, we discussed the vectors 
\begin{align} \label{eq:BBs}
B(r,s,t,u), \qquad B^*(r,s,t,u) \qquad \qquad (r,s,t,u) \in \mathcal P_N.
\end{align}
Up to notation, the vectors \eqref{eq:BBs}  are the same as the vectors
\begin{align} \label{eq:PQ}
P_{h,i,j}, \qquad Q_{h,i,j} \qquad \qquad (h,i,j) \in \mathcal P''_N.
\end{align}
\noindent  In the next section, we will adopt a point of view in which the notation \eqref{eq:PQ} is more convenient
than the notation \eqref{eq:BBs}. To prepare for the next section, we restate a few results using the notation \eqref{eq:PQ}.

\begin{corollary} \label{cor:Phij} The following {\rm (i)--(iv)} hold.
\begin{enumerate}
\item[\rm (i)] The vectors
\begin{align*}
P_{h,i,j} \qquad \qquad (h,i,j) \in \mathcal P''_N
\end{align*}
form an orthogonal basis for ${\rm Fix}(G)$.
\item[\rm (ii)] For $(h,i,j) \in \mathcal P''_N$,
\begin{align*}
&A^{*(1)} \bigl(  P_{h,i,j}   \bigr) = \theta^*_h P_{h,i,j}, \qquad \qquad
A^{*(2)} \bigl(  P_{h,i,j}   \bigr) = \theta^*_i P_{h,i,j}, \\
&A^{*(3)} \bigl(  P_{h,i,j}   \bigr) = \theta^*_j P_{h,i,j}.
\end{align*}
\item[\rm (iii)] The vectors
\begin{align*}
Q_{h,i,j} \qquad \qquad (h,i,j) \in \mathcal P''_N
\end{align*}
form an orthogonal basis for ${\rm Fix}(G)$.
\item[\rm (iv)] For $(h,i,j) \in \mathcal P''_N$,
\begin{align*}
&A^{(1)} \bigl(  Q_{h,i,j}   \bigr) = \theta_h Q_{h,i,j}, \qquad \qquad
A^{(2)} \bigl(  Q_{h,i,j}   \bigr) = \theta_i Q_{h,i,j}, \\
&A^{(3)} \bigl(  Q_{h,i,j}   \bigr) = \theta_j Q_{h,i,j}.
\end{align*}
\end{enumerate}
\end{corollary}
\begin{proof} (i) By Lemma \ref{lem:PPbij},  Remark \ref{def:Phij},  Lemma  \ref{lem:orthogB}, and Proposition \ref{lem:FixBasis}. \\
\noindent (ii) By Remark \ref{def:Phij} and Definition \ref{def:mapsAAAs}. \\
\noindent (iii) By Lemma \ref{lem:PPbij}, Definition \ref{lem:Bs}, Lemma \ref{lem:BsOrth}, and Proposition \ref{lem:BsBasis}. \\
\noindent (iv) By Lemma \ref{lem:Atensor}  and Definition \ref{def:Qhij}. 
\end{proof}

\section{The subconstituent algebra of $H(N,2)$}

\noindent 
We continue to discuss the hypercube $H(N,2)$. In \cite{go} the subconstituent algebra of $H(N,2)$ is described in detail.
In this section, we explain what  the subconstituent algebra of $H(N,2)$ has to do with our results from the previous sections.
\medskip

\noindent We now review some concepts and notation about $\Gamma=H(N,2)$.
\begin{definition}\label{def:exy} \rm For $x,y \in X$ we define a map $e_{x,y} \in {\rm End}(V)$ that sends $y \mapsto x$ and all other vertices to $0$. Note
that $\lbrace e_{x,y} \rbrace_{x,y \in X}$ form a basis for  ${\rm End}(V)$.
\end{definition}

\begin{definition}\rm We endow ${\rm End}(V)$ with a Hermitian form $\langle \,,\,\rangle $ with respect to which the basis $\lbrace e_{x,y} \rbrace_{x,y \in X}$ is
orthonormal.
\end{definition}

\noindent For $x \in X$ and $0 \leq i \leq N$, define the set $\Gamma_i(x)=\lbrace y \in X \vert \partial(x,y)=i\rbrace$.
 For $0 \leq i \leq N$ define ${\sf A}_i \in {\rm End}(V)$ such that
\begin{align*}
{\sf A}_i x = \sum_{\xi \in \Gamma_i(x)} \xi, \qquad \qquad x \in X.
\end{align*}
By \cite[Section~3]{int} we have ${\sf A}_i = \binom{N}{i} f_i({\sf A})$, where $f_i$ is from Definition \ref{def:fi}.
For the rest of this section, fix $\varkappa \in X$. Define ${\sf A}^* = {\sf A}^*(\varkappa) \in {\rm End}(V)$ such that
\begin{align}
{\sf A}^* x = \theta^*_{\partial(x, \varkappa)} x, \qquad \qquad x \in X. \label{eq:DualAdj}
\end{align}
By construction, the map ${\sf A}^*$ is diagonalizable with eigenvalues $\lbrace \theta^*_i \rbrace_{i=0}^N$.
By \cite[Theorem~4.2]{go} we have
\begin{align*}
\lbrack {\sf A}, \lbrack {\sf A}, {\sf A}^* \rbrack \rbrack = 4 {\sf A}^*, \qquad \qquad
\lbrack {\sf A}^*, \lbrack {\sf A}^*, {\sf A}  \rbrack \rbrack = 4 {\sf A}.
\end{align*}
\noindent Let $T=T(\varkappa)$ denote the subalgebra of ${\rm End}(V)$ generated by $\sf A, A^*$.  By \cite[Corollary~14.15]{go} we have ${\rm dim}\, T = \binom{N+3}{3}$.
We call $T$ the {\it subconstituent algebra} (or {\it Terwilliger algebra}) of
$H(N,2)$ with respect to $\varkappa$; see \cite[Definition~2.1]{go}.
\medskip

\noindent We mention some bases for the vector space $T$. For $0 \leq i \leq N$ define ${\sf E}^*_i = {\sf E}^*_i(\varkappa) \in {\rm End}(V)$ such that
\begin{align*}
{\sf E}^*_i x =  \begin{cases} x, & {\mbox{\rm if $\partial(x,\varkappa)=i$}}; \\
0, & {\mbox{\rm if $\partial(x,\varkappa)\not= i$}},
\end{cases} 
\qquad \qquad x \in X.
\end{align*}
By construction, ${\sf E}^*_i$ is the primitive idempotent of ${\sf A}^*$ for the eigenvalue $\theta^*_i$.
For $0 \leq i \leq N$  define ${\sf A}^*_i = {\sf A}^*_i(\varkappa) \in {\rm End}(V)$ such that
\begin{align*}
{\sf A}^*_i x = 2^N  \langle  {\sf E}_i \varkappa,  x \rangle x, \qquad \qquad x \in X.
\end{align*}
By \cite[Section~11]{int} we have ${\sf A}^*_i = \binom{N}{i} f_i({\sf A}^*) $.
By \cite[p.~403]{go} and Remark \ref{rem:IntKrein}, the following hold for $0 \leq h,i,j\leq N$:
\begin{align*}
{\sf E}^*_i {\sf A}_h {\sf E}^*_j  &\not=0 \quad \hbox{ \rm iff} \quad (h,i,j) \in \mathcal P''_N;      \\
{\sf E}_i {\sf A}^*_h {\sf E}_j &\not=0 \quad \hbox{ \rm iff} \quad (h,i,j) \in \mathcal P''_N.
\end{align*}
\begin{lemma} \label{lem:Tbasis} 
The following is an orthogonal basis for the vector space $T$:
\begin{align}
{\sf E}^*_i {\sf A}_h {\sf E}^*_j \qquad \qquad (h,i,j) \in \mathcal P''_N. \label{eq:TB1}
\end{align}
Moreover, the following is an orthogonal basis for the vector space $T$:
\begin{align}
{\sf E}_i {\sf A}^*_h {\sf E}_j \qquad \qquad (h,i,j) \in \mathcal P''_N. \label{eq:TB2}
\end{align}
\end{lemma}
\begin{proof} By \cite[Lemma~8.1]{int} and Remark \ref{rem:IntKrein}  and
\begin{align*}
{\rm dim}\, T = \binom{N+3}{3} = \vert \mathcal P_N \vert = \vert \mathcal P''_N \vert.
\end{align*}
\end{proof}

\noindent  By \cite[p.~404]{go} the algebra $T$ is semisimple.
\medskip

\noindent Next, we review the Wedderburn decomposition of $T$. By \cite[Corollary~14.12]{go} the center of $T$ is generated by
\begin{align}
\phi = \frac{  4 {\sf A}^2 + 4 {\sf A}^{*2} - ({\sf A A^*-A^* A})^2 }{8}. \label{eq:phiWedderburn}
\end{align}
By \cite[Theorem~6.3 and Lemma~14.6]{go},  the map $\phi$ is diagonalizable with eigenvalues
\begin{align*}
\frac{(N-2\ell)(N-2\ell+2)}{2} \qquad \qquad 0 \leq \ell \leq \lfloor N/2 \rfloor.
\end{align*}
For $0 \leq \ell \leq \lfloor N/2 \rfloor$ let $\phi_\ell \in {\rm End}(V)$ denote the primitive idempotent of $\phi$ associated with the eigenvalue
\begin{align*}
\frac{(N-2\ell)(N-2 \ell+2)}{2}.
\end{align*}
By \cite[Theorem~14.10]{go} the following is a basis for the center of $T$:
\begin{align*}
\phi_\ell, \qquad \qquad 0 \leq \ell \leq \lfloor N/2 \rfloor.
\end{align*}
By \cite[Corollary~14.9]{go},
\begin{align}
\phi = \sum_{\ell=0}^{\lfloor N/2 \rfloor} \frac{(N-2\ell)(N-2\ell+2)}{2} \phi_\ell.       \label{eq: PhiPhi}
\end{align}
By \cite[Theorem~14.14]{go}, for  $0 \leq \ell \leq \lfloor N/2 \rfloor$  the subspace $ \phi_\ell T$ is a minimal 2-sided ideal of $T$ with dimension $(N-2\ell+1)^2$.
By \cite[Theorems~14.10, 14.14]{go} we have
\begin{align} \label{eq:Wedderburn}
T = \sum_{\ell=0}^{\lfloor N/2 \rfloor}  \phi_\ell T \qquad \quad \hbox{\rm (orthogonal direct sum).}
\end{align}
This is the Wedderburn decomposition of $T$.
\medskip

\noindent Next, we define some maps in ${\rm End}(T)$.
\begin{definition} \label{def:NewAAA} \rm We define ${\mathcal A}^{(1)}, {\mathcal A}^{(2)}, {\mathcal A}^{(3)} \in {\rm End}(T)$ such that for $(h,i,j) \in \mathcal P''_N$,
\begin{align*}
{\mathcal A}^{(1)} \bigl( {\sf E}_i {\sf A}^*_h {\sf E}_j \bigr) &= \theta_h {\sf E}_i {\sf A}^*_h {\sf E}_j, \\
{\mathcal A}^{(2)} \bigl( {\sf E}_i {\sf A}^*_h {\sf E}_j \bigr) &= \theta_i {\sf E}_i {\sf A}^*_h {\sf E}_j, \\
{\mathcal A}^{(3)} \bigl( {\sf E}_i {\sf A}^*_h {\sf E}_j \bigr) &= \theta_j {\sf E}_i {\sf A}^*_h {\sf E}_j.
\end{align*}
\end{definition} 
\noindent The maps ${\mathcal A}^{(2)}, {\mathcal A}^{(3)}$ have the following interpretation.
\begin{lemma} \label{lem:NewAAALR} For $B \in T$,
\begin{align*}
{\mathcal A}^{(2)}(B) = {\sf A}B, \qquad \qquad {\mathcal A}^{(3)}(B) = B{\sf A}.
\end{align*}
\end{lemma}
\begin{proof} By the second assertion in Lemma \ref{lem:Tbasis}, along with Definition \ref{def:NewAAA}.
\end{proof} 

\begin{definition} \label{def:NewAsAsAs} \rm We define ${\mathcal A}^{*(1)}, {\mathcal A}^{*(2)}, {\mathcal A}^{*(3)} \in {\rm End}(T)$ such that for $(h,i,j) \in \mathcal P''_N$,
\begin{align*}
{\mathcal A}^{*(1)} \bigl( {\sf E}^*_i {\sf A}_h {\sf E}^*_j \bigr) &= \theta^*_h {\sf E}^*_i {\sf A}_h {\sf E}^*_j, \\
{\mathcal A}^{*(2)} \bigl( {\sf E}^*_i {\sf A}_h {\sf E}^*_j \bigr) &= \theta^*_j {\sf E}^*_i {\sf A}_h {\sf E}^*_j, \\
{\mathcal A}^{*(3)} \bigl( {\sf E}^*_i {\sf A}_h {\sf E}^*_j \bigr) &= \theta^*_i {\sf E}^*_i {\sf A}_h {\sf E}^*_j.
\end{align*}
\end{definition} 

\noindent The maps ${\mathcal A}^{*(2)}, {\mathcal A}^{*(3)}$ have the following interpretation.

\begin{lemma} \label{lem:NewAsAsAsLR} For $B \in T$,
\begin{align*}
{\mathcal A}^{*(2)}(B) = B{\sf A}^*, \qquad \qquad {\mathcal A}^{*(3)}(B) = {\sf A}^* B.
\end{align*}
\end{lemma}
\begin{proof} By the first assertion in Lemma \ref{lem:Tbasis}, along with Definition \ref{def:NewAsAsAs}.
\end{proof} 

\noindent We are going to show that the vector space $T$ becomes an $\mathfrak{sl}_4(\mathbb C)$-module  on which $A_i = {\mathcal A}^{(i)} $ and $A^*_i ={\mathcal A}^{*(i)}$
for $i \in \lbrace 1,2,3 \rbrace$. We are also going to show that the $\mathfrak{sl}_4(\mathbb C)$-modules  $T$ and $P_N$ are isomorphic. In addition, we will interpret
the Wedderburn decomposition of $T$ in terms of the decomposition of $P_N$ given in \eqref{eq:odsREMIND} with $i=1$.
\medskip

\noindent Recall from Definition \ref{def:V3} the vector space $V^{\otimes 3}$ and the set $X^{\otimes 3}$.

\begin{definition}\rm We define a $\mathbb C$-linear map $\varepsilon: V^{\otimes 3} \to {\rm End}(V)$ as follows. For $x \otimes y \otimes z \in X^{\otimes 3}$,
\begin{align*}
\varepsilon(x \otimes y \otimes z) =  \begin{cases} 2^{N/2}e_{y,z}, & {\mbox{\rm if $x=\varkappa$}}; \\
0, & {\mbox{\rm if $x \not=\varkappa$}}.
\end{cases} 
\end{align*}
\end{definition}

\begin{lemma} \label{lem:epsActP} For $x \otimes y \otimes z \in X^{\otimes 3}$ the map  $\varepsilon(x \otimes y \otimes z)$ sends
\begin{align*}
\psi \mapsto 2^{N/2} \langle x, \varkappa \rangle  \langle z,\psi \rangle  y, \qquad \qquad \psi \in X.
\end{align*}
\end{lemma}
\begin{proof} For $\psi \in X$ we have
\begin{align*}
\varepsilon(x \otimes y \otimes z) (\psi) &= 2^{N/2} \delta_{x,\varkappa} e_{y,z}(\psi) 
= 2^{N/2} \delta_{x,\varkappa} \delta_{\psi,z} y 
= 2^{N/2} \langle x, \varkappa \rangle \langle z, \psi \rangle y.
\end{align*}
\end{proof}

\begin{lemma} \label{lem:epsAction} For $u,v,w \in V$ the map $\varepsilon(u \otimes v \otimes w)$ sends
\begin{align*}
\psi \mapsto 2^{N/2} \langle u, \varkappa \rangle  \langle w, \overline \psi  \rangle  v, \qquad \qquad \psi \in V.
\end{align*}
\end{lemma}
\begin{proof}  By Lemma  \ref{lem:epsActP} and $\mathbb C$-linearity in each of the arguments $u,v,w,\psi$.
\end{proof}

\noindent Recall the vectors $P_{h,i,j}$ from Remark  \ref{def:Phij}.
\begin{lemma} \label{lem:epsP} For $(h,i,j) \in \mathcal P''_N$,
\begin{align}
\varepsilon(P_{h,i,j}) = 2^{N/2} {\sf E}^*_j {\sf A}_h {\sf E}^*_i. \label{eq:maps2}
\end{align}
\end{lemma}
\begin{proof} To verify \eqref{eq:maps2}, we apply each side to a vertex
 $\psi \in X$. First assume that $\partial(\psi, \varkappa) = i$. Then each side of \eqref{eq:maps2} sends
 \begin{align*}
 \psi \mapsto 2^{N/2} \sum_{y \in \Gamma_j(\varkappa) \cap \Gamma_h(\psi)} y.
 \end{align*}
 Next assume that $\partial(\psi, \varkappa) \not= i$. Then each side of \eqref{eq:maps2} sends $\psi \mapsto 0$.
 \end{proof}

\noindent We bring in some notation. 
For $u,v\in V$ we define a vector $u \circ v \in V$ as follows. Write 
\begin{align*}
u=\sum_{x \in X} u_x x, \qquad \qquad  v = \sum_{x \in X} v_x x, \qquad \qquad u_x, v_x \in \mathbb C.
\end{align*}
Define
\begin{align*}
u \circ v = \sum_{x \in X} u_x v_x x. 
\end{align*}
\noindent Note that $u_x = \langle u,x\rangle$ and $v_x = \langle v,x\rangle $ for $x \in X$. Therefore,
\begin{align*}
u \circ v = \sum_{x \in X} \langle u,x\rangle \langle v, x \rangle x.
\end{align*}

\noindent The following result is well known. We give a short proof for the sake of completeness.
\begin{lemma} \label{lem:Ash} {\rm (See \cite[Lemma~9.3]{int}.)} For $0 \leq h \leq N$ we have
\begin{align*}
{\sf A}^*_h v = 2^N {\sf E}_h \varkappa \circ v, \qquad \qquad v \in V.
\end{align*}
\end{lemma}
\begin{proof} For $v \in V$ we have
\begin{align*}
{\sf A}^*_h v &= {\sf A}^*_h \sum_{x \in X} \langle v,x \rangle x 
= \sum_{x \in X} \langle v,x\rangle {\sf A}^*_h x 
= 2^N \sum_{x \in X} \langle v,x \rangle \langle {\sf E}_h \varkappa, x\rangle x \\
&= 2^N \sum_{x \in X} \langle {\sf E}_h \varkappa, x \rangle \langle v,x\rangle x 
= 2^N {\sf E}_h \varkappa \circ v.
\end{align*}
\end{proof}

\noindent Recall the vectors $Q_{h,i,j}$ from Definition  \ref{def:Qhij}.
\begin{lemma} \label{lem:epsQ} For $(h,i,j) \in \mathcal P''_N$,
\begin{align}
\varepsilon(Q_{h,i,j}) =2^{N/2} {\sf E}_i {\sf A}^*_h {\sf E}_j. \label{eq:maps3}
\end{align}
\end{lemma}
\begin{proof} Let $\psi \in X$. Using in order Definition  \ref{def:Qhij},  Lemma \ref{lem:epsAction}, Lemma \ref{lem:Ebil}, Lemma \ref{lem:Ash}  we obtain
\begin{align*}
\varepsilon(Q_{h,i,j})(\psi) &= 2^N \sum_{x \in X} \varepsilon \bigl ( {\sf E}_h x \otimes {\sf E}_i x \otimes {\sf E}_j x           \bigr) (\psi) \\
&= 2^{3N/2} \sum_{x \in X} \langle {\sf E}_h x, \varkappa\rangle \langle {\sf E}_j x, \overline \psi \rangle {\sf E}_i x \\
&= 2^{3N/2} \sum_{x \in X} \langle {\sf E}_h x, \varkappa\rangle \langle {\sf E}_j x,  \psi \rangle {\sf E}_i x \\
&= 2^{3N/2} {\sf E}_i \sum_{x \in X} \langle {\sf E}_h \varkappa, x \rangle \langle {\sf E}_j \psi, x \rangle x \\
&= 2^{3N/2}{\sf E}_i \bigl(  {\sf E}_h \varkappa \circ {\sf E}_j \psi     \bigr) \\
&= 2^{N/2}{\sf E}_i {\sf A}^*_h {\sf E}_j (\psi).
\end{align*}
The result follows.
\end{proof}

\begin{lemma} \label{lem:epsBIJ} The restriction $\varepsilon\vert_{{\rm Fix}(G)}$ of $\varepsilon$ to ${\rm Fix}(G)$ gives a bijection 
$ \varepsilon\vert_{{\rm Fix}(G)}:    {\rm Fix}(G) \to T$.
\end{lemma}
\begin{proof} By Corollary \ref{cor:Phij}(iii)  and  Lemmas \ref{lem:Tbasis}, \ref{lem:epsQ}.
\end{proof}

\begin{lemma} \label{lem:diagCOM}
For $k\in \lbrace 1,2,3\rbrace $ the following diagrams commute:
\begin{equation*}
{\begin{CD}
{\rm Fix}(G) @>\varepsilon >>
               T
              \\
         @V A^{(k)} VV                   @VV {\mathcal A}^{(k)} V \\
     {\rm Fix}(G)  @>>\varepsilon>
                                 T
                        \end{CD}} 
                        \qquad \qquad \quad
    {\begin{CD}
    {\rm Fix}(G) @>\varepsilon >>
               T
              \\
         @V A^{*(k)} VV                   @VV {\mathcal A}^{*(k)}V \\
        {\rm Fix}(G) @>>\varepsilon >
                                T
                        \end{CD}} 
\end{equation*}
\end{lemma}
\begin{proof}  Let $(h,i,j) \in \mathcal P''_N$. To verify the first diagram,  chase $Q_{h,i,j}$ around the diagram using
Corollary \ref{cor:Phij}(iv), Definition  \ref{def:NewAAA}, and Lemma \ref{lem:epsQ}.
To verify the second diagram,  chase $P_{h,i,j}$ around the diagram using
Corollary \ref{cor:Phij}(ii), Definition \ref{def:NewAsAsAs}, and Lemma \ref{lem:epsP}.
\end{proof}

\begin{theorem} \label{prop:TisoFix} The vector space $T$ becomes an $\mathfrak{sl}_4(\mathbb C)$-module on which
\begin{align*}
A_i = {\mathcal A}^{(i)}, \qquad \qquad A^*_i ={\mathcal A}^{*(i)} \qquad \qquad i \in \lbrace 1,2,3\rbrace.
\end{align*}
Moreover, the map $\varepsilon\vert_{{\rm Fix}(G)}: {\rm Fix}(G) \to T$ is an isomorphism of $\mathfrak{sl}_4(\mathbb C)$-modules.
\end{theorem}
\begin{proof} By Theorem \ref{prop:GenAction} and Lemmas \ref{lem:epsBIJ}, \ref{lem:diagCOM}. 
\end{proof}

\noindent Next, we compare the Hermitian forms on ${\rm Fix}(G)$ and $T$.
\begin{lemma} \label{lem:PQcomp} For $(h,i,j) \in \mathcal P''_N$,
\begin{align*}
\Vert P_{h,i,j} \Vert^2 = 2^N \Vert {\sf E}^*_j {\sf A}_h {\sf E}^*_i \Vert^2, \qquad \qquad
\Vert Q_{h,i,j} \Vert^2 = 2^N \Vert {\sf E}_i {\sf A}^*_h {\sf E}_j \Vert^2.
\end{align*}
\end{lemma}
\begin{proof} In these equations, the left-hand side is computed using Remark \ref{def:Phij},  Lemma \ref{lem:orthogB}, and Lemma \ref{lem:Qcondition}(i),
while the right-hand side is computed using Remark \ref{rem:IntKrein} and  \cite[Corollary~8.2]{int}.
\end{proof}

\begin{theorem} \label{lem:HFFT} The Hermitian forms on ${\rm Fix}(G)$ and  $T$ are related as follows:
\begin{align*}
\langle u,v \rangle = \langle \varepsilon(u), \varepsilon(v) \rangle \qquad \qquad u,v \in {\rm Fix}(G).
\end{align*}
\end{theorem}
\begin{proof} Without loss of generality, we may assume that $u,v$ are in the basis for ${\rm Fix}(G)$ from Corollary \ref{cor:Phij}(i).
First assume that $u \not=v$. Then $\langle u,v \rangle =0$ by Corollary \ref{cor:Phij}(i), and
$\langle \varepsilon(u), \varepsilon(v) \rangle =0$ by Lemmas \ref{lem:Tbasis}, \ref{lem:epsP}.
Next assume that $u=v$, and write $u=v=P_{h,i,j}$. Using Lemmas \ref{lem:epsP}, \ref{lem:PQcomp} we obtain
\begin{align*}
\Vert \varepsilon(u) \Vert^2 &= \Vert 2^{N/2} {\sf E}^*_j {\sf A}_h {\sf E}^*_i \Vert^2 = 2^N \Vert {\sf E}^*_j {\sf A}_h {\sf E}^*_i \Vert^2 = \Vert P_{h,i,j} \Vert^2 = \Vert u \Vert^2.
\end{align*}
The result follows.
\end{proof}

\noindent Recall the map $\ddag: P_N \to {\rm Fix}(G)$ from Lemma \ref{lem:vee}.
\begin{definition} \label{def:iso} \rm We define a $\mathbb C$-linear map $\vartheta: P_N \to T$ to be the composition
\begin{equation*}
\vartheta: \quad {\begin{CD}
     P_N  @>>\ddag> {\rm Fix}(G) @ >> \varepsilon > T.                           
                        \end{CD}} 
\end{equation*}
\end{definition}

\noindent  In the next result, we clarify how the map $\vartheta$ acts on $P_N$.

\begin{theorem} \label{lem:VarThetaAct}  For $(r,s,t,u) \in \mathcal P_N$ the map $\vartheta$ sends
\begin{align}
x^r y^s z^t w^u &\mapsto \frac{r!s!t!u!}{(N!)^{1/2}} {\sf E}^*_j {\sf A}_h {\sf E}^*_i,           \label{eq:VarTheta1}\\
x^{*r} y^{*s} z^{*t} w^{*u} &\mapsto \frac{r!s!t!u!}{(N!)^{1/2}} {\sf E}_i {\sf A}^*_h {\sf E}_j, \label{eq:VarTheta2}
\end{align}
where
\begin{align*}
h= t+u, \qquad \quad i = u+s, \qquad \quad j= s+t.
\end{align*}
\end{theorem}
\begin{proof} The action \eqref{eq:VarTheta1} is obtained from Remark \ref{def:Phij},  Definition \ref{lem:Bdual}, and Lemmas \ref{lem:vee}, \ref{lem:epsP}.
The action \eqref{eq:VarTheta2} is obtained from  Definitions \ref{lem:Bs}, \ref{lem:Bsdual}, Proposition \ref{prop:Main2}, and Lemma \ref{lem:epsQ}.
\end{proof}

\begin{theorem} \label{thm:main1} The map $\vartheta : P_N \to T$ is an isomorphism of $\mathfrak{sl}_4(\mathbb C)$-modules.
\end{theorem}
\begin{proof}  By Definition \ref{def:iso} and Theorems \ref{prop:GenAction},  \ref{prop:TisoFix}.
\end{proof}

\begin{theorem} \label{thm:HFFT} The Hermitian forms on $P_N$ and  $T$ are related as follows:
\begin{align*}
\langle f,g \rangle = \langle \vartheta(f), \vartheta(g) \rangle \qquad \qquad f,g \in P_N.
\end{align*}
\end{theorem}
\begin{proof}  By Theorems \ref{lem:HermComp},  \ref{lem:HFFT}, and Definition \ref{def:iso}.
\end{proof}

\begin{definition} \label{lem:transpose} \rm Let $\dagger$ denote the antiautomorphism  of ${\rm End}(V)$ that sends
$e_{x,y} \leftrightarrow e_{y,x}$ for all $x,y \in X$.  We have $\dagger^2 = {\rm id}$. We call $\dagger$ the {\it transpose map}.
\end{definition}

\begin{lemma} \label{lem:daggerCom} The map $\dagger$ fixes each of ${\sf A}, {\sf A}^*$. Moreover, $T$ is invariant under $\dagger$.
The restriction of $\dagger $ to $T$ gives an antiautomorphism of the algebra $T$.
\end{lemma}
\begin{proof} The first assertion is a routine consequence of \eqref{eq:AdjMap} and \eqref{eq:DualAdj}. The second assertion holds because $T$ is generated by ${\sf A}, {\sf A}^*$.
The third assertion holds because the map $\dagger$ is invertible.
\end{proof}

\begin{lemma} \label{lem:daggerMoves} The map $\dagger$ fixes each of ${\sf A}_i, {\sf E}_i, {\sf A}^*_i, {\sf E}^*_i$ for $0 \leq i \leq N$.
\end{lemma}
\begin{proof} Each of ${\sf A}_i, {\sf E}_i$ is a polynomial in $\sf A$. Each of ${\sf A}^*_i, {\sf E}^*_i$ is a polynomial in ${\sf A}^*$. The result follows
from these comments and Lemma \ref{lem:daggerCom}.
\end{proof}

\begin{lemma} \label{lem:duality} There exists an automorphism of the algebra $T$ that swaps ${\sf A} \leftrightarrow {\sf A}^*$.
This map swaps ${\sf A}_i \leftrightarrow {\sf A}^*_i$ for $0 \leq i \leq N$ and ${\sf E}_i \leftrightarrow {\sf E}^*_i$ for $0 \leq i \leq N$.
\end{lemma}
\begin{proof}  The automorphism is given in \cite[Section~9]{InheritHSD} and  \cite[Theorem~6.4]{InheritHSD}.
\end{proof}

\begin{definition} \label{def:SS} \rm We define a $\mathbb C$-linear map $S: T \to T$ to be the composition
\begin{equation*}
S: \quad {\begin{CD}
    T  @>>{\sf A} \leftrightarrow {\sf A}^*> T @ >> \dagger > T .                           
                        \end{CD}} 
\end{equation*}
Note that $S$ is an antiautomorphism of  $T$ such that  $S^2 = {\rm id}$.
\end{definition}

\begin{lemma} \label{lem:SMoves} The map  $S $ swaps
\begin{align*}
{\sf E}_i {\sf A}^*_h {\sf E}_j \leftrightarrow {\sf E}^*_j {\sf A}_h {\sf E}^*_i 
\end{align*}
for $(h,i,j) \in \mathcal P''_N$.
\end{lemma}
\begin{proof} By Lemmas \ref{lem:daggerCom}--\ref{lem:duality} and Definition \ref{def:SS}.
\end{proof}

\noindent Recall the automorphism $\sigma$ of $P$ from Proposition  \ref{thm:aut2}.

\begin{proposition} \label{thm:diagCOM}
The following diagram  commutes:
\begin{equation*}
{\begin{CD}
P_N @>\vartheta >> T
              \\
         @ V\sigma VV                   @ VV S V \\
     P_N @>>\vartheta>
                                 T
                        \end{CD}} 
\end{equation*}
\end{proposition}
\begin{proof}  For $(r,s,t,u) \in \mathcal P_N$ chase $x^r y^s z^t w^u$ around the diagram, using
Theorem \ref{lem:VarThetaAct}, Lemma \ref{lem:SMoves}  and the comment at the end of Section 7.
\end{proof}

\noindent Next, we consider the decomposition of $P_N$ given in \eqref{eq:odsREMIND} with $i=1$.
We compare this decomposition with the Wedderburn
decomposition of $T$ given in \eqref{eq:Wedderburn}.

\begin{theorem} \label{thm:final} For $0 \leq \ell \leq \lfloor N/2 \rfloor$ the map $\vartheta$ sends
\begin{align*}
R_1^\ell \Bigl( {\rm Ker} (L_1) \cap P_{N-2\ell} \Bigr) \mapsto  \phi_\ell T.
\end{align*}
\end{theorem}
\begin{proof}  During this proof, we will refer to  the decomposition of $P_N$ given in \eqref{eq:odsREMIND}.  Throughout the proof, we assume that $i=1$ in \eqref{eq:odsREMIND}.
In  Proposition \ref{prop:CCC1} we defined $C_1 \in {\rm End}(P)$ such that on $P$,
\begin{align} \label{eq:C11}
C_1&= \frac{4 A_2^2+4A^{*2}_3 - (A_2 A^*_3 -A^*_3 A_2)^2}{8}.
\end{align}
We have $C_1 (P_N) \subseteq P_N$ by Lemma \ref{lem:PNCinv}. 
By Proposition \ref{lem:comment}(iii),
for $0 \leq \ell \leq \lfloor N/2 \rfloor$ the $\ell$-summand in   \eqref{eq:odsREMIND} is an eigenspace
for the action of $C_1$ on $P_N$; the eigenvalue is $(N-2\ell)(N-2\ell+2)/2$. We now bring in the $\mathfrak{sl}_4(\mathbb C)$-module $T$.
Define $C^{(1)} \in {\rm End}(T)$ by
\begin{align}
C^{(1)} =  \frac{4 \bigl({\mathcal A}^{(2)}\bigr)^2+4 \bigl({\mathcal A}^{*(3)} \bigr)^2 - \bigl({\mathcal A}^{(2)} {\mathcal A}^{*(3)} -{\mathcal A}^{*(3)} {\mathcal A}^{(2)} \bigr)^2}{8},   \label{eq:C12}
\end{align}
where ${\mathcal A}^{(2)}$, ${\mathcal A}^{*(3)}$ are from Definitions \ref{def:NewAAA}, \ref{def:NewAsAsAs}.
Comparing \eqref{eq:C11}, \eqref{eq:C12} and using Theorem  \ref{thm:main1},
we see that the following diagram commutes:
\begin{equation*}
{\begin{CD}
P_N @>\vartheta >>
               T
              \\
         @ V C_1 VV                   @VV C^{(1)}V \\
     P_N  @>>\vartheta>
                                 T
                        \end{CD}} 
\end{equation*}
 Recall the central element $\phi \in T$ from \eqref{eq:phiWedderburn}. By Lemmas  \ref{lem:NewAAALR},  \ref{lem:NewAsAsAsLR} and \eqref{eq:C12},
\begin{align}
C^{(1)} (B) = \phi B, \qquad \qquad B \in T.    \label{eq:C1Act}
\end{align} 
Consider the Wedderburn decomposition of $T$ from \eqref{eq:Wedderburn}. We claim that for $0 \leq \ell \leq \lfloor N/2 \rfloor$
the $\ell$-summand in \eqref{eq:Wedderburn} is an eigenspace for $C^{(1)}$ with eigenvalue $(N-2\ell)(N-2\ell+2)/2$.
To prove the claim, let $\ell$ be given. The $\ell$-summand in \eqref{eq:Wedderburn} is equal to $\phi_\ell T$. Using \eqref{eq: PhiPhi},  \eqref{eq:C1Act} we find that for  $B \in T$,
\begin{align*}
C^{(1)} (\phi_\ell B) = \phi \phi_\ell B = \frac{(N-2\ell)(N-2\ell+2)}{2} \phi_\ell B.
\end{align*}
The claim in proven.
Our discussion shows  that  for  $0 \leq \ell \leq \lfloor N/2 \rfloor$ the map $\vartheta$ sends the $\ell$-summand in 
 \eqref{eq:odsREMIND} to the $\ell$-summand in \eqref{eq:Wedderburn}. The result follows.

\end{proof}
\section{Directions for future research}

\noindent In this section, we give some suggestions for future research.

\begin{problem}\label{prob:prob1} \rm Recall the polynomial algebra $P= \mathbb C\lbrack x,y,z,w \rbrack$. Define
\begin{align*}
x^{\downarrow} &= \frac{x-y-z-w}{2}, \qquad \qquad 
y^{\downarrow} = \frac{-x+y-z-w}{2}, \\
z^{\downarrow} &= \frac{-x-y+z-w}{2}, \qquad \qquad
w^{\downarrow} = \frac{-x-y-z+w}{2}.
\end{align*}
The vectors $x^{\downarrow}, y^{\downarrow}, z^{\downarrow}, w^{\downarrow}$ form a basis for $P_1$.
Consequently, the following vectors form a basis for $P$:
\begin{align} \label{eq:PbasisDown}
x^{\downarrow r} y^{\downarrow s} z^{\downarrow t} w^{\downarrow u}, \qquad \qquad r,s,t,u \in \mathbb N.
\end{align}
This basis is an eigenbasis for the automorphism $\sigma$ of $P$ from Proposition \ref{thm:aut2}.
To see this, note that
\begin{align*}
x^{\downarrow} &=-\, \frac{x^*-y^*-z^*-w^*}{2}, \qquad \qquad 
y^{\downarrow} = \frac{-x^*+y^*-z^*-w^*}{2}, \\
z^{\downarrow} &= \frac{-x^*-y^*+z^*-w^*}{2}, \qquad \qquad
w^{\downarrow} = \frac{-x^*-y^*-z^*+w^*}{2},
\end{align*}
where $x^*, y^*, z^*, w^*$ are from Definition \ref{def:xyzws}.
Therefore,  $\sigma$ sends
\begin{align*}
x^{\downarrow} \mapsto - x^{\downarrow}, \qquad \quad
y^{\downarrow} \mapsto y^{\downarrow}, \qquad \quad
z^{\downarrow} \mapsto z^{\downarrow}, \qquad \quad
w^{\downarrow} \mapsto w^{\downarrow}.
\end{align*}
It would be interesting to explore how $\mathfrak{sl}_4(\mathbb C)$ acts on the basis  \eqref{eq:PbasisDown}.
\end{problem}

\begin{problem}\label{prob:prob2} \rm We refer to Problem \ref{prob:prob1}.
The vectors $x^{\downarrow}, y^{\downarrow}, z^{\downarrow}, w^{\downarrow}$ are common eigenvectors for the following three
elements of $\mathfrak{sl}_4(\mathbb C)$:
\begin{align*}
A^{\downarrow}_1 = \frac{\lbrack \lbrack A^*_3, A_1\rbrack, A^*_2 \rbrack}{4} 
= \begin{pmatrix} 0 &-1& 0 & 0 \\
                            -1 &0&0&0\\
                             0&0&0&1 \\
                             0&0&1&0
   \end{pmatrix}, \\       
   A^{\downarrow}_2 = \frac{\lbrack \lbrack A^*_1, A_2\rbrack, A^*_3 \rbrack}{4} 
= \begin{pmatrix} 0 &0& -1 & 0 \\
                            0 &0&0&1\\
                             -1&0&0&0 \\
                             0&1&0&0
   \end{pmatrix}, \\    
      A^{\downarrow}_3 = \frac{\lbrack \lbrack A^*_2, A_3\rbrack, A^*_1 \rbrack}{4} 
= \begin{pmatrix} 0 &0& 0 & -1 \\
                            0 &0&1&0\\
                             0&1&0&0 \\
                             -1&0&0&0
   \end{pmatrix}.
   \end{align*}
Specifically,
\begin{enumerate}
\item[\rm (i)]   $A^{\downarrow}_1$ sends
\begin{align*}
x^{\downarrow} \mapsto  x^{\downarrow}, \qquad \quad
y^{\downarrow} \mapsto y^{\downarrow}, \qquad \quad
z^{\downarrow} \mapsto -z^{\downarrow}, \qquad \quad
w^{\downarrow} \mapsto -w^{\downarrow};
\end{align*}
\item[\rm (ii)]  $A^{\downarrow}_2$ sends
\begin{align*}
x^{\downarrow} \mapsto  x^{\downarrow}, \qquad \quad
y^{\downarrow} \mapsto -y^{\downarrow}, \qquad \quad
z^{\downarrow} \mapsto z^{\downarrow}, \qquad \quad
w^{\downarrow} \mapsto -w^{\downarrow};
\end{align*}
\item[\rm (iii)] $A^{\downarrow}_3$ sends
\begin{align*}
x^{\downarrow} \mapsto  x^{\downarrow}, \qquad \quad
y^{\downarrow} \mapsto -y^{\downarrow}, \qquad \quad
z^{\downarrow} \mapsto -z^{\downarrow}, \qquad \quad
w^{\downarrow} \mapsto w^{\downarrow}.
\end{align*}
\end{enumerate}
The elements $A^{\downarrow}_1, A^{\downarrow}_2, A^{\downarrow}_3$ form a basis for a Cartan subalgebra $\mathbb H^{\downarrow}$ of $\mathfrak{sl}_4(\mathbb C)$.
It would be interesting to explore how $\mathbb H$, $\mathbb H^*$, $\mathbb H^{\downarrow}$ are related.
\end{problem}

\begin{problem}\label{prob:prob3} \rm In the present paper we treated the graph $H(N,2)$ is an $S_3$-symmetric way. As we mentioned in Section 17, 
$H(N,2)$ is a $Q$-polynomial distance-regular graph that has diameter $N$ and is  
a bipartite antipodal 2-cover.
Let $\Gamma$ denote any $Q$-polynomial distance-regular graph that has diameter $N$ and is a bipartite antipodal 2-cover. Such a graph is called 2-homogeneous; see \cite{2Hom, 2homT}. 
To avoid trivialities, let us assume that $\Gamma$ is not isomorphic to $H(N,2)$. The intersection numbers
of $\Gamma$ are determined by  $N$ and a certain scalar parameter $q$; see \cite[Theorem~35]{2Hom}.
We seek an $S_3$-symmetric treatment of $\Gamma$ that is analogous to the present paper. Such a treatment would amount to a $q$-analog of the treatment in the present paper.
\end{problem}
\section{Acknowledgement} 
\noindent The authors thank  Kazumasa Nomura for reading the manuscript carefully, and offering valuable comments.


\bigskip

\noindent William J. Martin \hfil\break
\noindent Department of Mathematical Sciences \hfil\break
\noindent Worcester Polytechnic Institute \hfil\break
\noindent Worcester, MA  01609 USA \hfil\break
\noindent email: {\tt martin@wpi.edu} \hfil\break


\noindent Paul Terwilliger \hfil\break
\noindent Department of Mathematics \hfil\break
\noindent University of Wisconsin \hfil\break
\noindent 480 Lincoln Drive \hfil\break
\noindent Madison, WI 53706-1388 USA \hfil\break
\noindent email: {\tt terwilli@math.wisc.edu }\hfil\break

\section{Statements and Declarations}

\noindent {\bf Funding}: The author declares that no funds, grants, or other support were received during the preparation of this manuscript.
\medskip

\noindent  {\bf Competing interests}:  The author  has no relevant financial or non-financial interests to disclose.
\medskip

\noindent {\bf Data availability}: All data generated or analyzed during this study are included in this published article.

\end{document}